\documentclass[a4paper,10pt]{article}

\usepackage{latexsym}
\usepackage{mathrsfs}
\usepackage{amsfonts,amsmath,amssymb}
\usepackage{indentfirst}
\usepackage{subeqnarray}
\usepackage{verbatim}
\usepackage[pdftex]{graphicx}
\usepackage{epstopdf}
\usepackage{stmaryrd}
\usepackage{multirow}
\usepackage{diagbox}
\usepackage{subcaption}
\usepackage{pdfpages}
\usepackage{algorithm}
\usepackage{algorithmic}
\usepackage[title]{appendix}
\usepackage{bm}

\newcommand{\ud}{\,\mathrm{d}}
\newcommand{\bx}{\bm{x}}
\newcommand{\bX}{\bm{X}}
\newcommand{\bY}{\bm{Y}}

\newcommand{\bzeta}{\bm{\zeta}}
\newcommand{\bpsi}{\bm{\psi}}
\newcommand{\bphi}{\bm{\phi}}

\newcommand{\ba}{\bm{a}}
\newcommand{\bA}{\bm{A}}
\newcommand{\bb}{\bm{b}}

\newcommand{\bF}{\bm{F}}

\newcommand{\bG}{\bm{G}}
\newcommand{\bh}{\bm{h}}
\newcommand{\bH}{\bm{H}}

\newcommand{\br}{\bm{r}}

\newcommand{\bv}{\bm{v}}

\newcommand{\bW}{\bm{W}}

\newcommand{\bp}{\bm{p}}

\newcommand{\cB}{\mathcal{B}}
\newcommand{\cU}{\mathcal{U}}

\newcommand{\cG}{\mathcal{G}}
\newcommand{\cL}{\mathcal{L}}
\newcommand{\cN}{\mathcal{N}}

\newcommand{\balpha}{\mbox{\boldmath{$\alpha$}}}

\newcommand{\Real}{\mathbb{R}}

\newtheorem{theorem}{Theorem}[section]

\newtheorem{assumption}[theorem]{Assumption}

\newtheorem{example}{Example}

\newtheorem{remark}[theorem]{Remark}
\numberwithin{equation}{section}

\DeclareMathAlphabet\mathbfcal{OMS}{cmsy}{b}{n}

\newenvironment{proof}[1][Proof]{\textbf{#1.} }
{\ \rule{0.75em}{0.75em}\smallskip}

\usepackage[pdftex,unicode,colorlinks,linkcolor=blue,citecolor=red]{hyperref}

\begin{document}

\begin{center}
\Large\bf \textbf{Adaptive-Growth Randomized Neural Networks for Level-Set Computation of Multivalued Nonlinear First-Order PDEs with Hyperbolic Characteristics}
\end{center}

\begin{center}
    {\large\sc Haoning Dang}\footnote{School of Mathematics and Statistics, Xi'an Jiaotong University, Xi'an, Shaanxi 710049, China. E-mail: {\tt haoningdang.xjtu@stu.xjtu.edu.cn}},\quad
    {\large\sc Shi Jin}\footnote{School of Mathematical Sciences \& Institute of Natural Sciences, Shanghai Jiao Tong University, Shanghai 200240, China. Email: {\tt shijin-m@sjtu.edu.cn}},\quad
    {\large\sc Fei Wang}\footnote{School of Mathematics and Statistics \& State Key Laboratory of Multiphase Flow in Power Engineering, Xi'an Jiaotong University, Xi'an, Shaanxi 710049, China. Email: {\tt feiwang.xjtu@xjtu.edu.cn}}

\end{center}

\medskip
\begin{quote}
{\bf Abstract.} This paper proposes an Adaptive-Growth Randomized Neural Network (AG-RaNN) method for computing multivalued solutions of nonlinear first-order PDEs with hyperbolic characteristics, including quasilinear hyperbolic balance laws and Hamilton--Jacobi equations. Such solutions arise in geometric optics, seismic waves, semiclassical limit of quantum dynamics and high frequency limit of linear waves, and differ markedly from the viscosity or entropic solutions. The main computational challenges lie in that the solutions are no longer functions, and become union of multiple branches, after the formation of singularities.   Level-set formulations offer a systematic alternative by embedding the nonlinear dynamics into linear transport equations posed in an augmented phase space, at the price of substantially increased dimensionality. To alleviate this computational burden, we combine AG-RaNN with an adaptive collocation strategy that concentrates samples in a tubular neighborhood of the zero level set, together with a layer-growth mechanism that progressively enriches the randomized feature space. Under standard regularity assumptions on the transport field and the characteristic flow, we establish a convergence result for the AG-RaNN approximation of the level-set equations. Numerical experiments demonstrate that the proposed method can efficiently recover multivalued structures and resolve nonsmooth features in high-dimensional settings.

\end{quote}

{\bf Keywords.} Randomized neural networks, adaptive growth, first-order PDEs, level-set methods, multivalued solutions
\medskip

{\bf Mathematics Subject Classification.} 68T07, 68W25, 65M22, 65N22

\section{Introduction}

Many nonlinear partial differential equations may develop singularities in finite time, even for smooth initial data. Classical examples include the Hamilton--Jacobi equations and hyperbolic conservation laws. In the former case, derivatives of the solution can become discontinuous, leading to the formation of caustics, whereas in the latter case the solution itself may become discontinuous, giving rise to shock waves. After singularities form, solutions must be interpreted in a weak or generalized sense; without additional admissibility conditions, uniqueness typically fails. Therefore, entropy conditions for conservation laws and viscosity-solution theory for the Hamilton--Jacobi equations are introduced to identify physically relevant solutions (\cite{lax1973, crandall1983}).

There are important physical applications in which viscosity or entropy solutions are not physically relevant, and the appropriate description instead involves multivalued solutions. Representative examples include geometric optics, the semiclassical limit of the Schr\"odinger equation, seismic wave propagation, and the high-frequency limit of linear wave equations (\cite{Brenier84, ER1996, FS02, JinLi03, SpMaMa}). In these settings, the underlying dynamics is governed by Hamiltonian systems, where energy conservation holds, and the solution structure is naturally described in phase space. Multivalued solutions arise from the crossing of characteristics and can be interpreted as classical solutions of the governing equations by method of characteristics, when viewed in an augmented phase-space formulation (\cite{CourantH, SpMaMa}).

The computation of multivalued solutions is challenging, since such solutions are no longer single-valued functions, and new variables must be introduced to represent the multiple phases emerging during time evolution. Moreover, the total number of phases is generally unknown a priori (\cite{Brenier84, ER1996, JinLi03}). Deep neural networks have been explored to address these difficulties, for instance, by learning closure relations for moment systems associated with multivalued solutions in the semiclassical limit of the Schr\"odinger equation (\cite{JLLZ}).

Alternatively, the level-set methods introduced in \cite{Jin2003Levelset, Liu2006Levelset} provide a systematic framework for computing multivalued solutions in phase space. By introducing a level-set function in an augmented phase space, the original nonlinear evolution can be reformulated as a linear transport equation, whose zero level set represents the multivalued solution. Although solving such linear equations is conceptually simpler than directly tracking a nonlinear, time-varying number of phases, the dimensionality of the level-set formulation can be significantly higher—often can be  twice that of the original problem—leading to severe curse-of-dimensionality issues unless sophisticated localization strategies are employed (\cite{Peng1999Locallevelse}).

Deep neural networks have emerged as powerful tools for solving high-dimensional PDEs. Representative approaches include the deep Ritz method (\cite{E2018DRM}), physics-informed neural networks (\cite{Raissi2019Physics}), and related variants (\cite{Sirignano2018DGM, Xu2020FiniteNeuron, Zang2020WAN, Liao2021DNM}). The approximation properties of shallow neural networks are well understood (\cite{Cybenko1989Approximation, Barron1993Universal, Barron1994Approximation}), and have been further investigated in various contexts in recent years (\cite{E2022Representation, He2023expressivity, Liao2023Spectral, Ming2023BarronClass, Jiao2024deep, Jiao2024drm}). These methods typically rely on nonlinear, nonconvex loss functions optimized by stochastic gradient-based algorithms, which may suffer from slow convergence or suboptimal local minima, especially in high-dimensional settings. As a result, optimization errors may noticeably deteriorate the accuracy of the computed solutions.

To alleviate optimization difficulties, Randomized Neural Networks (RaNNs) provide an attractive alternative by fixing most nonlinear parameters and only training linear coefficients. Extreme Learning Machines (ELMs), as a special case of RaNNs with a single hidden layer, have been successfully applied to PDE problems. For example, local ELM methods (\cite{Dong2021Local, Dong2023Method}) couple subdomain solutions through interface collocation points, while random feature methods (\cite{Chen2022RFM, Chen2024RFM}) employ partition-of-unity frameworks to ensure global continuity. RaNNs have also been combined with classical numerical methods, including Petrov–Galerkin formulations (\cite{Shang2023Randomized, Shang2023RNNPG}), discontinuous Galerkin methods (\cite{Sun2022LRNNDG, Sun2023DVWE, Sun2024KDV}), hybridized discontinuous Petrov–Galerkin methods (\cite{Dang2024HDPG}), and finite difference schemes (\cite{Li2025LRNN}). To further improve computational efficiency, preconditioning strategies for the resulting least-squares systems have been developed in \cite{Shang2024DD}.

In this work, we employ the Adaptive-Growth RaNN (AG-RaNN) proposed in \cite{Dang2024AGRNN} to solve for multivalued solutions of quasi-linear hyperbolic balance laws and Hamilton–Jacobi equations via level-set formulations. By reformulating the original problem as a linear level-set equation in an augmented space–time domain, the resulting numerical method becomes a fully space–time approach, without time marching. Consequently, the computation reduces to solving a single linear least-squares problem. Although AG-RaNN fixes most parameters during training, existing approximation theory shows that it retains strong expressive power. Moreover, the resulting least-squares problem is linear and convex, leading to robust and efficient numerical solutions. Building on this framework, we further incorporate adaptive collocation near the zero level set and a layer-growth mechanism to improve efficiency and accuracy in the high-dimensional level-set setting.

The remainder of this paper is organized as follows. Section \ref{sec:levelset} reviews first-order PDEs and level-set formulations. Section \ref{sec:AGRANN} introduces the AG-RaNN method, including the layer-growth strategy and adaptive collocation point selection. Section \ref{sec:analysis} presents the error analysis of the RaNN method, covering approximation, statistical, and optimization errors. Section \ref{sec:experiments} describes the numerical procedure for locating the zero level set and demonstrates the accuracy and efficiency of the proposed method through numerical experiments. Finally, Section \ref{sec:summary} summarizes the main results of this paper.

\section{First-Order PDEs and Level-Set Method} \label{sec:levelset}

In this section, we review the level-set formulation for a class of nonlinear first-order partial differential equations. We consider the following general first-order partial differential equation:
\begin{align} \label{eq:order1st}
    \partial_tu+G(t,\bx,u,\nabla_{\bx}u)=0,
\end{align}
where $u(t,\bx):\Real\times\Real^d\to\Real$ is a scalar-valued function and $G(t,\bx,z,\bp):\Real\times\Real^d\times\Real\times\Real^d\to\Real$ is a given scalar (nonlinear) function.
Following \cite{Jin2003Levelset,Liu2006Levelset}, we introduce a level-set function $\phi$ such that
\begin{align} \label{eq:levelfun1}
    \phi(t,\bx,z,\bp)=0\quad\text{at}\quad z=u(t,\bx)\text{ and }\bp=\nabla_{\bx}u(t,\bx).
\end{align}
The level-set function $\phi$ satisfies the following linear transport equation:
\begin{align} \label{eq:leveleq1}
    \partial_t\phi+\nabla_{\bp}G\cdot\nabla_{\bx}\phi+(\bp\cdot\nabla_{\bp}G-G)\,\partial_z\phi-(\nabla_{\bx}G+\bp\,\partial_zG)\cdot\nabla_{\bp}\phi=0.
\end{align}
This equation can be written in the unified operator form
\begin{align}\label{eq:ls-operator}
  \cG\phi(\bX):=\partial_t\phi(\bX)+\ba(\bX)\cdot\nabla_{\bY}\phi(\bX) = 0,\qquad \bX=(t,\bY),
\end{align}
where \(\bY=(\bx,z,\bp)\) collects the physical and auxiliary variables, and \(\ba(\bX)\) denotes the associated characteristic velocity field. For simplicity, we restrict our analysis to a bounded computational domain $D=(0,T)\times\Omega=(0,T)\times\Omega_{\bx}\times I_z\times\Omega_{\bp}$, and $\Gamma\subset\partial D$ is the inflow boundary. For the theoretical analysis, we assume that the computational domain $\Omega$ is chosen sufficiently large so that characteristics issued from the initial-time boundary remain inside $D=(0,T)\times\Omega$ for all $t\in[0,T]$. This assumption is reasonable due to the finite propagation speed of the linear transport equation  \eqref{eq:leveleq1}. Equivalently, no characteristic enters $D$ through the lateral boundary of $\Omega$ over the time interval of interest. Under this assumption, the inflow boundary reduces to the initial-time boundary, i.e., $\Gamma=\{0\}\times\Omega$. In the numerical experiments, we adopt a sufficiently large computational domain to make this assumption plausible in practice.

Since the level-set formulation embeds the solution manifold into an augmented phase space of dimension $2d+1$, one needs to determine the intersection of $d+1$ level-set functions in order to recover the unique solution $u(t,\bx)$ of the original equation. According to different initial conditions, we can obtain all $d+1$ level-set functions  $\phi_i(t,\bx,z,\bp)$ satisfying the same level-set equation \eqref{eq:leveleq1}. Two representative constructions are presented in Subsections \ref{Sec:Onlygrad} and \ref{Sec:scalarhyper}. Within the space-time framework, we treat the inflow boundary condition as a Dirichlet boundary condition on the space-time domain and write
\begin{align} \label{eq:ls-bc}
    \cB\phi=\phi|_{\Gamma}=g,
\end{align}
where $g$ denotes the prescribed boundary data. In all cases considered in this work, \(\cB\) is the identity operator; it is introduced here solely to provide a unified notation that facilitates the subsequent analysis.

\subsection{Hamilton--Jacobi Equations: Level-Set Formulation for Gradient Fields}
\label{Sec:Onlygrad}

A special case arises when $G=G(t,\bx,\bp)$ is independent of $z$, i.e., the first-order equation does not depend on the solution $u$ explicitly. In this setting, the level-set function \eqref{eq:levelfun1} is defined as 
\begin{align} \label{eq:levelfun2}
    \phi(t,\bx,\bp)=0\quad\text{at}\quad \bp=\nabla_{\bx}u(t,\bx),
\end{align}
while the level-set equation \eqref{eq:leveleq1} reduces to
\begin{align} \label{eq:leveleq2}
    \partial_t\phi+\nabla_{\bp}G\cdot\nabla_{\bx}\phi-\nabla_{\bx}G\cdot\nabla_{\bp}\phi=0.
\end{align}
In this case, the zero level set of $\phi(t,\bx,\bp)$ characterizes the gradient field $\nabla_{\bx}u(t,\bx)$, rather than the scalar solution $u$ itself.

A representative example of this case is the Hamilton--Jacobi equation
\begin{align} \label{eq:HJ1}
    \partial_t S + H(\bx,\nabla_{\bx} S) &= 0, \quad(t,\bx)\in(0,\infty)\times\Real^d,\\
    S(0,\bx) &= S_0(\bx),\quad \bx\in\Real^d,
\end{align}
where $H(\bx,\nabla_{\bx}S):\Real^{d}\times \Real^{d}\to\Real$ denotes the Hamiltonian. 

If only the gradient $\nabla_{\bx}S$ is of interest, one can introduce $d$ level-set functions
\begin{align*}
\phi_i(t,\bx,\bp)=p_i-\partial_{x_i}S(t,\bx),\qquad i=1,\cdots,d,
\end{align*}
so that $\phi_i=0$ encodes $p_i=\partial_{x_i}S$. Each $\phi_i$ satisfies
\begin{align}\label{eq:HJ5}
    \partial_t\phi_i + \nabla_{\bp} H\cdot\nabla_{\bx}\phi_i-\nabla_{\bx}H\cdot\nabla_{\bp}\phi_i=0,
    \qquad i=1,\cdots,d,
\end{align}
with the initial conditions
\begin{align}
    \phi_i(0,\bx,\bp)=p_i-\partial_{x_i}S_0(\bx),\qquad i=1,\cdots,d.
\end{align}

If the scalar function $S$ itself is also required, one introduces $d+1$ level-set functions
\begin{align}
    \phi(t,\bx,z,\bp)=0\quad\text{at}\quad z=S\text{ and }\bp=\nabla_{\bx}S,
\end{align}
which satisfy the level-set equations
\begin{align} \label{eq:HJ2}
    \partial_t\phi_i + \nabla_{\bp} H\cdot\nabla_{\bx}\phi_i+(\bp\cdot\nabla_{\bp}H-H)\partial_z\phi_i-\nabla_{\bx}H\cdot\nabla_{\bp}\phi_i=0,\quad i=0,1,\cdots,d.
\end{align}
For the initial conditions, a direct choice is
\begin{align} 
    \phi_0(0,\bx,z,\bp)=z-S_0(\bx),\quad\phi_i(0,\bx,z,\bp)=p_i-\partial_{x_i}S_0(\bx),\quad i=1,\cdots,d.\label{eq:HJ3}
\end{align} 
Here the augmented phase space introduces the auxiliary variables $(z,\bp)$ (of dimension $d+1$), and the solution manifold is characterized as the intersection of the $d+1$ zero level sets associated with \eqref{eq:HJ3}.

In the semiclassical (WKB) limit of the Schr\"odinger equation, the phase function $u$ satisfies the eikonal equation, which is a Hamilton-Jacobi equation with Hamiltonian \(H(\bx,\bp)=\tfrac12|\bp|^2+V(\bx)\). In this setting, \(\nabla_{\bx}S\) represents the associated momentum field.

\subsection{Scalar Hyperbolic Balance Laws: Level-Set Formulation for the Solution}
\label{Sec:scalarhyper}

In the case where $G$ depends linearly on $\bp$, the first-order equation reduces to a scalar hyperbolic balance law
\begin{align} \label{eq:order1st3}
\partial_tu+G(t,\bx,u,\nabla_{\bx}u)=\partial_tu+\bA(t,\bx,u)\cdot\nabla_{\bx}u+B(t,\bx,u)=0.
\end{align}
In this setting, it suffices to introduce a single level-set function
\begin{align} \label{eq:levelfun3}
    \phi(t,\bx,z)=0\quad\text{at}\quad z=u(t,\bx),
\end{align}
which satisfies the reduced level-set equation
\begin{align} \label{eq:leveleq3}
    \partial_t\phi+\nabla_{\bp}G\cdot\nabla_{\bx}\phi+(\bp\cdot\nabla_{\bp}G-G)\partial_z\phi=0.
\end{align}
With \eqref{eq:order1st3}, the above equation is equivalent to
\begin{align} \label{eq:leveleq3e}
    \partial_t\phi+\bA(t,\bx,z)\cdot\nabla_{\bx}\phi-B(t,\bx,z)\partial_z\phi=0.
\end{align}

A special case is the following quasi-linear hyperbolic balance law. The solution $u(t,\bx)$ is a scalar function satisfying
\begin{align}
    \partial_t u + \bF(u)\cdot\nabla_{\bx}u + q(\bx,u) &= 0, \quad(t,\bx)\in(0,\infty)\times\Real^d,\label{eq:Hyper1}\\
    u(0,\bx) &= u_0(\bx),\quad \bx\in\Real^d,\label{eq:Hyper2}
\end{align}
where $\bF(u):\Real\to\Real^d$ is a vector and $q(\bx,u):\Real^d\times\Real\to\Real$ is a scalar source term. Equation \eqref{eq:Hyper1} is linearly dependent on the gradient of $u$, according to \eqref{eq:leveleq3e}, the level-set equation is
\begin{align}\label{eq:Hyper3}
    \partial_t\phi + \bF(z)\cdot\nabla_{\bx}\phi-q(\bx,z)\partial_z\phi=0.
\end{align}
For the initial condition, a direct choice is
\begin{align} \label{eq:Hyper4}
    \phi(0,\bx,z)=z-u_0(\bx).
\end{align}
However, for discontinuous $u_0$, this choice may be non-smooth and it can be replaced by a signed distance function to the interface $z=u_0(\bx)$. Since only one new dimension is introduced in \eqref{eq:Hyper3}, we only calculate the solution with the initial condition \eqref{eq:Hyper4}.

\section{Adaptive-Growth Randomized Neural Networks} \label{sec:AGRANN}

We use adaptive-growth randomized neural networks to solve nonlinear first-order PDEs via the level-set formulation, including Hamilton-Jacobi equations and quasi-linear hyperbolic balance laws, with the goal of computing multivalued solutions. To improve the expressive power of randomized features, we adopt the layer growth strategy proposed in \cite{Dang2024AGRNN}. For completeness, we briefly summarize the network representation and the layer-growth procedure, and then describe the adaptive selection of collocation points near the zero level set.

\subsection{AG-RaNN with the Layer Growth Strategy}

We briefly review the network structure introduced in \cite{Dang2024AGRNN}. Consider an AG-RaNN with $L$ hidden layers, hence $L+2$ layers in total when counting the input and output layers. Let $m_l$ denote the number of neurons in layer $l$ for $l=0,1,\cdots,L+1$, where $m_0$ is the input dimension and $m_{L+1}$ is the output dimension. The activation function in the $l$-th hidden layer is denoted by $\rho_l$ for $l=1,\cdots,L$.

Since the nonlinear parameters in randomized neural networks are fixed after initialization, the resulting approximation can be written as a linear combination of basis functions generated by the hidden layers.  Specifically, we define the hidden-layer basis vector $\bpsi_l(\bx)$ recursively by
\begin{align*}
    &\bpsi_1(\bx)=\rho_1(\bW_1\bx+\bb_1),\\
    &\bpsi_l(\bx)=[\bpsi_{l-1}(\bx);~\rho_l(\bW_l\bpsi_{l-1}(\bx)+\bb_l)\odot \bG_l(\bx)],\quad l=2,\cdots,L,
\end{align*}
where $\bW_1\in\Real^{m_1\times m_0}$, $\bW_l\in\Real^{m_l\times\sum_{i=1}^{l-1}m_i}~(l=2,\cdots,L)$ and $\bb_l\in\Real^{m_l\times1}~(l=1,\cdots,L)$ are weights and biases. The vector $\bpsi_l(\bx)$ contains $\sum_{i=1}^l m_i$ basis functions. The first-layer parameters $(\bW_1,\bb_1)$ can be determined by the frequency-based initialization and neuron-growth strategy in \cite{Dang2024AGRNN}, or sampled from a uniform distribution $\cU(-\br_1,\br_1)$. In what follows, we focus on constructing $\bW_l$ and $\bb_l$ for $l=2,\cdots,L$.

The layer-growth procedure for AG-RaNN is performed in an iterative manner. Suppose that an RaNN with $l-1$ hidden layers has been constructed, and let $\bpsi_{l-1}(\bx)$ denote the corresponding feature (basis) vector generated up to the $l-1$-st hidden layer. The goal of the next growth step is to construct the parameters $\bW_l$, $\bb_l$, as well as the localization function $\bG_l(\bx)$ for the $l$-th hidden layer.

Based on the current feature representation $\bpsi_{l-1}(\bx)$, we obtain the intermediate RaNN approximation $\phi_\rho^{l-1}(\bx) = \balpha_{l-1} \bpsi_{l-1}(\bx)$, where $\balpha_{l-1} \in \Real^{1 \times \sum_{i=1}^{l-1} m_i}$ denotes the vector of linear coefficients. Using this approximation, we select a set of error-indicator points $\bX_{\mathrm{err}}^l = (\bx_1^l, \cdots, \bx_{m_l}^l)$, which consists of the $m_l$ points associated with the largest values of a prescribed error estimator. The parameters of the new hidden layer are then defined by
\begin{align*}
    \bW_l=\bH_l\balpha_{l-1},\quad\bb_l=-\bH_l\odot[\phi_\rho^{l-1}(\bX_\text{err}^l)]^\top,
\end{align*}
together with the localization function
\begin{align*}
    \bG_l(\bx)=(G_l^1(\bx),\cdots,G_l^{m_l}(\bx))^\top,\quad G_l^j(\bx)=\exp\left(-\frac{\left\|\eta^l_2\bh^l_j\odot(\bx-\bx^l_j)^\top\right\|_2^2}{2}\right).
\end{align*}
Here the scaling vector $\bH_l$ is given by
\begin{align*}
    \bH_l=\sqrt{(\bH_l^0\odot\bH_l^0)\mathbf{1}_{m_0\times1}},\quad\bH_l^0=\eta_1^l\nabla\phi_\rho^{l-1}(\bX_\text{err}^{l\top})\text{ or }\bH_l^0=(\bh_1^{l\top},\cdots,\bh_{m_l}^{l\top})^\top\in\Real^{m_l\times m_0},\quad \bh_j^l\sim\cU(-\br_l,\br_l).
\end{align*}

\begin{remark}
Several parameters appear in the above construction, including $\br_l$, $\eta_1^l$, and $\eta_2^l$ ($l=2,\cdots,L$), which must be specified in advance. In all numerical experiments reported in this paper, we sample $\bh_j^l \sim \mathcal{U}(-\br_l, \br_l)$ and set $\eta_2^l = 15$, so that $\br_l$ is the only parameter that requires tuning. 
\end{remark}

\subsection{Adaptive Collocation Points} \label{Sec:AdaptPoint}

Before layer growth, we first compute a coarse RaNN solution $\phi_\rho^{1}(t,\bY)$, where $t$ denotes time and $\bY$ collects the remaining variables in the level-set formulation. Since our goal is to accurately capture the zero level set of $\phi(t,\bY)$, it is sufficient (and more efficient) to sample collocation points in a neighborhood of $\{\phi=0\}$. Inspired by the local level-set method (\cite{Peng1999Locallevelse}), we therefore adopt an adaptive collocation strategy that updates sampling points based on the coarse solution $\phi_\rho^{1}$.

The level-set equation is posed on an unbounded domain, and uniform sampling over a large truncated region quickly becomes inefficient. To initialize the procedure, we sample the variables $\bY$ from a normal distribution and the time variable $t$ uniformly over a bounded interval, thereby defining an initial computational domain (illustrated by the red rectangle in Figure \ref{fig:CollocationPoints}). Several parameters are involved in this initialization, including the mean, variance, and the number of samples. These parameters are used only for initialization and are not heavily tuned; they are chosen to ensure sufficient coverage of the computational domain. With this initial sampling, a coarse RaNN solution $\phi_\rho^{1}$ is obtained, and an approximation of the zero level set can be extracted (shown as the green curve in the middle subfigure of Figure \ref{fig:CollocationPoints}).

\begin{figure}[!htbp]
    \centering
    \includegraphics[width=0.9\textwidth]{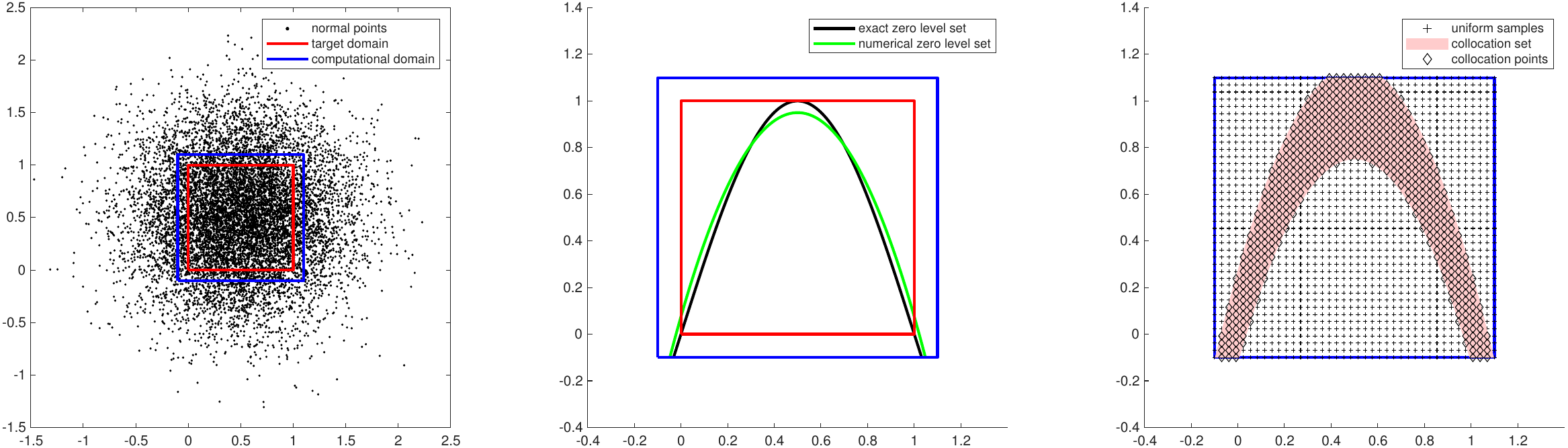}
    \caption{Schematic diagram of adaptive collocation points.}
    \label{fig:CollocationPoints}
\end{figure}

To update the interior collocation points, the influence of regions outside the initial target domain cannot be neglected. We therefore introduce a computational domain $D=(0,T)\times\Omega$ (depicted by the blue rectangle in Figure \ref{fig:CollocationPoints}) and uniformly sample points $\{(t_i,\bY_i)\}_{i=1}^{N}$ from this domain. Based on the coarse solution $\phi_\rho^{1}$, we define the updated set of interior collocation points
\begin{align*}
\Lambda_A=\left\{(t_i,\bY_i):\;|\phi_\rho^{1}(t_i,\bY_i)|\le\varepsilon_A\right\},
\end{align*}
which corresponds to a tubular neighborhood of the zero level set (the red shaded region in the right subfigure of Figure \ref{fig:CollocationPoints}), where $\varepsilon_A>0$ is a user-specified threshold. This procedure relies on the coarse approximation being sufficiently accurate to localize the neighborhood of the zero level set. Boundary collocation points are treated in an analogous manner. Using the updated collocation sets, we recompute the RaNN approximation, which is again denoted by $\phi_\rho^{1}$ for simplicity. 

\begin{remark}
To ensure numerical accuracy, the parameter $\varepsilon_A$ should be chosen large enough so that $\Lambda_A$ contains the zero level set of the exact solution. An excessively large $\varepsilon_A$, however, leads to an increased number of collocation points and reduced computational efficiency. Alternatively, a smaller $\varepsilon_A$ can be used in combination with multiple update cycles, at the cost of additional computational effort. In all numerical experiments reported in this paper, we select a sufficiently large $\varepsilon_A$ and update the collocation points only once.
\end{remark}

\subsection{Error-Indicator Points}

To construct the error-indicator points $\bX_{\mathrm{err}}^l$, we first need an approximation of the zero level set in high dimensions; a practical procedure is described in Subsection \ref{Sec:zerolevelset}. Once a set of points on (or near) the zero level set is available, we select $\bX_{\mathrm{err}}^l$ using an error estimator. In our experiments, we use the residual $\big|\cG\phi_\rho^l(\bX)\big|$ as the error indicator for $\bX\in\Lambda_A$.

\section{Numerical Analysis} \label{sec:analysis}

Following the operator form \eqref{eq:ls-operator}, let $D\subset\Real^{m}$ be a bounded Lipschitz domain in the extended variable $\bX$, and $\Gamma\subset\partial D$ be the inflow boundary. We consider the Hilbert spaces
\begin{align*}
    V_1 = L^2(D),\qquad V_2 = H^1(D),\qquad Y = L^2(D),\qquad Z = L^2(\Gamma).
\end{align*}
Here we use separate $V_1$ and $V_2$, rather than the same $V$ as in \cite{Dang2024AGRNN}.

\subsection{Graph Norm Equivalence}

To derive an error estimate, we establish a graph norm equivalence of the same type as Assumption 4.1 in \cite{Dang2024AGRNN}. The following structural assumption encodes the regularity of the vector field and the well-posedness of the characteristic flow.

\begin{assumption}[Regular flow of characteristics]\label{ass:char}
We assume that
\begin{enumerate}
  \item The transport field $\ba\in C^1(\overline D;\Real^{m-1})$ is bounded on $\overline D$: there exists $C_a>0$ such that
        \begin{align*}
          |\ba(\bX)| \le C_a, \qquad \forall\,\bX\in\overline D.
        \end{align*}
  \item For every $\bzeta\in\Gamma$ the characteristic ODE
        \begin{align}\label{eq:char-ode}
          \frac{\mathrm{d}}{\mathrm{d}s}\bX(s) = (1,\ba(\bX(s))),\qquad \bX(0)=\bzeta,
        \end{align}
        has a unique solution on $s\in[0,T]$ that remains inside $D$, and the associated flow map
        \begin{align*}
          \Psi:(0,T)\times\Gamma\to D,\qquad\Psi(s,\bzeta)=\bX(s;\bzeta)
        \end{align*}
        is a $C^1$-diffeomorphism. Its Jacobian $J(s,\bzeta):=\bigl|\det D\Psi(s,\bzeta)\bigr|$ satisfies uniform bounds
        \begin{align*}
          0<c_0\le J(s,\bzeta)\le C_0<\infty,\qquad\forall\,(s,\bzeta)\in(0,T)\times\Gamma.
        \end{align*}
\end{enumerate}
\end{assumption}

\begin{remark}
    Assumption \ref{ass:char} is a standard requirement in the theory of first-order linear transport equations, see classical results in \cite{Evans2022PDE}. For the level-set equations \eqref{eq:HJ5}, \eqref{eq:HJ2}, and \eqref{eq:Hyper3}, it is satisfied whenever $\bF$, $q$ and $H$ (together with their first derivatives) are bounded and Lipschitz on the bounded computational domain in the $(\bx,z,\bp)$ variables.
\end{remark}

We now show that under Assumption \ref{ass:char}, operator $\cG$ admits the following graph norm equivalence.

\begin{theorem}[Graph norm equivalence for the level-set equations]
\label{thm:graph-levelset}
Under Assumption \ref{ass:char}, consider the operator $\cG:V_2\to Y$ given by \eqref{eq:ls-operator} and the boundary operator $\cB:V_2\to Z$ defined in \eqref{eq:ls-bc}. Then there exist constants $C_L,C_U>0$, depending only on $D$, $C_a$, $c_0$, and $C_0$, such that for all $\phi\in H^1(D)$
\begin{align}\label{eq:graph-levelset}
  C_L\|\phi\|_{0,D}^2\;\le\;\|\cG\phi\|_{0,D}^2 + \|\cB\phi\|_{0,\Gamma}^2\;\le\;C_U\|\phi\|_{1,D}^2.
\end{align}
\end{theorem}

\begin{proof}
We first provide the proof for the case where $\phi \in C^1(\overline{D})$, dividing it into two steps.

\noindent\emph{Step 1: Upper bound.}
Since $\ba$ is bounded on $\overline D$, there exists $C_a>0$ such that $|\ba(\bX)|\le C_a$ for all $\bX=(t,\bY)\in D$. For $\phi\in C^1(\overline D)$ we have
\begin{align*}
  |\cG\phi(\bX)| = \bigl|\partial_t\phi(\bX) + \ba(\bX)\cdot\nabla_{\bY}\phi(\bX)\bigr| \le |\partial_t\phi(\bX)| + C_a|\nabla_{\bY}\phi(\bX)|.
\end{align*}
Squaring and integrating over $D$ yields
\begin{align*}
  \|\cG\phi\|_{0,D}^2\le 2\|\partial_t\phi\|_{0,D}^2+2C_a^2\|\nabla_{\bY}\phi\|_{0,D}^2\le C\Bigl(\|\partial_t\phi\|_{0,D}^2+\|\nabla_{\bY}\phi\|_{0,D}^2\Bigr)\le C\,\|\phi\|_{1,D}^2,
\end{align*}
for some constant $C>0$ depending only on $C_a$ and $D$. Moreover, by the trace theorem on the Lipschitz domain $D$,
\begin{align*}
  \|\cB\phi\|_{0,\Gamma}^2=\|\phi\|_{0,\Gamma}^2\le C_{\mathrm{tr}}^2 \|\phi\|_{1,D}^2,
\end{align*}
where $C_{\mathrm{tr}}$ is the trace constant. Combining the two estimates, we obtain
\begin{align*}
  \|\cG\phi\|_{0,D}^2 + \|\cB\phi\|_{0,\Gamma}^2
  \le (C+C_{\mathrm{tr}}^2)\,\|\phi\|_{1,D}^2
  =: C_U\,\|\phi\|_{1,D}^2,
\end{align*}
which proves the upper bound in \eqref{eq:graph-levelset}.

\noindent\emph{Step 2: Lower bound via characteristics.}
Let $\phi\in C^1(\overline D)$ and denote
\begin{align*}
  F(\bX) := \cG\phi(\bX)=\partial_t\phi(\bX) + \ba(\bX)\cdot\nabla_{\bY}\phi(\bX).
\end{align*}
For each $\bzeta\in\Gamma$ consider the characteristic curve $\bX(\tau;\bzeta)$ solving \eqref{eq:char-ode}. Along this curve we have
\begin{align}\label{eq:char-phi-deriv}
  \frac{\mathrm{d}}{\mathrm{d}\tau}\phi\bigl(\bX(\tau;\bzeta)\bigr)=\partial_\tau\phi\bigl(\bX(\tau;\bzeta)\bigr)+\ba\bigl(\bX(\tau;\bzeta)\bigr)\cdot\nabla_{\bY}\phi\bigl(\bX(\tau;\bzeta)\bigr)=F\bigl(\bX(\tau;\bzeta)\bigr).
\end{align}
Integrating \eqref{eq:char-phi-deriv} from $0$ to $s\in[0,T]$ gives
\begin{align}\label{eq:char-phi-int}
  \phi\bigl(\bX(s;\bzeta)\bigr)=\phi(\bzeta)+\int_0^s F\bigl(\bX(\tau;\bzeta)\bigr)\,\ud\tau .
\end{align}
Using $(a+b)^2\le 2a^2+2b^2$ and the Cauchy--Schwarz inequality, we find
{\small\begin{align*}
  \bigl|\phi\bigl(\bX(s;\bzeta)\bigr)\bigr|^2
  &\le 2|\phi(\bzeta)|^2+2\left(\int_0^s |F(\bX(\tau;\bzeta))|\,\ud\tau\right)^2\le 2|\phi(\bzeta)|^2+2s\int_0^s |F(\bX(\tau;\bzeta))|^2\,\ud\tau\\
  &\le 2|\phi(\bzeta)|^2+2T\int_0^T |F(\bX(\tau;\bzeta))|^2\,\ud\tau.
\end{align*}}
Integrating this inequality in $s$ over $[0,T]$ yields
\begin{align}\label{eq:phi-char-L2}
  \int_0^T\bigl|\phi\bigl(\bX(s;\bzeta)\bigr)\bigr|^2\ud s\le2T|\phi(\bzeta)|^2+2T^2\int_0^T |F(\bX(\tau;\bzeta))|^2\;\ud\tau .
\end{align}

We now change variables from $\bX\in D$ to $(s,\bzeta)\in(0,T)\times\Gamma$ via the flow map $\Psi(s,\bzeta)=\bX(s;\bzeta)$. By Assumption \ref{ass:char},
\begin{align*}
  \ud\bX = J(s,\bzeta)\,\ud s\,\ud\sigma(\bzeta),
\end{align*}
where $\ud\sigma$ denotes the surface measure on $\Gamma$. Using the upper bound $J(\tau,\bzeta)\le C_0$ and \eqref{eq:phi-char-L2}, we obtain
\begin{align*}
  \|\phi\|_{0,D}^2&=\int_D |\phi(\bX)|^2\;\ud\bX=\int_{\Gamma}\int_0^T\bigl|\phi\bigl(\bX(\tau;\bzeta)\bigr)\bigr|^2J(\tau,\bzeta)\,\ud \tau\,\ud\sigma(\bzeta)\\
  &\le C_0\int_{\Gamma}\int_0^T\bigl|\phi\bigl(\bX(\tau;\bzeta)\bigr)\bigr|^2\,\ud \tau\,\ud\sigma(\bzeta)\\
  &\le 2TC_0\int_{\Gamma}|\phi(\bzeta)|^2\;\ud\sigma(\bzeta)+2T^2C_0\int_{\Gamma}\int_0^T|F(\bX(\tau;\bzeta))|^2\;\ud\tau\,\ud\sigma(\bzeta).
\end{align*}
For the second term we use the lower bound $J(\tau,\bzeta)\ge c_0$ to change variables back to $\bX\in D$:
\begin{align*}
  \int_{\Gamma}\int_0^T|F(\bX(\tau;\bzeta))|^2\;\ud\tau\,\ud\sigma(\bzeta)&=\int_D |F(\bX)|^2 \frac{1}{J(\tau(\bX),\bzeta(\bX))}\;\ud\bX\\
  &\le \frac{1}{c_0}\int_D |F(\bX)|^2\;\ud\bX=\frac{1}{c_0}\|\cG\phi\|_{0,D}^2,
\end{align*}
where $(\tau(\bX),\bzeta(\bX))$ is the unique preimage of $\bX$ under $\Psi$. As a result,
\begin{align*}
  \|\phi\|_{0,D}^2\le 2TC_0\|\phi\|_{0,\Gamma}^2+\frac{2T^2C_0}{c_0}\|\cG\phi\|_{0,D}^2.
\end{align*}
Rearranging gives
\begin{align*}
  \|\cG\phi\|_{0,D}^2 + \|\phi\|_{0,\Gamma}^2\ge C_L\,\|\phi\|_{0,D}^2\text{ with }C_L:=\min\left\{\frac{1}{2TC_0},\;\frac{c_0}{2T^2C_0}\right\}>0.
\end{align*}
This is exactly the lower bound in \eqref{eq:graph-levelset} for $\phi\in C^1(\overline D)$.

Since $D$ is Lipschitz, $C^1(\overline D)$ is dense in $H^1(D)$. Hence, for any $\phi\in H^1(D)$, there exists a sequence $\{\phi_n\}\subset C^1(\overline D)$ such that
\begin{align*}
\|\phi_n-\phi\|_{1,D}\to 0 \quad \text{as } n\to\infty.
\end{align*}
Because $\ba\in C^1(\overline D)$ is bounded, the operator $\cG:H^1(D)\to L^2(D)$ is bounded and linear, and the trace operator $\cB:H^1(D)\to L^2(\Gamma)$ is bounded, hence
\begin{align*}
\|\cG(\phi_n-\phi)\|_{0,D}\to 0 ,\quad\quad\|\cB(\phi_n-\phi)\|_{0,\Gamma}\to 0.
\end{align*}
Passing to the limit in the inequality for $\phi_n$ yields \eqref{eq:graph-levelset} for all $\phi\in H^1(D)$.
\end{proof}

With the graph norm equivalence at hand, we can invoke the Hilbert-space error-decomposition framework developed in \cite{Dang2024AGRNN} for the present level-set operator $\cG$ and boundary operator $\cB$. 
In particular, we follow the notation and constants in \cite{Dang2024AGRNN} and only highlight the modifications specific to the level-set setting; detailed proofs of the approximation and statistical components are therefore omitted and can be found in \cite{Dang2024AGRNN}.

For clarity, we first illustrate the analysis for a single-hidden-layer RaNN:
\begin{align*}
    \cN_\rho(D)=\left\{\phi(\bx)=\balpha\rho(\bW\bx+\bb):\bx\in D,\,\|\balpha\|_1\le C_N\right\}.
\end{align*}
Let $\phi^*$ be the exact solution of the PDEs \eqref{eq:ls-operator} -- \eqref{eq:ls-bc},  we introduce
\begin{align*}
    \phi_a=\mathop{\arg\min}\limits_{\phi\in\cN_\rho(D)}\|\phi-\phi^*\|_{1,D},
\end{align*}
which implies that $\phi_a$ is the best approximation of $\phi^*$ in the RaNN space $\cN_\rho(D)$. Our goal is to compute the approximate solution $\phi_\rho$ by minimizing the loss function
\begin{align}
    \cL_\eta(\phi)=\|\cG\phi\|_{0,D}^2+\eta\|\cB\phi-g\|_{0,\Gamma}^2,
\end{align}
where $\eta$ is a boundary penalty parameter. In our experiments we take $\eta=15$ and omit $\eta$ in what follows. With proper collocation points $\{\bX^I_i\}_{i=1}^{N^I}$, $\{\bX^B_i\}_{i=1}^{N^B}$ and basis functions $\bpsi_1(\bx)$, that is
\begin{align}\label{eq:emp-loss-L2}
    \widetilde{\cL}(\phi)=\frac{|D|}{N^I}\sum_{i=1}^{N^I}|\cG\phi(\bX^I_i)|^2+\frac{|\Gamma|}{N^B}\sum_{i=1}^{N^B}|\cB\phi(\bX^B_i)-g(\bX^B_i)|^2.
\end{align}
Since $\phi(\bx)=\balpha\bpsi_1(\bx)$ is the linear combination of basis functions, we only use least-squares method to calculate $\balpha$. Let $\bar{\phi}_\rho$ be the minimizer of the empirical loss function $\widetilde{\cL}(\phi)$ over $\cN_\rho(D)$:
\begin{align*}
    \bar{\phi}_\rho=\mathop{\arg\min}\limits_{\phi\in\cN_\rho(D)}\widetilde{\cL}(\phi),
\end{align*}
and $\phi_\rho$ be the solution obtained by least-squares solver.

\begin{assumption}[Activation admissibility]\label{As:act}
The activation function $\rho$ satisfies:
\begin{itemize}
    \item[(A1)]$\rho\in C^1(\Real)\cap W^{1,\infty}(\Real)$.
    \item[(A2)] $\widehat{\rho'}(1)\neq 0$, $\rho'\in L^1(\Real)$ and $\rho''\in L^1(\Real)$.
\end{itemize}
\end{assumption}

\begin{remark}
Activation functions satisfying Assumption \ref{As:act} include $\tanh$, the logistic function, and other smooth sigmoid-type functions. In contrast, $\mathrm{ReLU}$ and $\mathrm{ReLU}^k$ do not belong to this class due to their lack of sufficient smoothness. Nevertheless, ReLU-type activations are also known to enjoy universal approximation properties; see, e.g., \cite{Liu2025ReLU} for the proof.
\end{remark}

By leveraging the graph norm equivalence and following a similar approach to that in \cite{Dang2024AGRNN}, we get
\begin{align*}
    C_L\|\phi^*-\phi_\rho\|_{0,D}^2\le\cL(\phi_\rho)
    \le C_U\|\phi^*-\phi_a\|_{1,D}^2+2\sup\limits_{\phi\in\cN_\rho(D)}|\cL(\phi)-\widetilde{\cL}(\phi)| +(\widetilde{\cL}(\phi_\rho)-\widetilde{\cL}(\phi_a)).
\end{align*}
The estimates for the first two terms have been established in \cite{Dang2024AGRNN}, the key results are summarized as follows.

\begin{theorem}[Approximation Error]
Suppose that Assumption \ref{ass:char} holds, that $\phi^*\in H^{1+\varepsilon_C}(D)$ for some $\varepsilon_C>0$, and that the activation function $\rho$ satisfies Assumption \ref{As:act}. Assume further that the inner parameters $(\bW,\bb)$ of the RaNN are sampled from a fixed bounded box $\Theta\subset\mathbb R^{m+1}$.

Then, for every tolerance $\varepsilon>0$ and confidence level $\delta\in(0,1)$, there exist a network width $M$ such that the corresponding best approximation $\phi_a\in\mathcal N_\rho(D)$ satisfies
\begin{align*}
    \|\phi^*-\phi_a\|_{1,D} \;\le\; \varepsilon
\end{align*}
with probability at least $1-\delta$ with respect to the randomness of the inner parameters.
\end{theorem}

\begin{theorem}[Statistical Component Bound] \label{Thm:statisticerr}
Under the Assumption \ref{ass:char}, with probability at least $1-\delta_s$ for the random choice of collocation points, the statistical component has bound
\begin{align*}
    \sup_{\phi\in\cN_\rho(D)}\left|\cL(\phi)-\widetilde{\cL}(\phi)\right|\,\lesssim\,C_N^2C_R^2M\left(\sqrt{\frac{\log(6M/\delta_s)}{N^I}}+\sqrt{\frac{\log(6M/\delta_s)}{N^B}}\right),
\end{align*}
where $\widetilde{\cL}(\phi)$ is the empirical $L^2$ loss defined in \eqref{eq:emp-loss-L2}.
\end{theorem}

Then, the subsequent analysis is primarily concerned with the last optimization term.

\subsection{Optimization Term}
Recall that $\bar{\phi}_\rho$ is the minimizer of the empirical loss function $\widetilde{\cL}(\phi)$ over $\cN_\rho(D)$, it follows that for any $\phi\in\cN_\rho(D)$, $\widetilde{\cL}(\bar{\phi}_\rho)\le\widetilde{\cL}(\phi)$. With the fact $\phi_a\in\cN_\rho(D)$, the difference $\widetilde{\cL}(\bar{\phi}_\rho)-\widetilde{\cL}(\phi_a)$ is non-positive. The optimization error can be decomposed as 
\begin{align*}
    &\widetilde{\cL}(\phi_\rho)-\widetilde{\cL}(\phi_a)=\widetilde{\cL}(\phi_\rho)-\widetilde{\cL}(\bar{\phi}_\rho)+\widetilde{\cL}(\bar{\phi}_\rho)-\widetilde{\cL}(\phi_a)\le\widetilde{\cL}(\phi_\rho)-\widetilde{\cL}(\bar{\phi}_\rho)\\
    =&\frac{|D|}{N^I}\sum_{i=1}^{N^I}\left(|\cG\phi_\rho(\bX^I_i)|^2-|\cG\bar{\phi}_\rho(\bX^I_i)|^2\right)+\frac{|\Gamma|}{N^B}\sum_{i=1}^{N^B}\left(|\cB\phi_\rho(\bX^B_i)-g(\bX^B_i)|^2-|\cB\bar{\phi}_\rho(\bX^B_i)-g(\bX^B_i)|^2\right)\\
    \le&\frac{|D|}{N^I}\left(\sum_{i=1}^{N^I}|\cG\phi_\rho(\bX^I_i)-\cG\bar{\phi}_\rho(\bX^I_i)|^2\right)^{1/2}\left(\sum_{i=1}^{N^I}|\cG\phi_\rho(\bX^I_i)+\cG\bar{\phi}_\rho(\bX^I_i)|^2\right)^{1/2}\\
    &+\frac{|\Gamma|}{N^B}\left(\sum_{i=1}^{N^B}|\cB\phi_\rho(\bX^B_i)-\cB\bar{\phi}_\rho(\bX^B_i)|^2\right)^{1/2}\left(\sum_{i=1}^{N^B}|\cB\phi_\rho(\bX^B_i)+\cB\bar{\phi}_\rho(\bX^B_i)-2g(\bX^B_i)|^2\right)^{1/2}.
\end{align*}

\begin{assumption}[Discrete operator boundedness]
For each finite-dimensional RaNN space $\cN_\rho(D)$ and for each choice of interior and boundary collocation points $\{\bX^I_i\}_{i=1}^{N^I}\subset D$, $\{\bX^B_i\}_{i=1}^{N^B}\subset\Gamma$, there exist constants $C_2>0$ and $C_3>0$, independent of $\phi\in\cN_\rho(D)$, such that
\begin{align*}
    \sum_{i=1}^{N^I}|\cG\phi(\bX^I_i)|^2\le C_2\sum_{i=1}^{N^I}|\phi(\bX^I_i)|^2,\quad\sum_{i=1}^{N^B}|\cB\phi(\bX^B_i)|^2\le C_3\sum_{i=1}^{N^B}|\phi(\bX^B_i)|^2.
\end{align*}
\end{assumption}
In finite dimensional RaNN space, this assumption is reasonable provided that the sampling points are sufficiently dense and well-distributed.

Since $\rho$ is continuous and $D$ is bounded, $\phi\in\cN_\rho(D)$ is always bounded, which means that there are constants $C_g$, $C_I$ and $C_B$ such that
\begin{align*}
    \left(\frac{1}{N^B}\sum_{i=1}^{N^B}|g(\bX^B_i)|^2\right)^{1/2}\le C_g,\quad\left(\frac{1}{N^I}\sum_{i=1}^{N^I}|\phi(\bX^I_i)|^2\right)^{1/2}\le C_I,\quad\left(\frac{1}{N^B}\sum_{i=1}^{N^B}|\phi(\bX^B_i)|^2\right)^{1/2}\le C_B,
\end{align*}
where $C_g, C_I, C_B$ depend only on the uniform $L^\infty$ bounds of $g$ and of the RaNN basis functions (through the constraint $\|\balpha\|_1\le C_N$), but are independent of the particular choice of $\phi\in\cN_\rho(D)$.

Then, we have
\begin{align*}
    \resizebox{\linewidth}{!}{
    $\widetilde{\cL}(\phi_\rho)-\widetilde{\cL}(\phi_a)\le\frac{2C_I\sqrt{C_2}|D|}{\sqrt{N^I}}\left(\sum_{i=1}^{N^I}|\cG\phi_\rho(\bX^I_i)-\cG\bar{\phi}_\rho(\bX^I_i)|^2\right)^{1/2}+\frac{2\sqrt{2}\sqrt{C_g^2+C_B^2}|\Gamma|}{\sqrt{N^B}}\left(\sum_{i=1}^{N^B}|\cB\phi_\rho(\bX^B_i)-\cB\bar{\phi}_\rho(\bX^B_i)|^2\right)^{1/2}.$
    }
\end{align*}
Taking $\phi_\rho= \balpha\bpsi $ and $\bar{\phi}_\rho=\bar{\balpha}\bpsi $, with $\balpha=(\alpha_1,\cdots,\alpha_M)$ and $\bpsi=(\psi_1,\cdots,\psi_M)^\top$, letting $\sigma_\text{max}(A)$ be the maximum singular values of $A$, where $A_{ij}=\cG\psi_j(\bX^I_i)$ as $1\le i\le N^I$ and $A_{ij}=\cB\psi_j(\bX^B_{i-N^I})$ as $N^I+1\le i\le N^I+N^B$, we obtain
\begin{align*}
    \widetilde{\cL}(\phi_\rho)-\widetilde{\cL}(\phi_a)\le\max\left\{\frac{2C_I\sqrt{C_2}|D|}{\sqrt{N^I}},\frac{2\sqrt{2}\sqrt{C_g^2+C_B^2}|\Gamma|}{\sqrt{N^B}}\right\}\sigma_\text{max}(A)\left\|\balpha-\bar{\balpha}\right\|_2.
\end{align*}
The $l^2$ norm estimate of $\balpha-\bar{\balpha}$ is given by the following theorem.
\begin{theorem}[Theorem 5.3.1 in \cite{Golub2013matrix}]
    Suppose that $\balpha$, $\br$, $\bar{\balpha}$, and $\bar{\br}$ satisfy  
    \begin{gather*}
        \|A\balpha^\top-b\|_2=\min, \quad \br=b-A\balpha^\top,\\
        \|(A+\delta A)\bar{\balpha}^\top-(b+\delta b)\|_2=\min, \quad \bar{\br}=(b+\delta b)-(A+\delta A)\bar{\balpha}^\top,
    \end{gather*}
    where $A$ has full column rank, $\|\delta A\|_2 < \sigma_\text{min}(A)$ and $\sigma_\text{min}(A)$ be the minimum singular values of $A$. Assume that $b$, $\br$ and $\balpha$ are not zero. Let $\theta\in(0,\pi/2)$ be defined by $\sin(\theta) = \|\br\|_2/\|b\|_2$. If  
    \begin{align*}
        \varepsilon_L=\max\left\{\frac{\|\delta A\|_2}{\|A\|_2}, \frac{\|\delta b\|_2}{\|b\|_2}\right\},\quad\nu=\frac{\|A\balpha^\top\|_2}{\sigma_\text{min}(A)\|\balpha\|_2},
    \end{align*}
    then
    \begin{align*}
        \frac{\|\balpha-\bar{\balpha}\|_2}{\|\balpha\|_2}\le\varepsilon_L\left\{\frac{\nu}{\cos(\theta)}+\left[1+\nu\tan(\theta)\right]\frac{\sigma_\text{max}(A)}{\sigma_\text{min}(A)}\right\}+O(\varepsilon_L^2).
    \end{align*}
\end{theorem}

In our setting, $(A,b)$ denotes the least-squares matrix and right-hand side corresponding to the chosen basis functions and collocation points, while $(A+\delta A,b+\delta b)$ represent their finite precision counterparts used in the numerical least-squares solver. The perturbation level $\varepsilon_L$ is therefore controlled by the floating-point accuracy and the stability of the solver. In practice, this implies that the optimization term can be made negligible compared to the approximation and statistical components.

\subsection{Layer Growth}
The layer-growth strategy enriches the RaNN hypothesis space by augmenting the feature set with newly generated neurons at deeper layers. As a result, the approximation error typically decreases as the network capacity increases. On the other hand, a larger number of basis functions may enlarge the statistical component unless the number of collocation points is increased accordingly, as discussed in \cite{Dang2024AGRNN}. In our experiments, we choose sufficiently many collocation points to balance this trade-off.

\section{Numerical Experiments} \label{sec:experiments}

\subsection{Zero Level Set} \label{Sec:zerolevelset}

We first describe how to approximate the zero level set of a scalar-valued level-set function. 
A key advantage of RaNN-based solvers is that the numerical approximation $\phi_\rho$ can be evaluated at arbitrary query points, which avoids interpolation when extracting the zero level set.

\paragraph{Scalar-valued case}
Given a numerical approximation $\phi_\rho(\bX)$ and an initial set of points $\{\bX_i\}_{i=1}^I$ located in a neighborhood of its zero level set, we refine each $\bX_i$ by local perturbations. 
Specifically, for each $i$, we generate a finite set of candidate points $\{\bX_{ij}\}_{j=1}^{J_i}$ in a small neighborhood of $\bX_i$, evaluate $\phi_\rho$ at these candidates, and update
\begin{align*}
\tilde{\bX}_i = \mathop{\arg\min}\limits_{\bX_{ij}}\ |\phi_\rho(\bX_{ij})|.
\end{align*}
Replacing $\bX_i$ by $\tilde{\bX}_i$ and repeating this procedure yields a set of refined points that approximate the zero level set $\{\phi_\rho=0\}$.

\paragraph{Initialization}
The critical step is to construct the initial points $\{\bX_i\}_{i=1}^I$ near the zero level set. We sample a set of candidate points in a slightly enlarged domain. For each candidate point, we examine the signs of $\phi_\rho$ at the point and its local neighbors; if the signs are not all identical in this neighborhood, then the neighborhood likely intersects the zero level set, and the candidate point is selected as an initial point. In two and three dimensions, MATLAB routines such as \texttt{contour} and \texttt{isosurface} can also be used for visualization and initialization.

\paragraph{Illustrative example}
Consider $\phi_\rho(\bx)=\sin(\pi x_1)-x_2$ on $[0,1]\times[0,1]$. We first sample a uniform grid on the enlarged domain $[-0.1,1.1]\times[-0.1,1.1]$ with $50$ points in each direction. Starting from the detected initial points, we set $J_i=9$ and generate $\{\bx_{ij}\}_{j=1}^9$ by the tensor-product stencil
\begin{align*}
(x_{i1}-h_1,x_{i1},x_{i1}+h_1)\times(x_{i2}-h_2,x_{i2},x_{i2}+h_2),
\end{align*}
where $h_1$ and $h_2$ are the grid spacings. In each iteration, the step size \(\bh=(h_1,h_2)\) is reduced by a factor of two. After five iterations, the resulting points lie in a small neighborhood of the zero level set. Figure \ref{fig:LevelSet} illustrates the refinement process.

\paragraph{Vector-valued case}
For a vector-valued level-set function $\bphi_\rho(\bX)\in\Real^k$, the zero level set corresponds to the intersection of the zero level sets of its components. In practice, we use the same refinement procedure as in the scalar case, but replace the scalar objective by a vector norm:
\begin{align*}
\tilde{\bX}_i=\mathop{\arg\min}\limits_{\bX_{ij}}\ \|\bphi_\rho(\bX_{ij})\|_{\ell^2}.
\end{align*}
For initialization, instead of intersecting separately selected point sets (which can be unstable), we select candidate points satisfying a small-residual criterion such as $\|\bphi_\rho(\bX)\|_{\ell^2}\le \varepsilon_0$ together with a local sign-change test applied componentwise.

\begin{figure}[!htbp]
    \centering
    \includegraphics[width=0.16\textwidth]{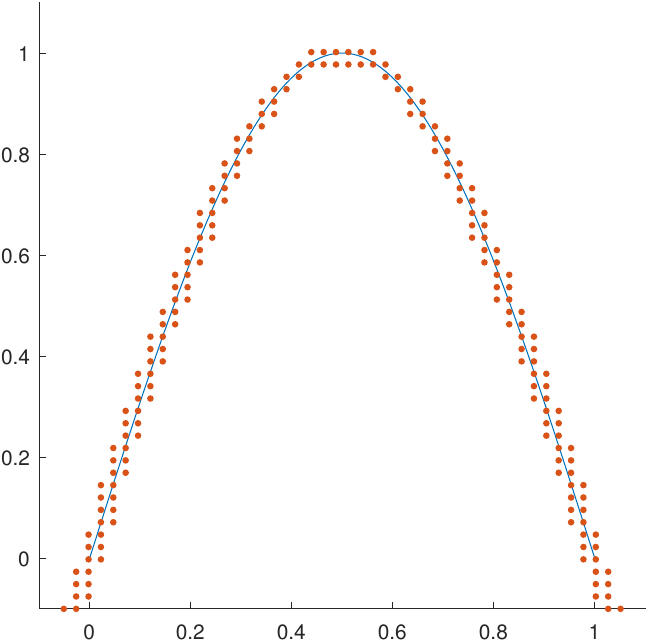}
    \includegraphics[width=0.16\textwidth]{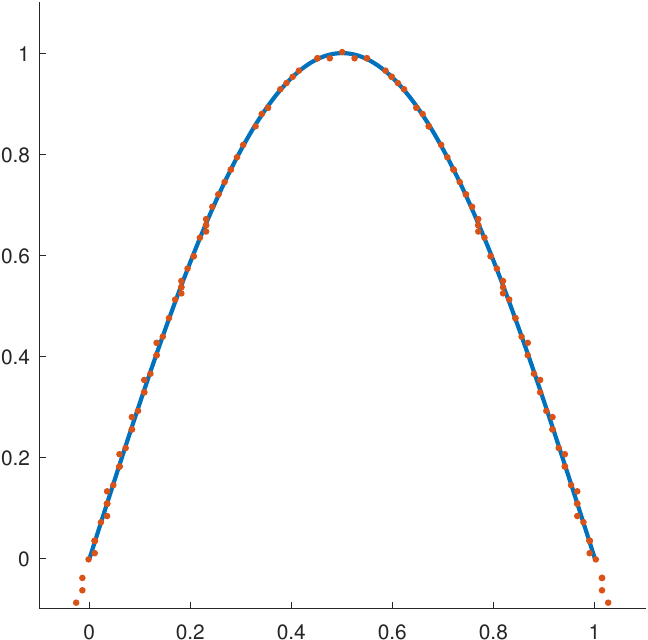}
    \includegraphics[width=0.16\textwidth]{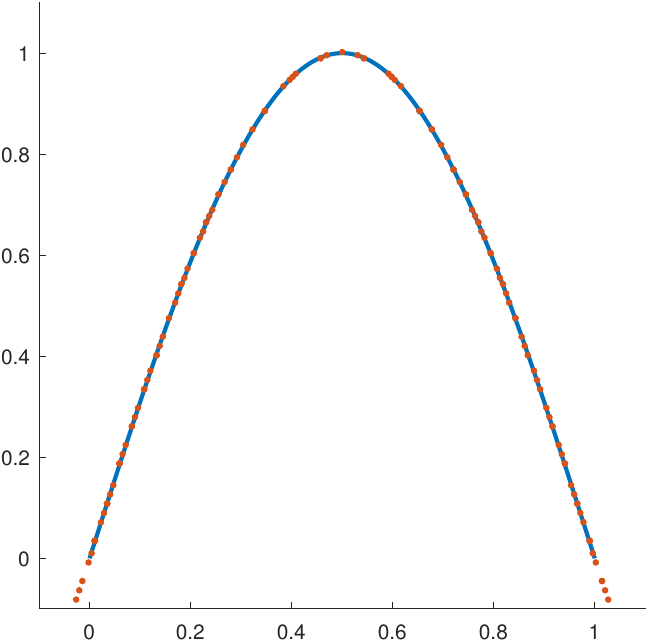}
    \includegraphics[width=0.16\textwidth]{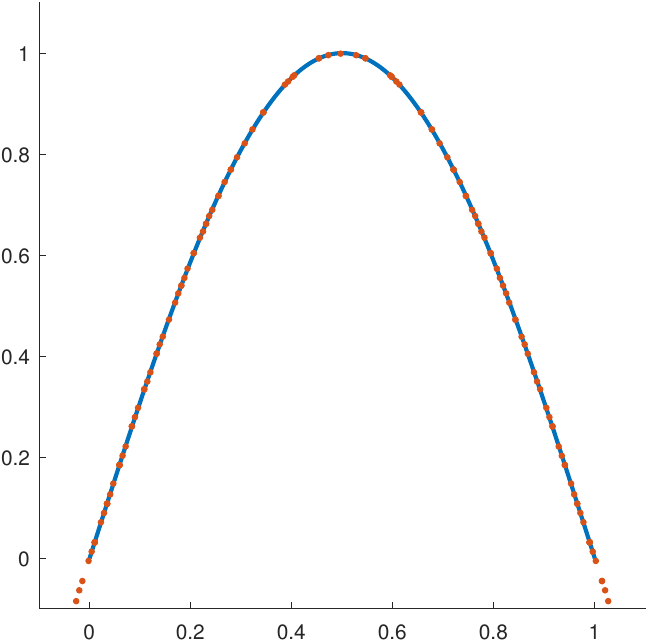}
    \includegraphics[width=0.16\textwidth]{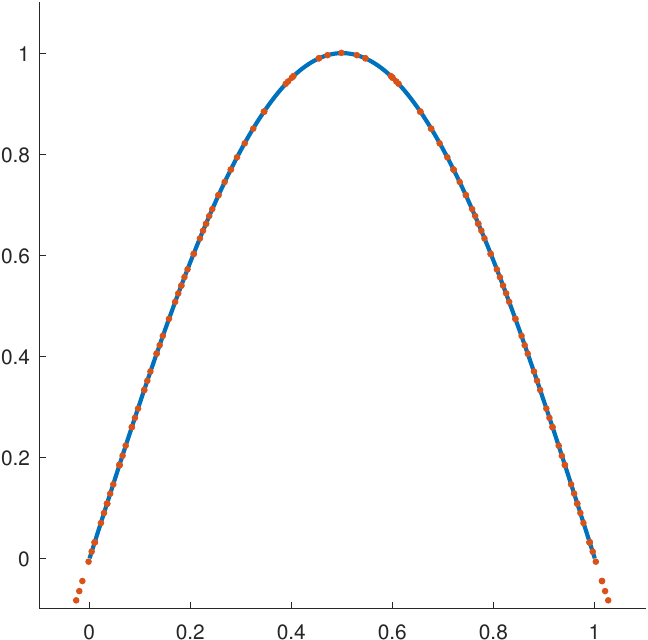}
    \includegraphics[width=0.16\textwidth]{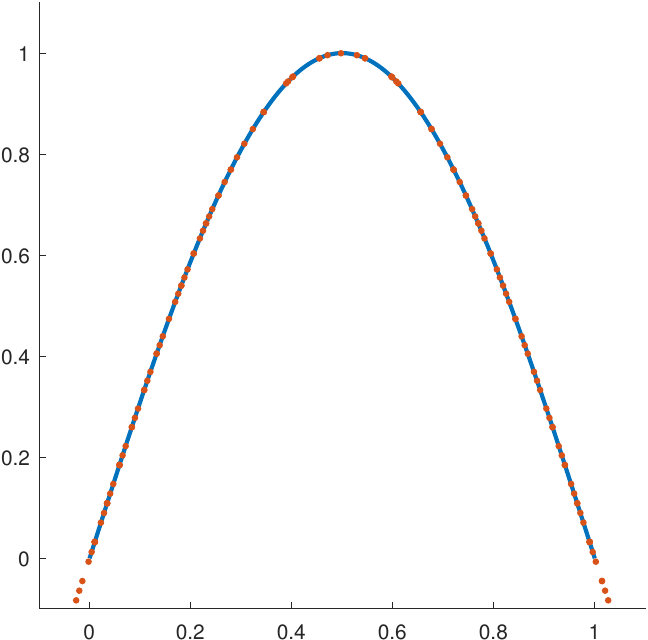}
    \caption{The iterative method to find the zero level set. }
    \label{fig:LevelSet}
\end{figure}

\subsection{Examples}

All experiments are conducted in MATLAB. To ensure reproducibility, we fix the random seed by \texttt{rng(1)}. The least-squares problems are solved by a QR-based routine (MATLAB \texttt{linsolve} with \texttt{opts.RECT=true}). In all examples, we use the Gaussian activation $\rho(x)=\exp(-x^2/2)$ in each layer. Since the basis functions are fixed, we compute the required differential operators using analytical derivatives, avoiding automatic differentiation and finite-difference approximations, which improves efficiency while maintaining high accuracy.

Our adaptive strategy concentrates collocation points in the tube $\Lambda_A$ around the zero level set. As a result, accuracy may deteriorate away from $\Lambda_A$. Therefore, when visualizing the solution, we restrict attention to points that are sufficiently close to $\Lambda_A$. To identify such points efficiently, we employ a $k$-nearest neighbor search (MATLAB \texttt{knnsearch}) to estimate the distance from plotting points to the collocation set.

The main parameters used in the experiments are summarized as follows:
\begin{itemize}
    \item $N^I$ and $N^B$ denote the numbers of interior and boundary collocation points.
    \item $\mathrm{E}$ and $\mathrm{Var}$ denote the mean and variance parameters used for the initial normal sampling in the adaptive-collocation procedure.
    \item $N$ is the total number of candidate points used to update interior/boundary collocation sets, $\varepsilon_A$ is the tube width defining $\Lambda_A=\{|\phi_\rho^1|\le \varepsilon_A\}$, and $N_A^I$, $N_A^B$ are the number of updated interior/boundary collocation points.
    \item $T$ is the final time. We distinguish the target domain (where accuracy is assessed/visualized) from a larger computational domain $\Omega$ used to update collocation points, as described in Subsection \ref{Sec:AdaptPoint}.
    \item $m_l$ and $\br_l$ denote the number of neurons and the uniform-sampling range for inner parameters at the $l$-th layer. For all examples, $m_1=2000$ and $m_2=m_3=1000$.
    \item $t_1$, $t_2$, and $t_3$ denote the running times using (i) normal collocation, (ii) adaptive collocation $\Lambda_A$, and (iii) layer growth, respectively.
\end{itemize}
Unless otherwise stated, red curves/surfaces indicate the reference (exact) solutions.

\begin{example}[1D inviscid forced Burgers' Equation] \label{ex1}
We consider the following 1D inviscid forced Burgers' equation:
\begin{align}
    \partial_t u + u\partial_x u + \partial_x V &= 0, \quad(t,x)\in(0,\infty)\times\Real,\\
    u(0,x) &= u_0(x),\quad x\in\Real,
\end{align}
where $V(x)$ is the potential. We explore four cases, varying $V(x)$ and $u_0(x)$, which were tested in \cite{Jin2003Levelset}. For all cases, we apply layer growth in the third-row experiments and the first two rows use single-layer RaNN baselines.
\end{example}

We present all parameters for each case in Table \ref{tab:ex1_parameter}.
\begin{table}[!htbp]
    \centering
    \begin{tabular}{|c|c|c|c|c|c|c|c|c|c|c|c|c|c|c|c|}
    \hline
     & $N^I$ & $N^B$ & E & Var & $N$ & $\varepsilon_A$ & T & $\Omega$ & $\br_1$ & $\br_2$ \\ \hline
    \textbf{Case 1} & \multirow{4}{*}{64000} & \multirow{4}{*}{5000} & (0,0) & (1,1) & \multirow{4}{*}{$51^3$} & 0.4 & 1 & $[-1,1]\times[-1,1]$ & \multirow{4}{*}{(3,3,3)} & \multirow{4}{*}{(5,5,5)} \\ \cline{1-1} \cline{4-5} \cline{7-9}
    \textbf{Case 2} &  &  & (0,0) & (1.5,1.5) &  & 0.2 & 1 & $[-1,1]\times[-1.1,1.1]$ &  &  \\ \cline{1-1} \cline{4-5} \cline{7-9}
    \textbf{Case 3} &  &  & (0,0.5) & (2,1.5) &  & 0.6 & 1 & $[-1.3,1.3]\times[-0.3,1.3]$ &  &  \\ \cline{1-1} \cline{4-5} \cline{7-9}
    \textbf{Case 4} &  &  & (0,0) & (1.5,1.5) &  & 0.9 & $\pi$ & $[-2,2]\times[-2,2]$ &  &  \\ \hline    
    \end{tabular}
    \caption{Parameters used in Example \ref{ex1}. }
    \label{tab:ex1_parameter}
\end{table}

\noindent\textbf{Case 1: The continuous initial data.}

In this case, the initial condition is
\begin{align}
    u_0(x)=-\sin(\pi x)
\end{align}
and the potential is $V(x)=0$. The numerical results are shown in Figure \ref{fig:ex1_case11}.

\begin{figure}[!htbp]
    \centering
    \includegraphics[width=0.20\textwidth]{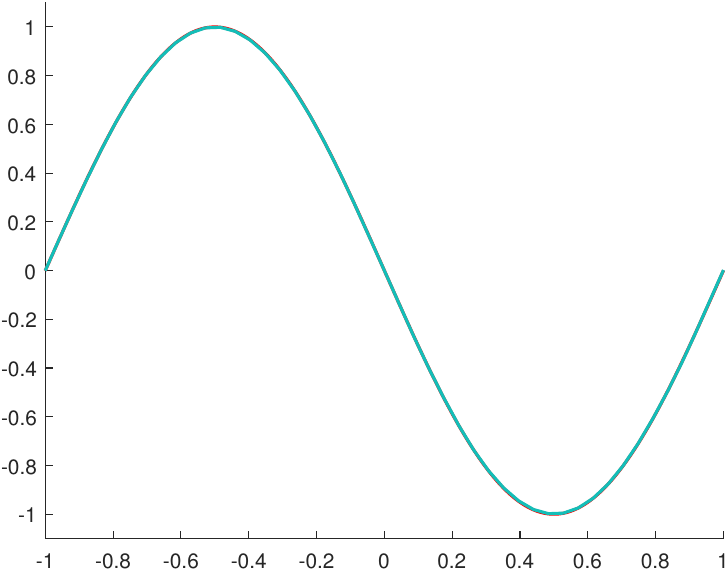}
    \includegraphics[width=0.20\textwidth]{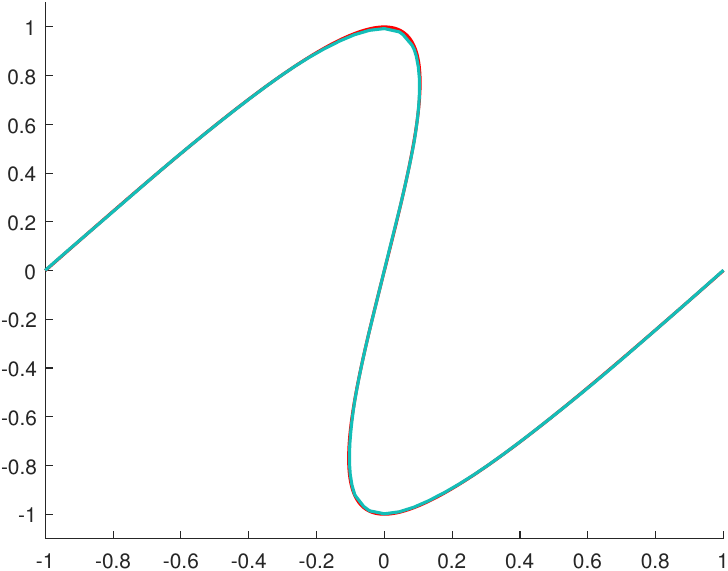}
    \includegraphics[width=0.20\textwidth]{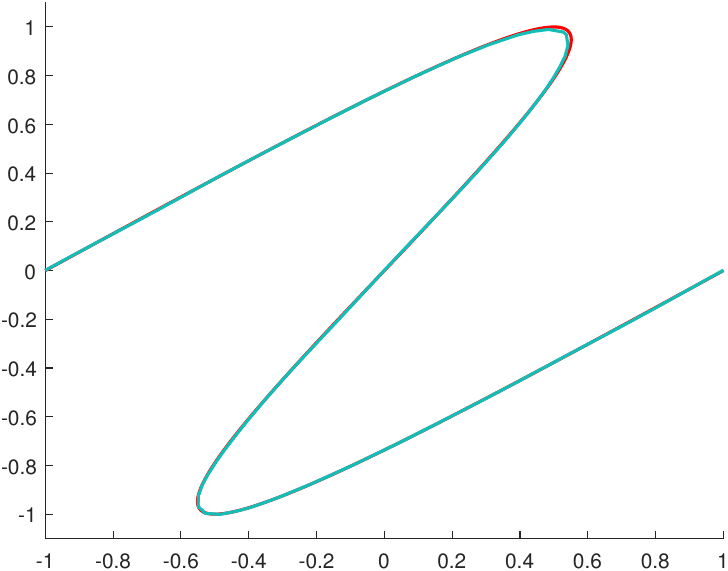}
    \includegraphics[width=0.20\textwidth]{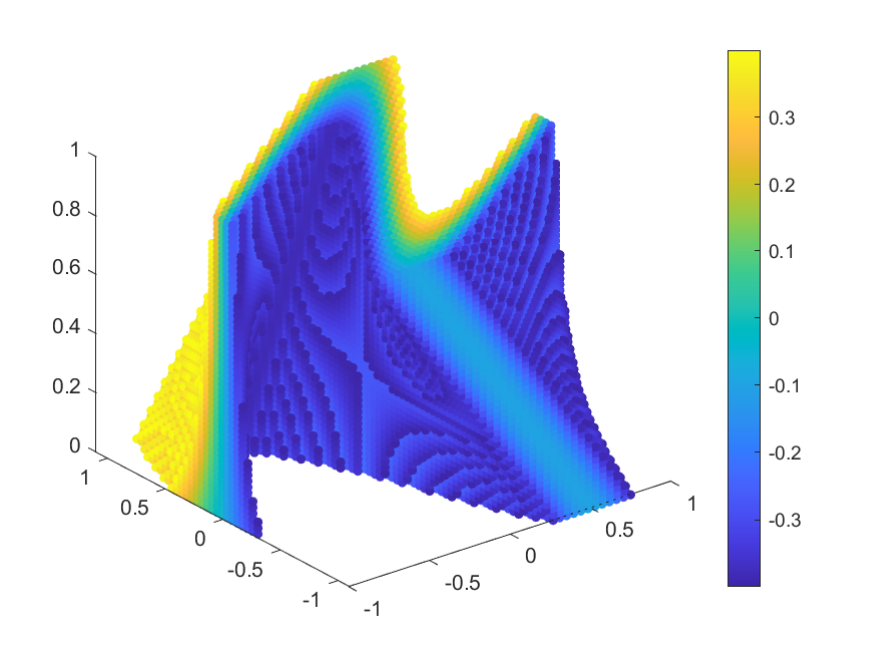}
    \includegraphics[width=0.20\textwidth]{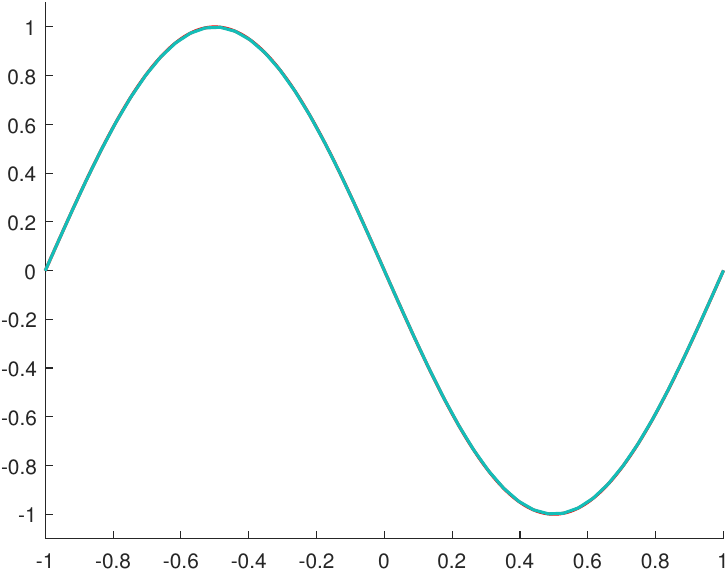}
    \includegraphics[width=0.20\textwidth]{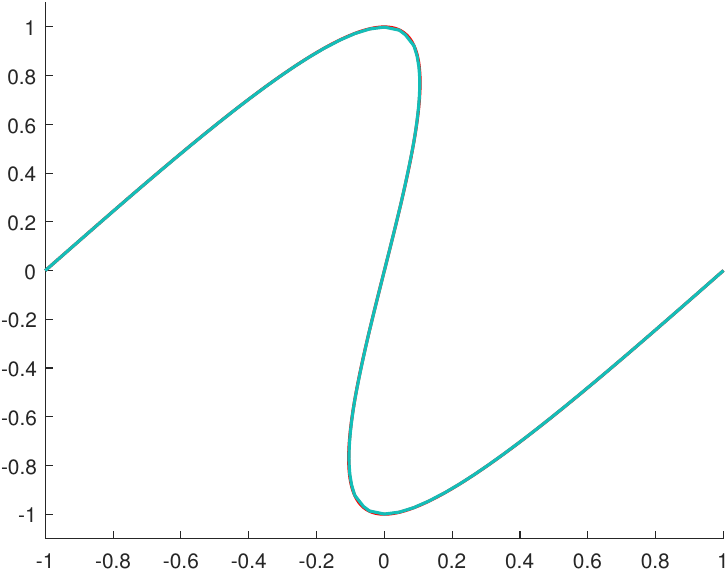}
    \includegraphics[width=0.20\textwidth]{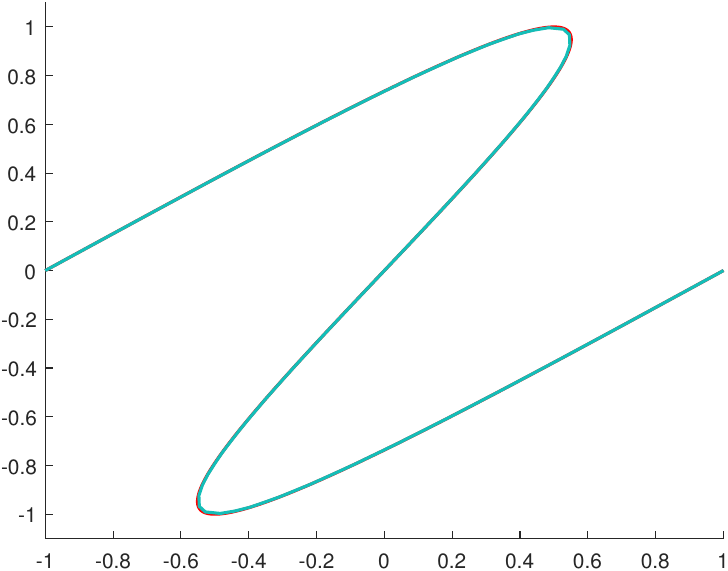}
    \includegraphics[width=0.20\textwidth]{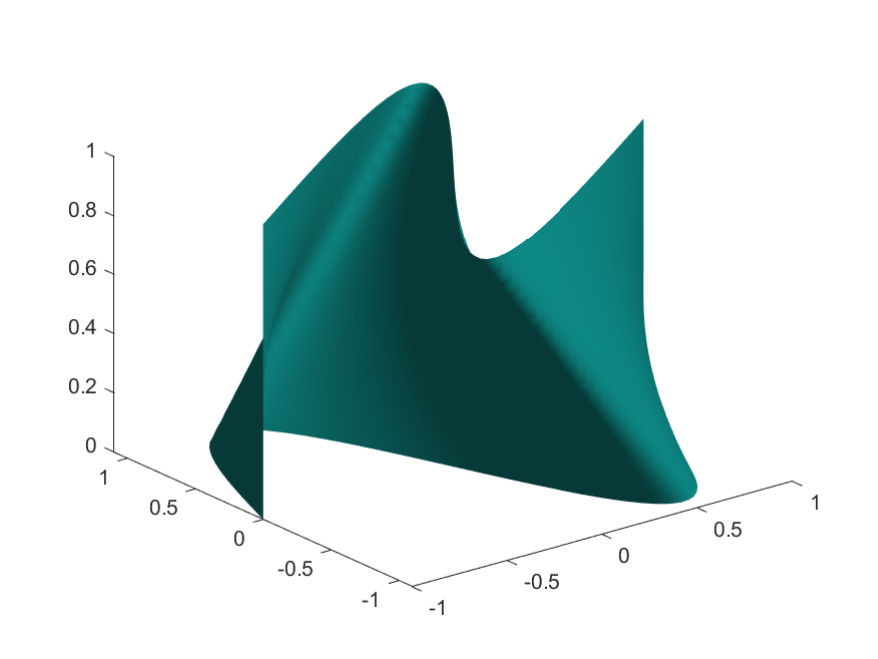}
    \includegraphics[width=0.20\textwidth]{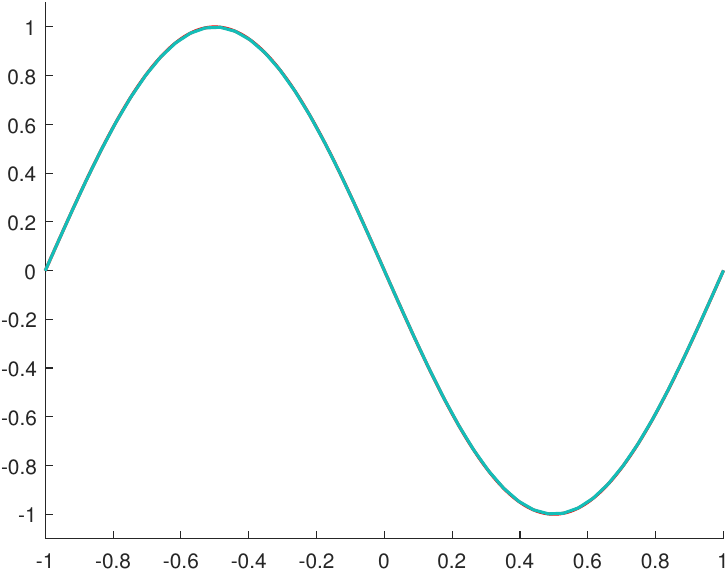}
    \includegraphics[width=0.20\textwidth]{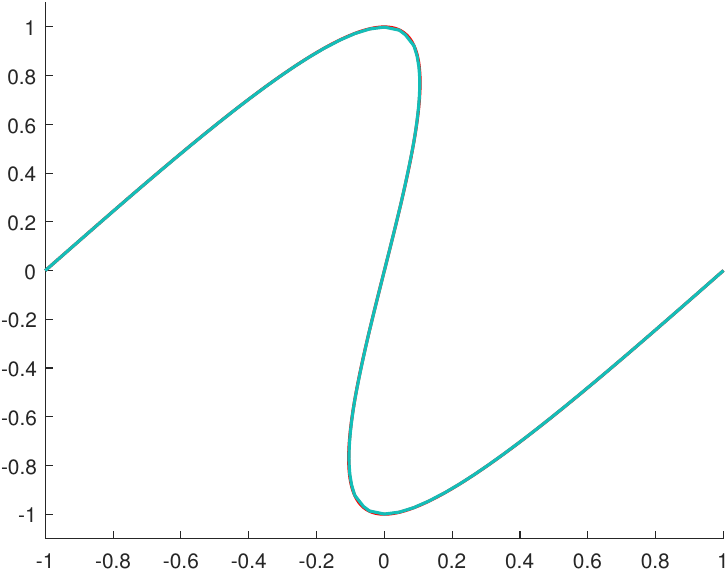}
    \includegraphics[width=0.20\textwidth]{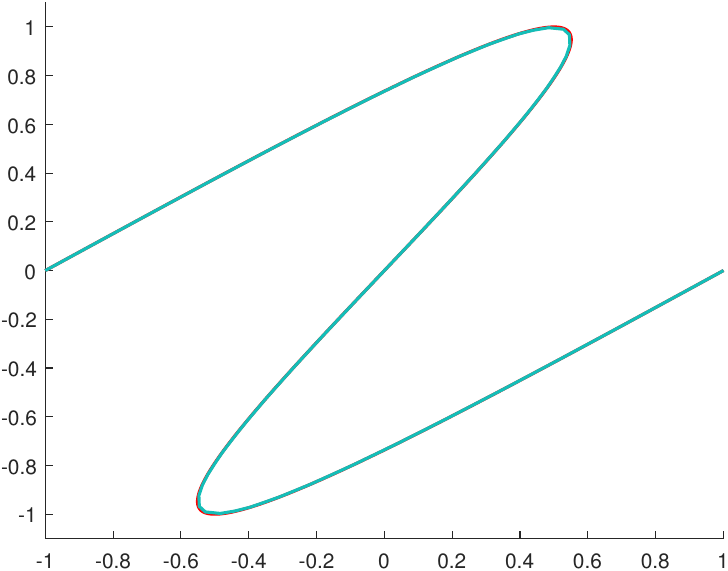}
    \includegraphics[width=0.20\textwidth]{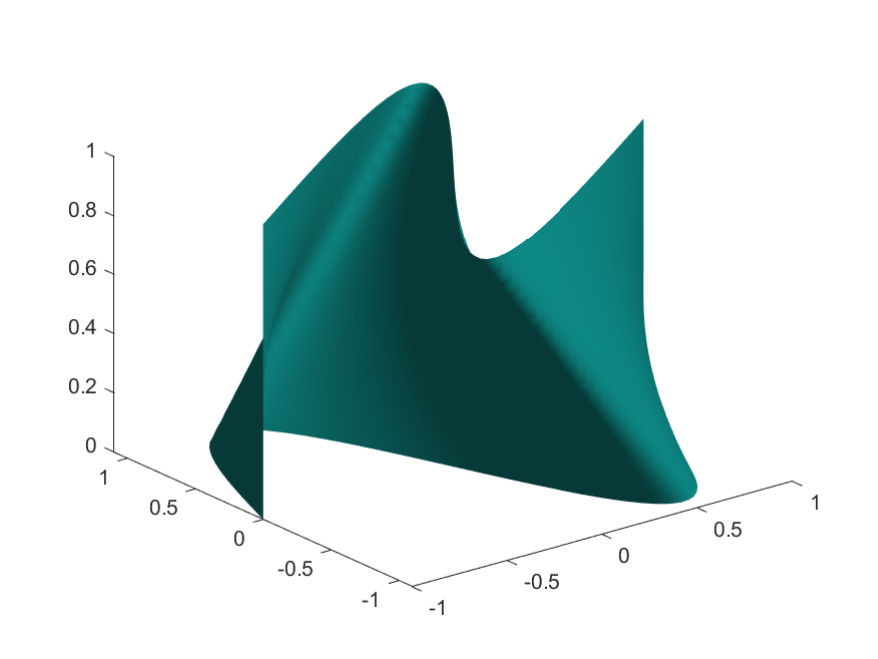}
    \caption{AG-RaNN method results for Case 1 of Example \ref{ex1}. Columns 1-3 correspond to $t=0,0.5,1$, and column 4 shows the zero level set. Rows (top to bottom) use normal collocation points, collocation set $\Lambda_A$ ($N_A^I=39784$, $N_A^B=836$), and the layer growth strategy, respectively. Running times: $(t_1,t_2,t_3)=(2.45,1.63,7.23)$.}
    \label{fig:ex1_case11}
\end{figure}

\noindent\textbf{Case 2: The rarefaction initial data.}

In this case, the initial condition is
\begin{align}
    u_0(x)=\begin{cases}
        -1,&x<0\\
        1,&x>0
    \end{cases}
\end{align}
and the potential is $V(x)=0$. The numerical results are shown in Figure \ref{fig:ex1_case2}.

\begin{figure}[!htbp]
    \centering
    \includegraphics[width=0.20\textwidth]{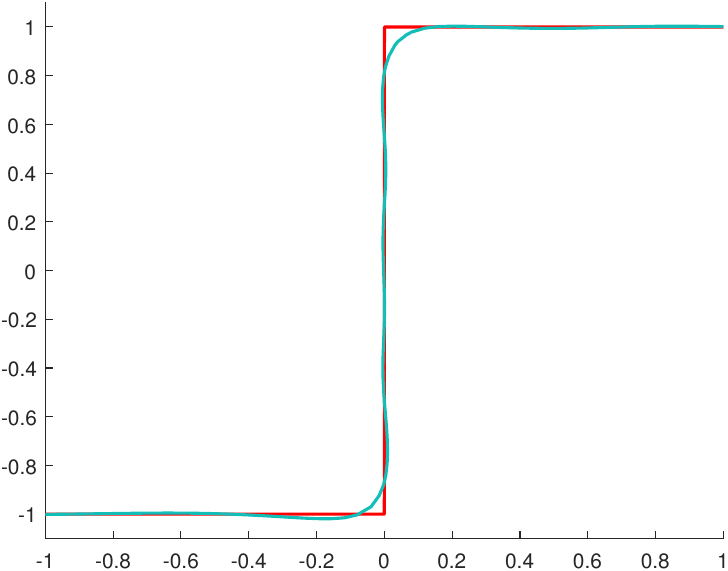}
    \includegraphics[width=0.20\textwidth]{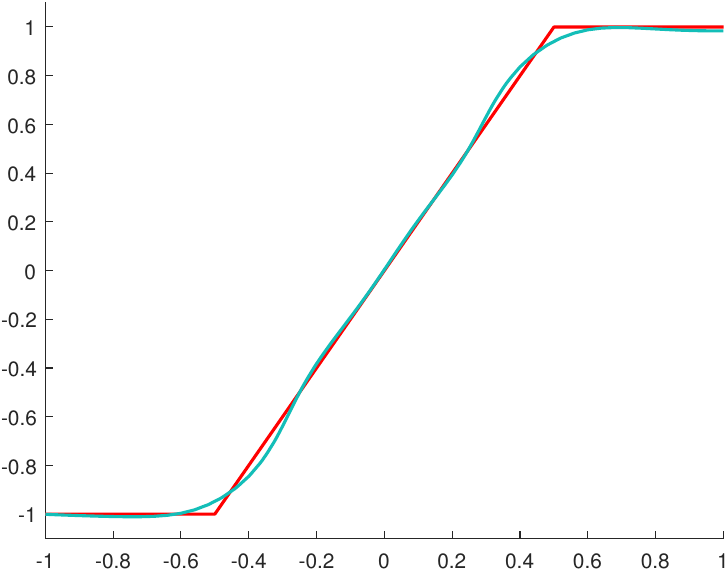}
    \includegraphics[width=0.20\textwidth]{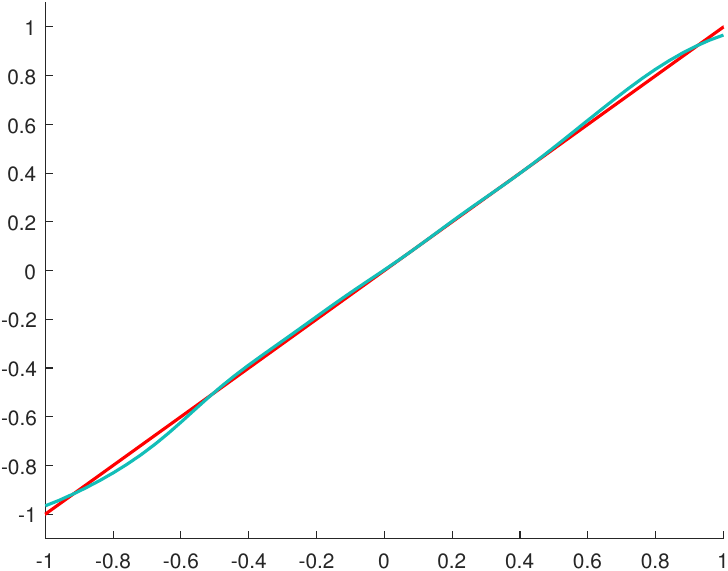}
    \includegraphics[width=0.20\textwidth]{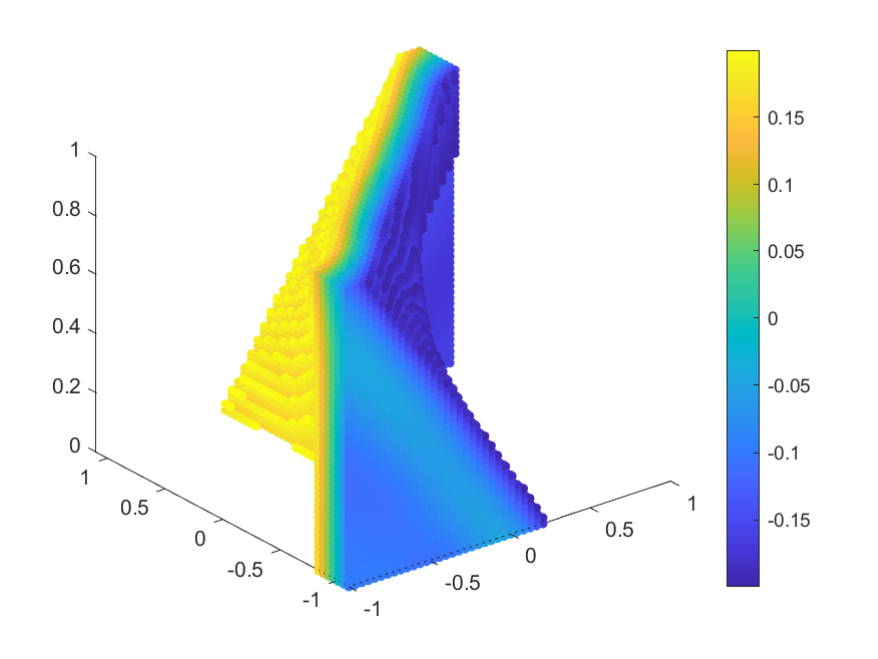}
    \includegraphics[width=0.20\textwidth]{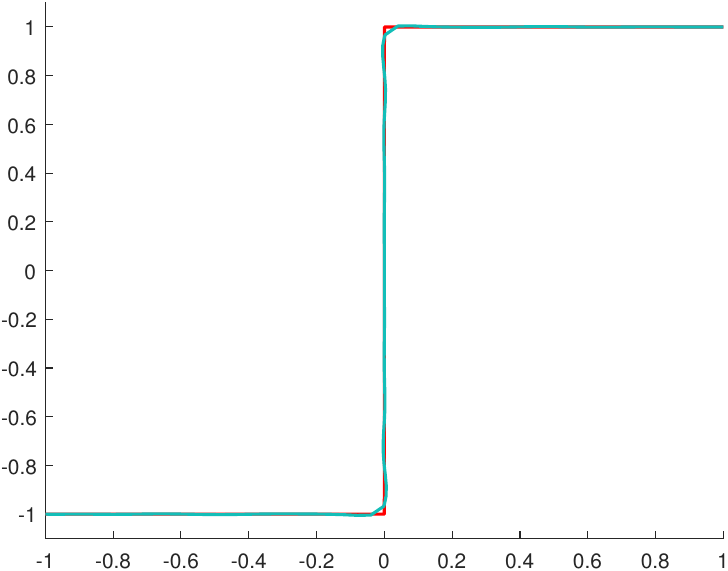}
    \includegraphics[width=0.20\textwidth]{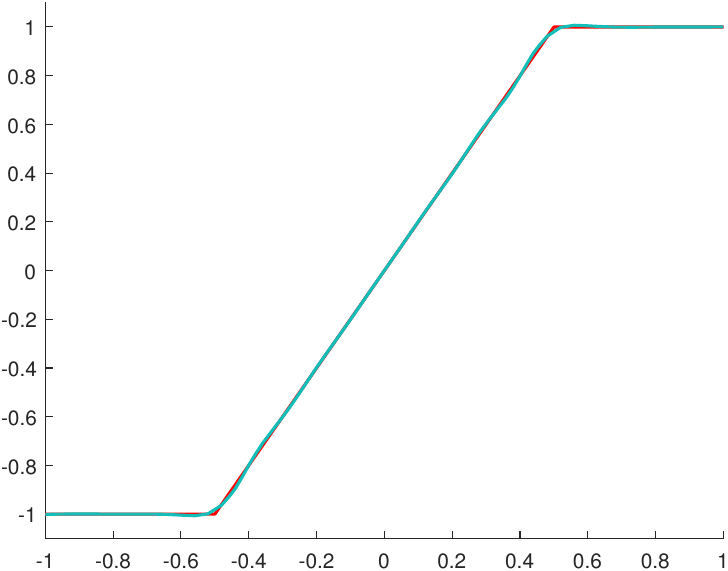}
    \includegraphics[width=0.20\textwidth]{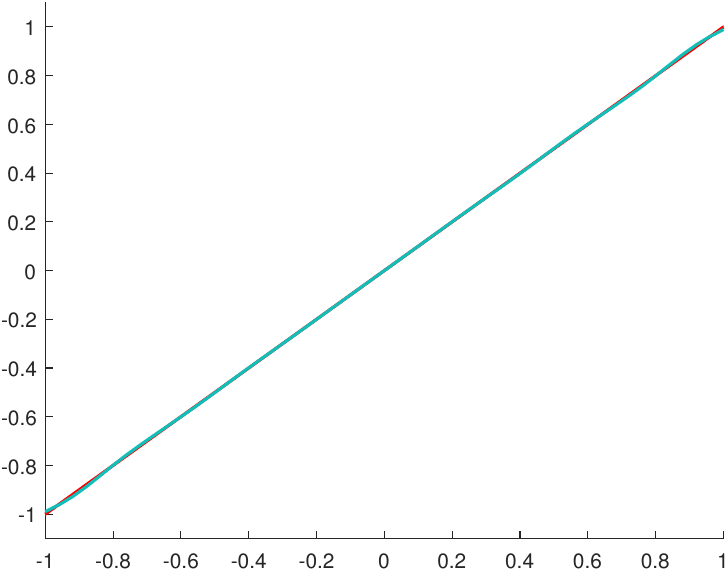}
    \includegraphics[width=0.20\textwidth]{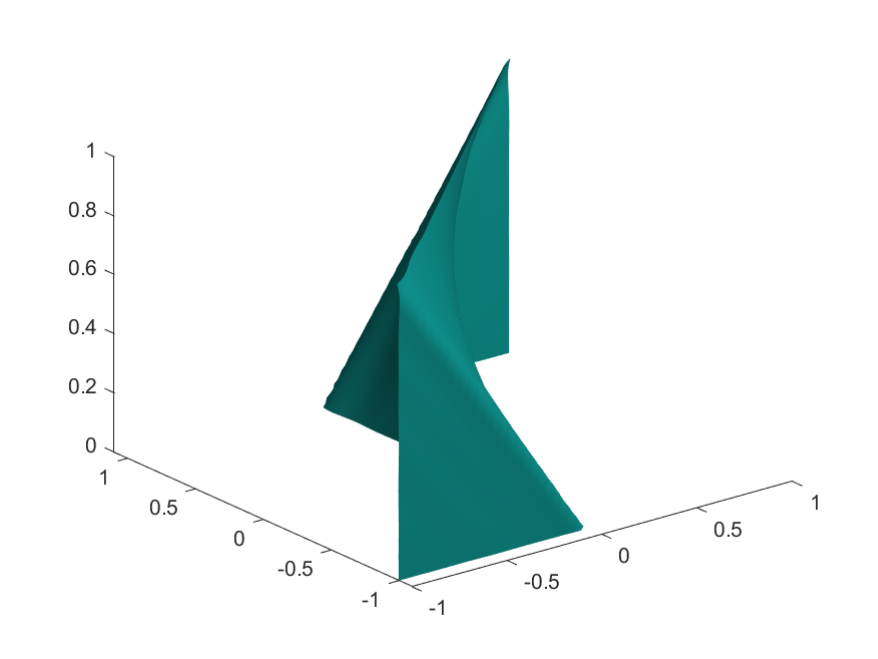}
    \includegraphics[width=0.20\textwidth]{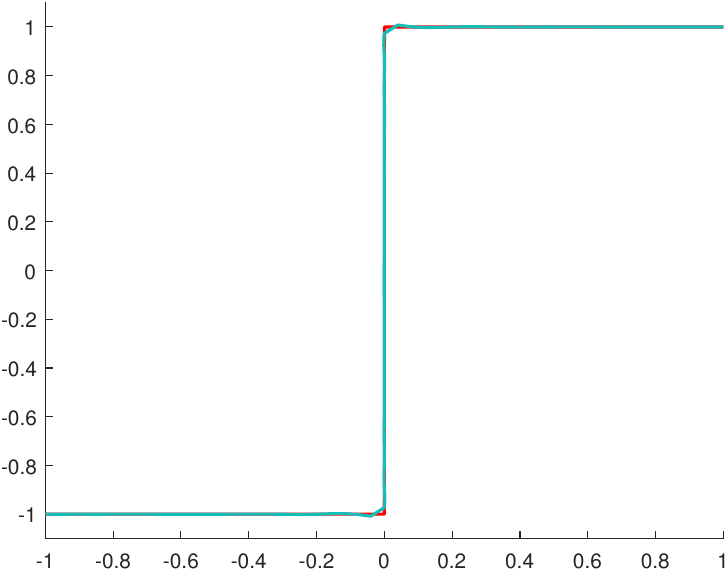}
    \includegraphics[width=0.20\textwidth]{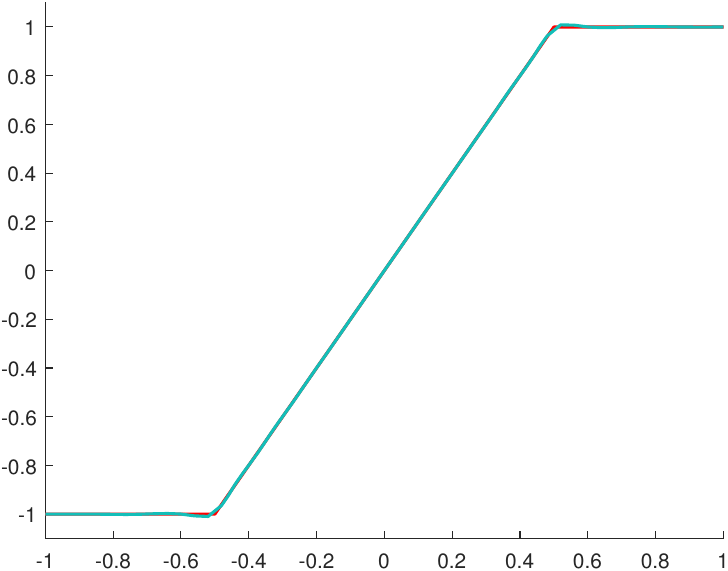}
    \includegraphics[width=0.20\textwidth]{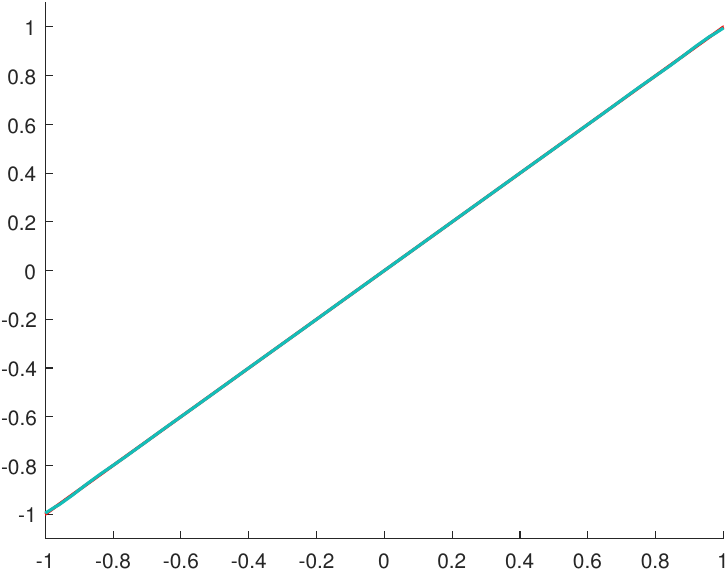}
    \includegraphics[width=0.20\textwidth]{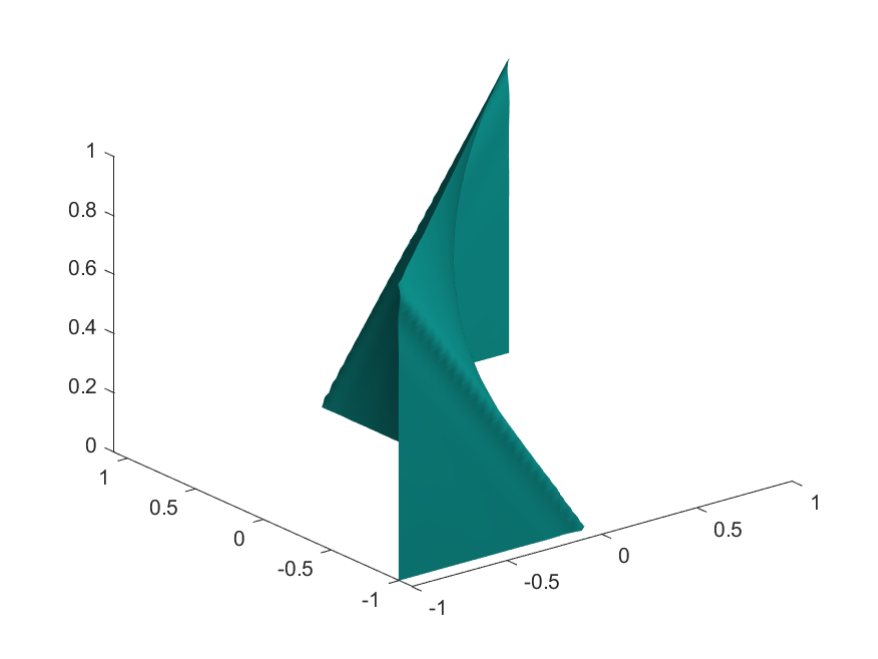}
    \caption{AG-RaNN method results for Case 2 of Example \ref{ex1}. Columns 1-3 correspond to $t=0,0.5,1$, and column 4 shows the zero level set. Rows (top to bottom) use normal collocation points, collocation set $\Lambda_A$ ($N_A^I=33051$, $N_A^B=799$), and the layer growth strategy, respectively. Running times: $(t_1,t_2,t_3)=(2.47,1.38,8.70)$.}
    \label{fig:ex1_case2}
\end{figure}

\noindent\textbf{Case 3: The shock initial data.}

In this case, the initial condition is
\begin{align}
    u_0(x)=\begin{cases}
        1,&x<0\\
        0,&x>0
    \end{cases}
\end{align}
and the potential is $V(x)=0$. The numerical results are shown in Figure \ref{fig:ex1_case3}. 
\begin{figure}[!htbp]
    \centering
    \includegraphics[width=0.20\textwidth]{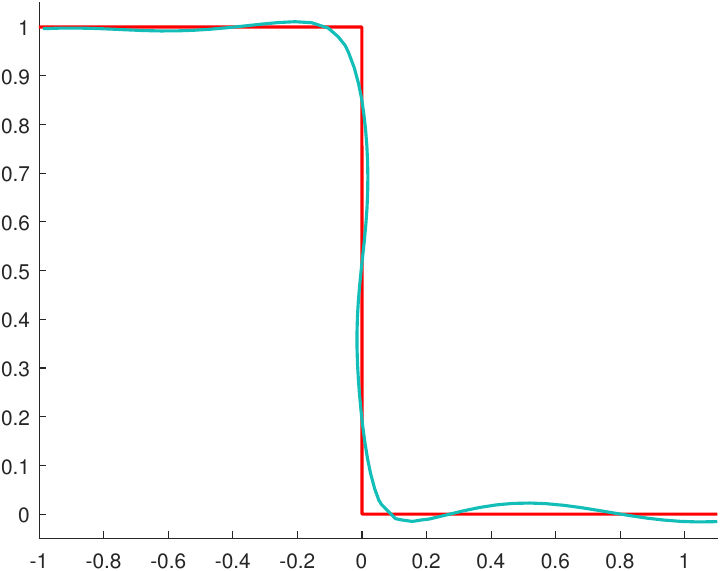}
    \includegraphics[width=0.20\textwidth]{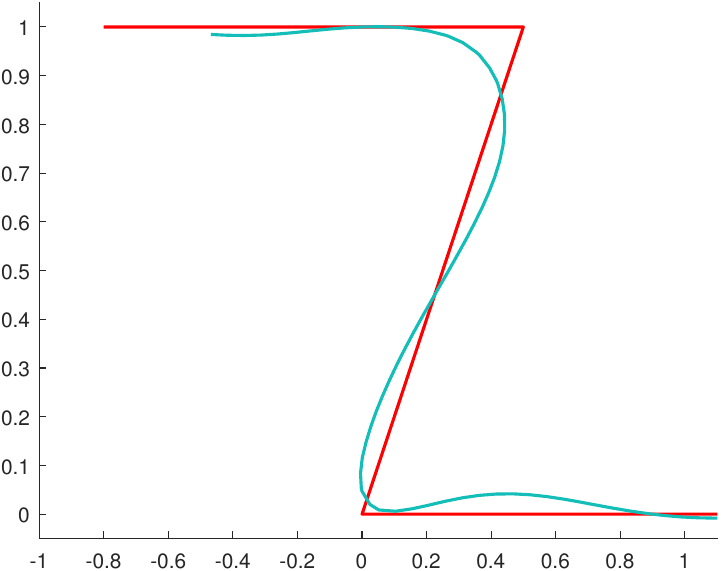}
    \includegraphics[width=0.20\textwidth]{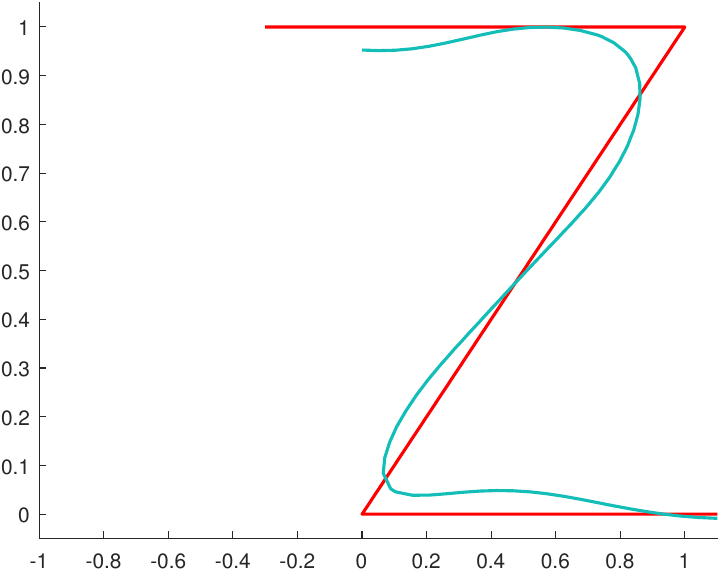}
    \includegraphics[width=0.20\textwidth]{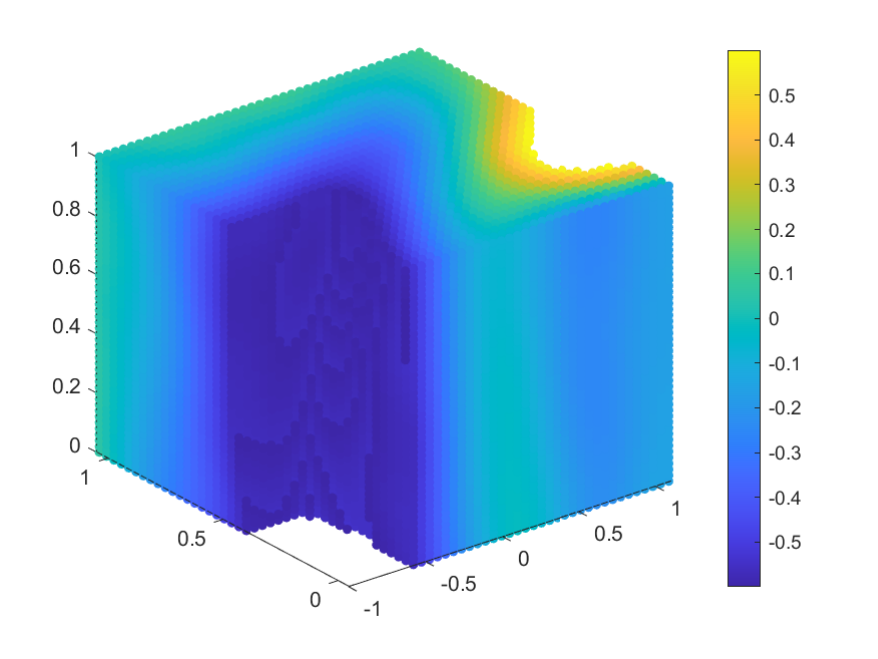}
    \includegraphics[width=0.20\textwidth]{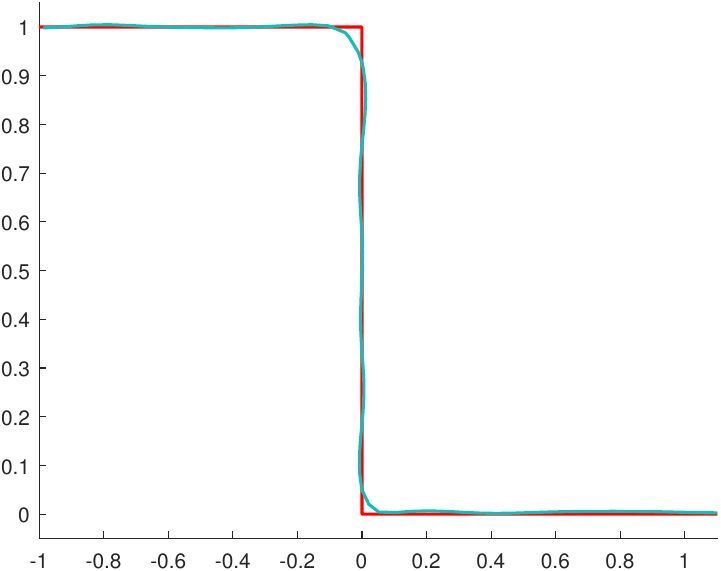}
    \includegraphics[width=0.20\textwidth]{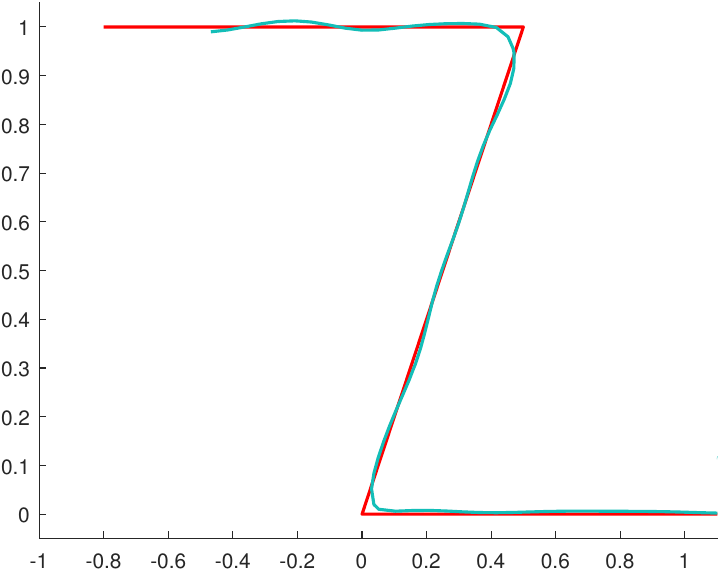}
    \includegraphics[width=0.20\textwidth]{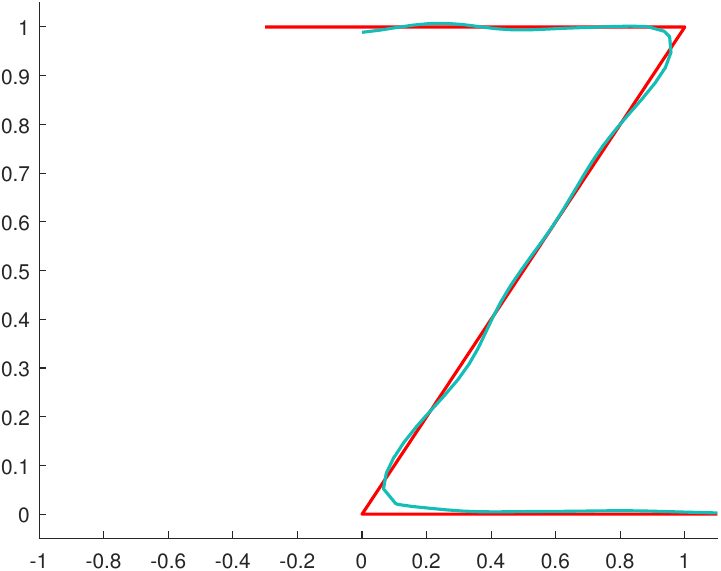}
    \includegraphics[width=0.20\textwidth]{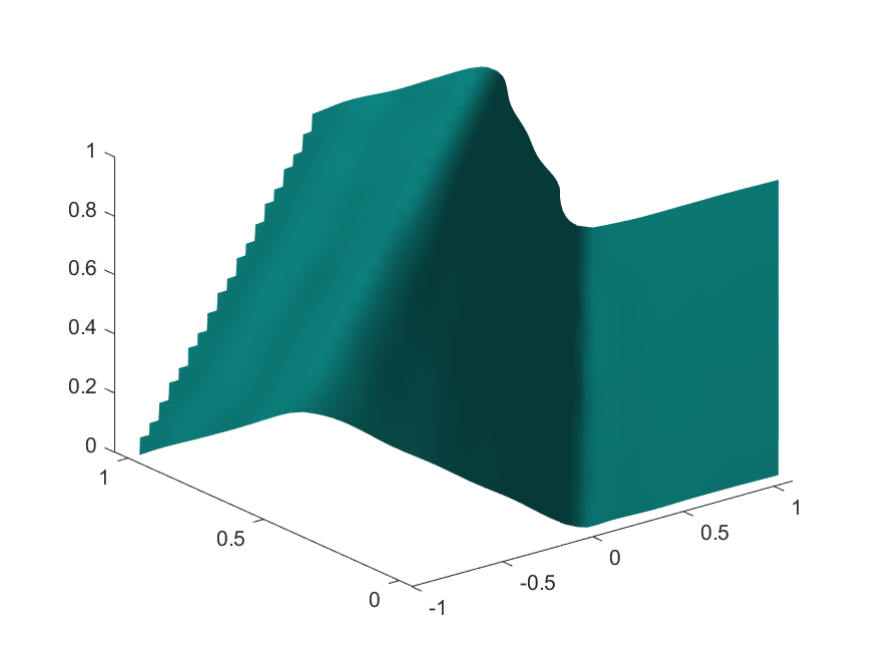}
    \includegraphics[width=0.20\textwidth]{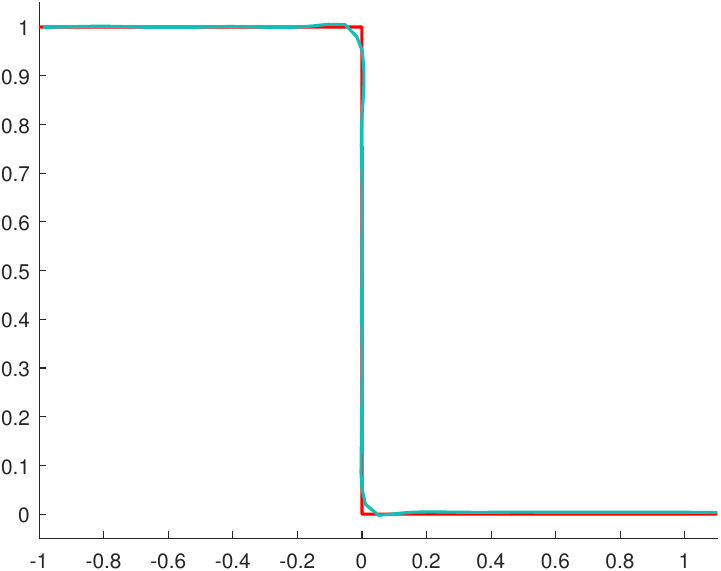}
    \includegraphics[width=0.20\textwidth]{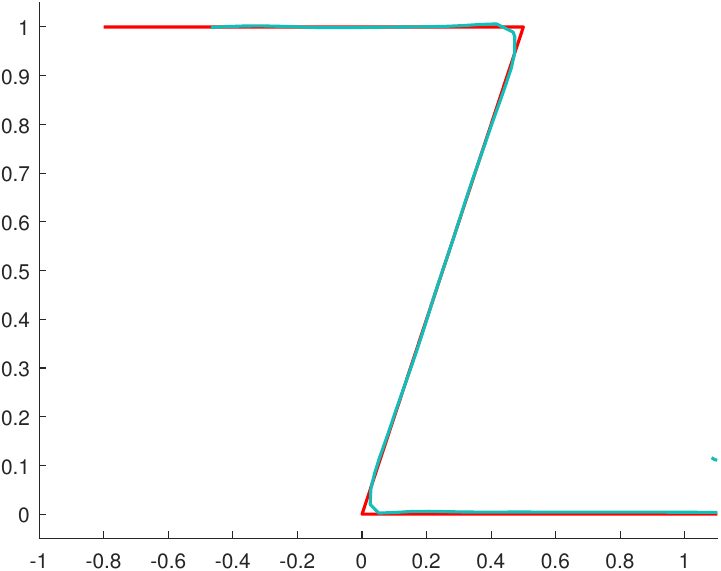}
    \includegraphics[width=0.20\textwidth]{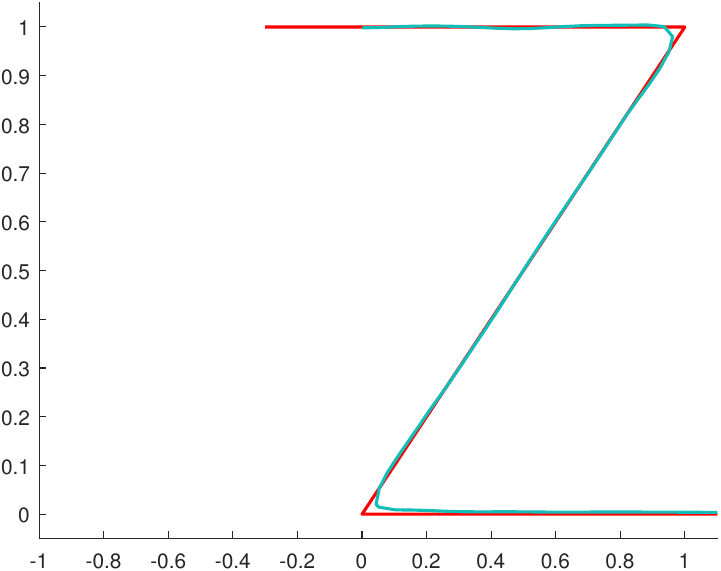}
    \includegraphics[width=0.20\textwidth]{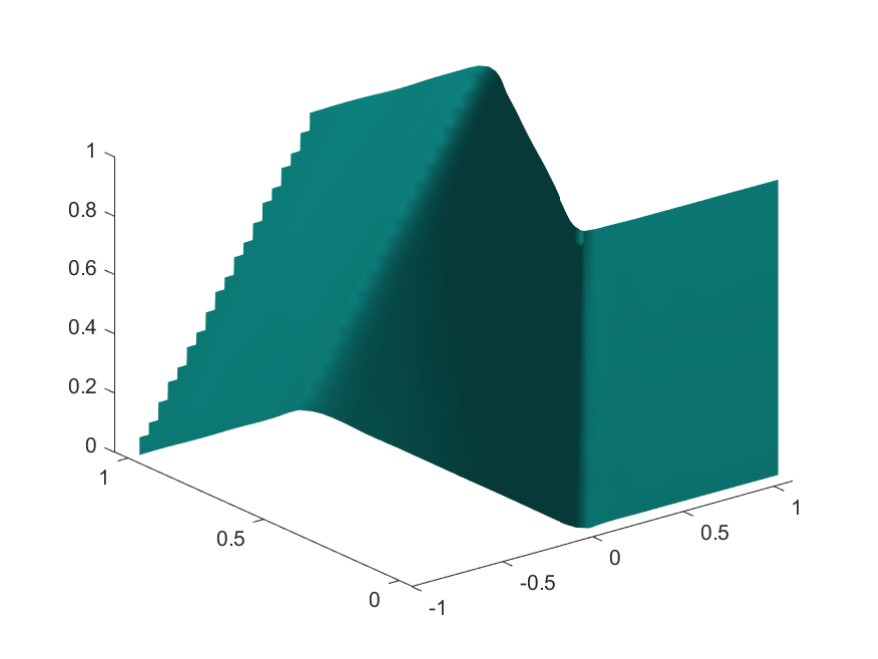}
    \caption{AG-RaNN method results for Case 3 of Example \ref{ex1}. Columns 1-3 correspond to $t=0,0.5,1$, and column 4 shows the zero level set. Rows (top to bottom) use normal collocation points, collocation set $\Lambda_A$ ($N_A^I=96122$, $N_A^B=1685$), and the layer growth strategy, respectively. Running times: $(t_1,t_2,t_3)=(2.47,5.52,23.28)$.}
    \label{fig:ex1_case3}
\end{figure}

\noindent\textbf{Case 4: The harmonic oscillator.}

In this case, the initial condition is
\begin{align}
    u_0(x)=-\tanh(5x)
\end{align}
and the potential is $V(x)=\frac{1}{2}x^2$. The numerical results are shown in Figure \ref{fig:ex1_case4}.

\begin{figure}[!htbp]
    \centering
    \includegraphics[width=0.20\textwidth]{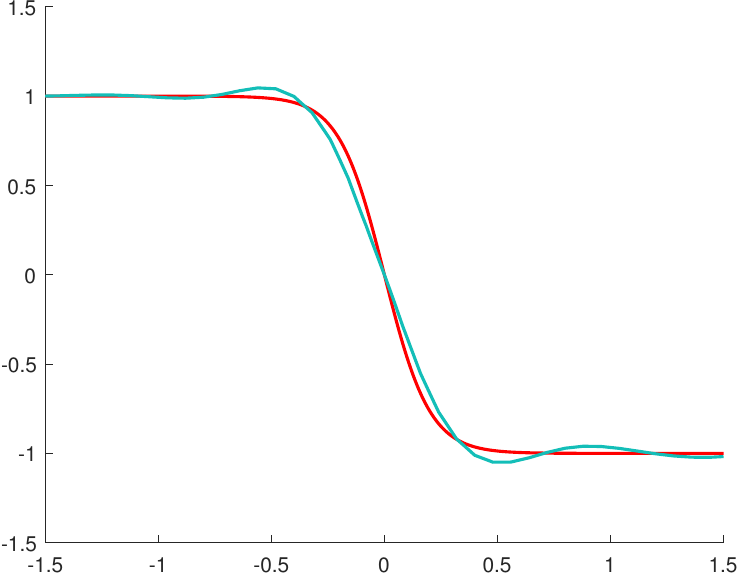}
    \includegraphics[width=0.20\textwidth]{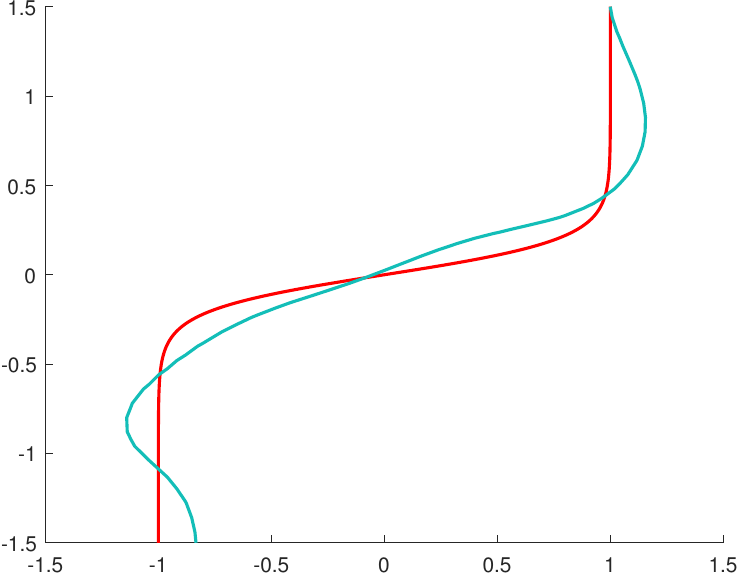}
    \includegraphics[width=0.20\textwidth]{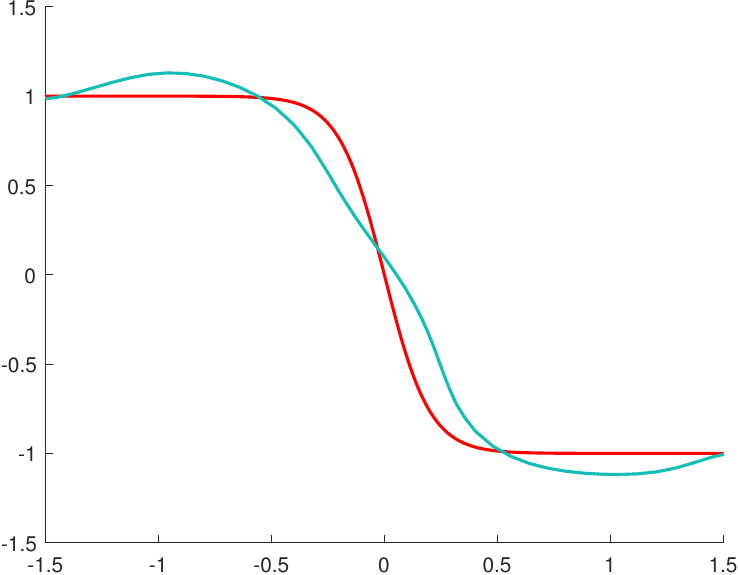}
    \includegraphics[width=0.20\textwidth]{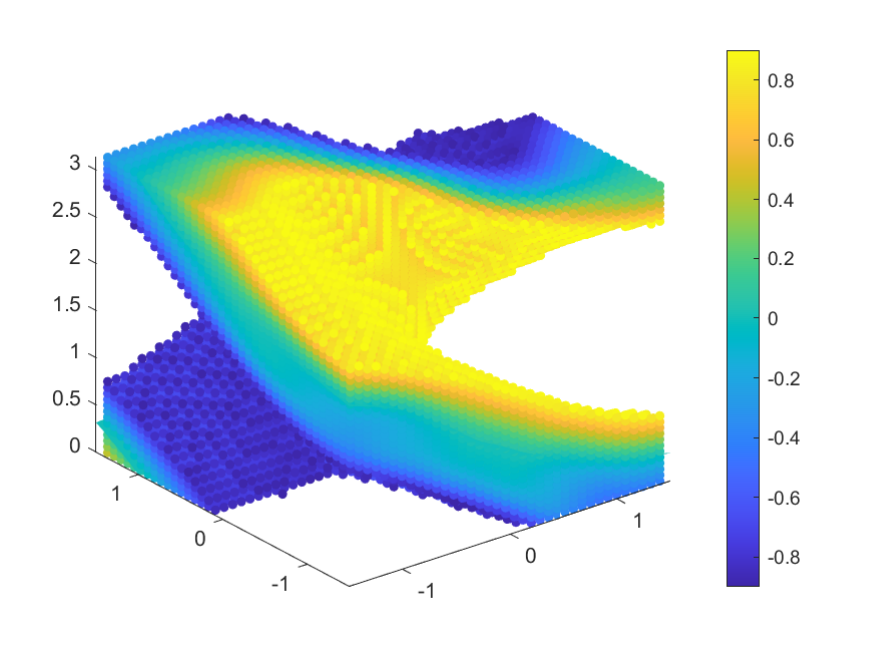}
    \includegraphics[width=0.20\textwidth]{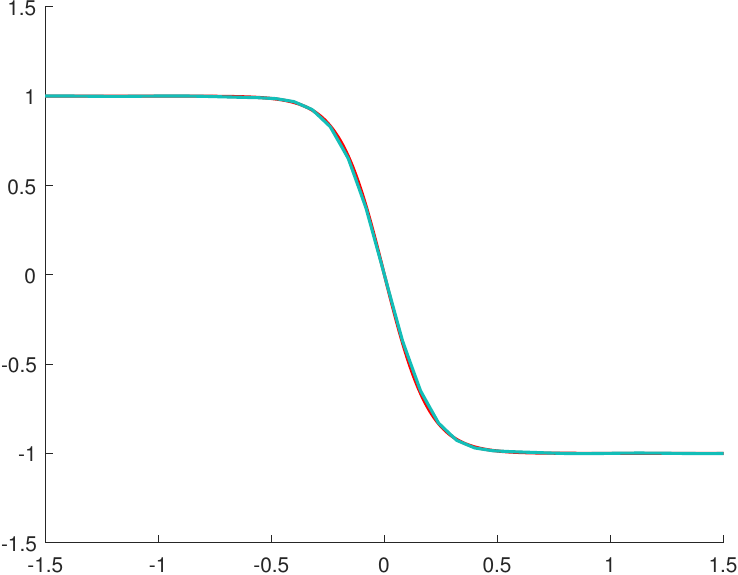}
    \includegraphics[width=0.20\textwidth]{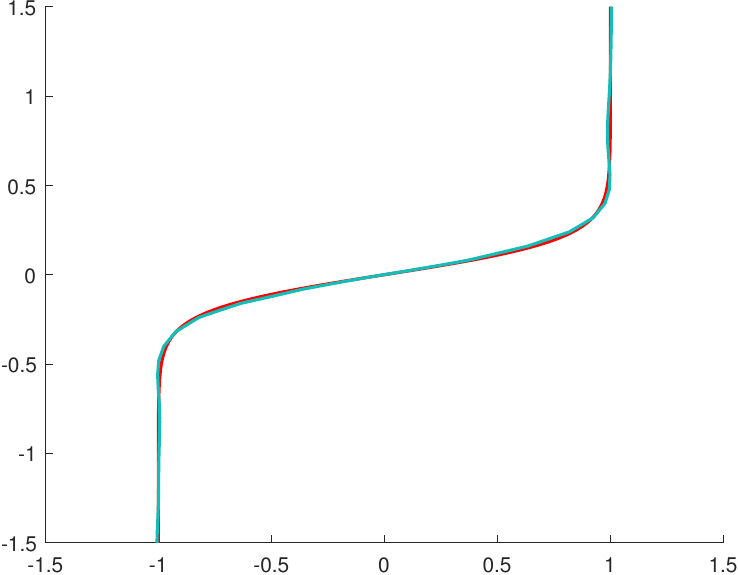}
    \includegraphics[width=0.20\textwidth]{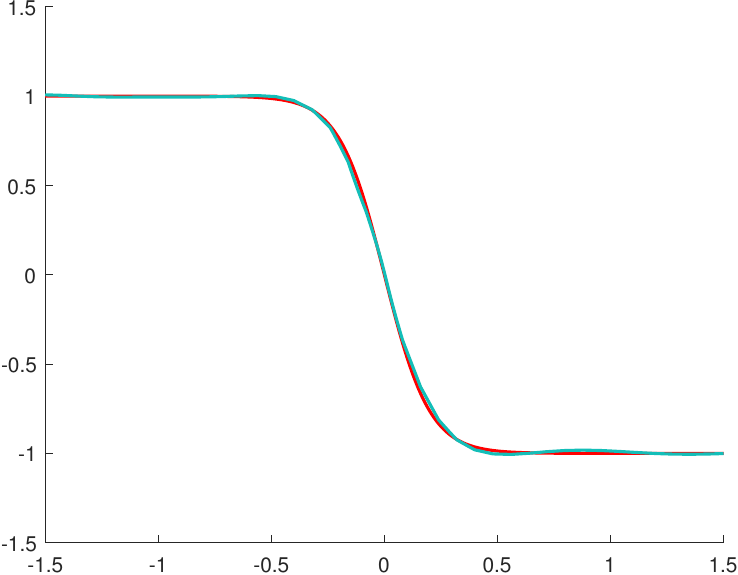}
    \includegraphics[width=0.20\textwidth]{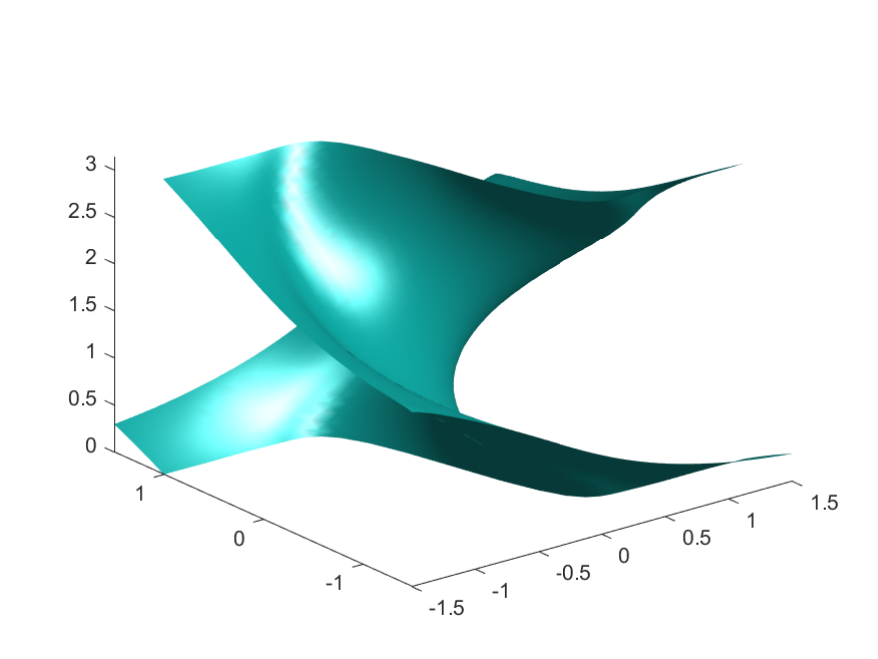}
    \includegraphics[width=0.20\textwidth]{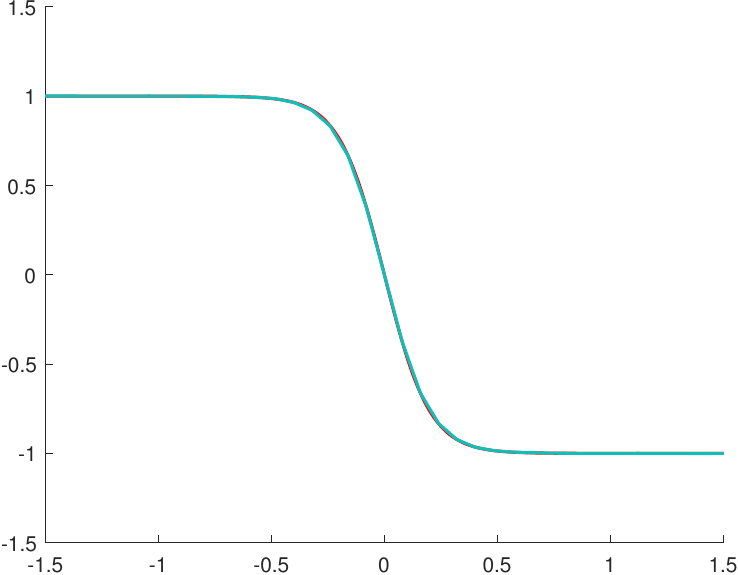}
    \includegraphics[width=0.20\textwidth]{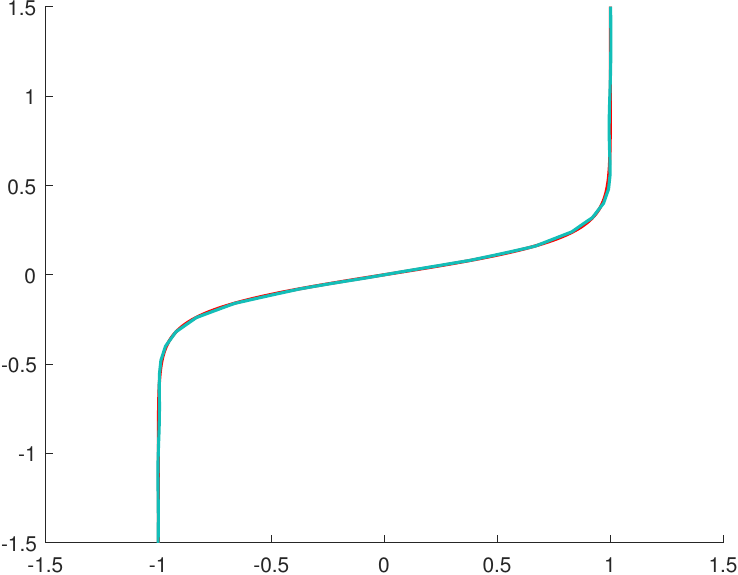}
    \includegraphics[width=0.20\textwidth]{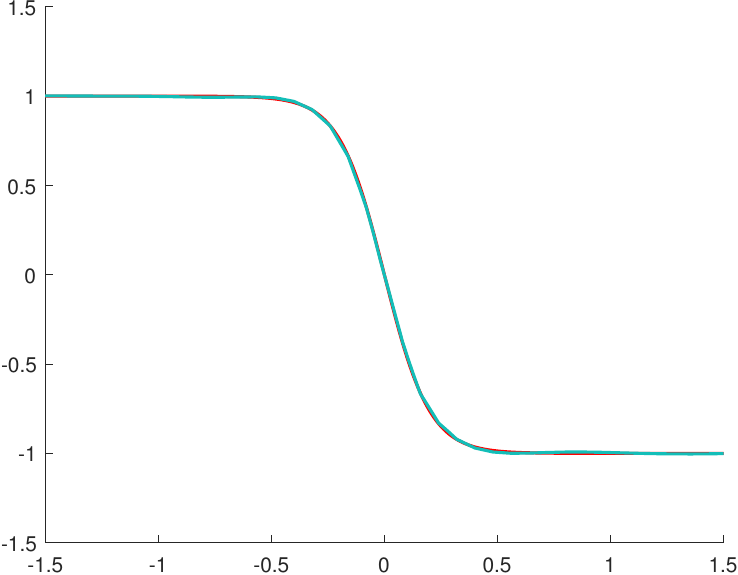}
    \includegraphics[width=0.20\textwidth]{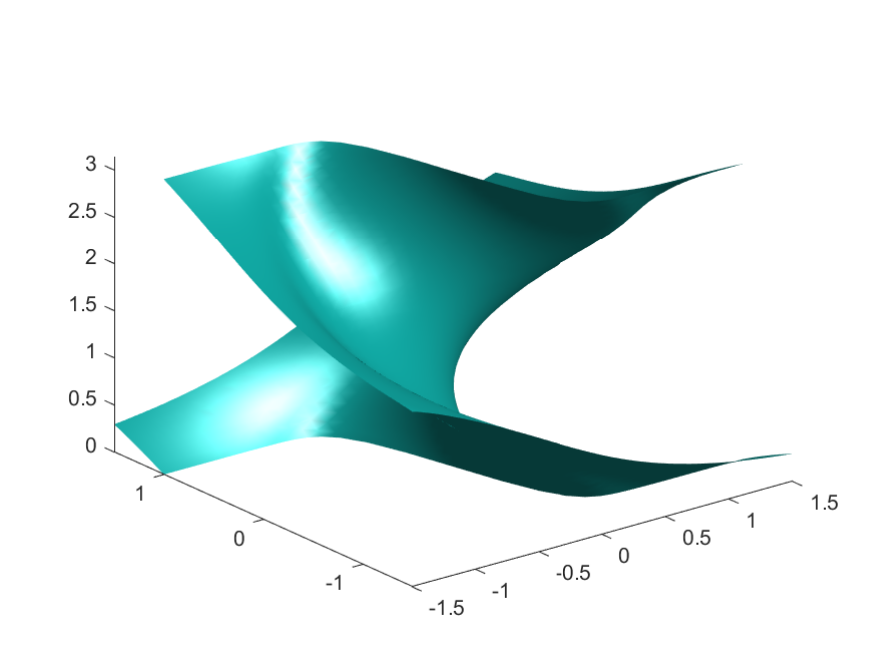}
    \caption{AG-RaNN method results for Case 4 of Example \ref{ex1}. Columns 1-3 correspond to $t=0,\pi/2,\pi$, and column 4 shows the zero level set. Rows (top to bottom) use normal collocation points, collocation set $\Lambda_A$ ($N_A^I=55236$, $N_A^B=1142$), and the layer growth strategy, respectively. Running times: $(t_1,t_2,t_3)=(2.62,2.07,9.87)$.}
    \label{fig:ex1_case4}
\end{figure}

For all cases, the results obtained by normal collocation points are less accurate. With collocation set $\Lambda_A$, the numerical solutions are closer to the exact solutions. After layer growth, the numerical solutions have higher precision. Compared to the results obtained by finite difference method in \cite{Jin2003Levelset}, RaNN method can get the solutions with similar accuracy at a faster speed, and we are able to compute the solution over a longer time horizon in these tests.

\begin{example}[2D Burgers' Equation] \label{ex2}
We also test the following 2D inviscid Burgers' equation:
\begin{align}
    \partial_t u + u\partial_x u + u\partial_y u &= 0, \quad(t,x,y)\in(0,\infty)\times\Real^2,\\
    u(0,x,y) &= u_0(x,y),\quad (x,y)\in\Real^2.
\end{align}
We explore two cases, varying $u_0(x,y)$. 
\end{example}

We present all parameters for each case in Table \ref{tab:ex2_parameter}.
\begin{table}[!htbp]
    \centering
    \begin{tabular}{|c|c|c|c|c|c|c|c|c|c|c|c|c|c|c|c|}
    \hline
     & $N^I$ & $N^B$ & E & Var & $N$ & $\varepsilon_A$ & T & $\Omega$ & $\br_1$ & $\br_2=\br_3$ \\ \hline
    \textbf{Case 1} & \multirow{2}{*}{200000} & \multirow{2}{*}{20000} & \multirow{2}{*}{(0,0)} & (1,1) & \multirow{2}{*}{$31^4$} & \multirow{2}{*}{0.5} & 1 & \multirow{2}{*}{$[-1,1]^3$} & (3,3,3,3) & N/A \\ \cline{1-1} \cline{5-5} \cline{8-8} \cline{10-11}
    \textbf{Case 2} &  &  &  & (2,2) &  &  & 0.5 &  & (2,2,2,2) & (2,2,2,2) \\ \hline    
    \end{tabular}
    \caption{Parameters used in Example \ref{ex2}. }
    \label{tab:ex2_parameter}
\end{table}

\noindent\textbf{Case 1: A continuous initial data.}

In this case, the initial condition is
\begin{align}
    u_0(x,y)=0.45\cos(\pi x)[\sin(\pi y)-1].
\end{align}
The numerical results are shown in Figure \ref{fig:ex2_case1}.

\begin{figure}[!htbp]
    \centering
    \includegraphics[width=0.25\textwidth]{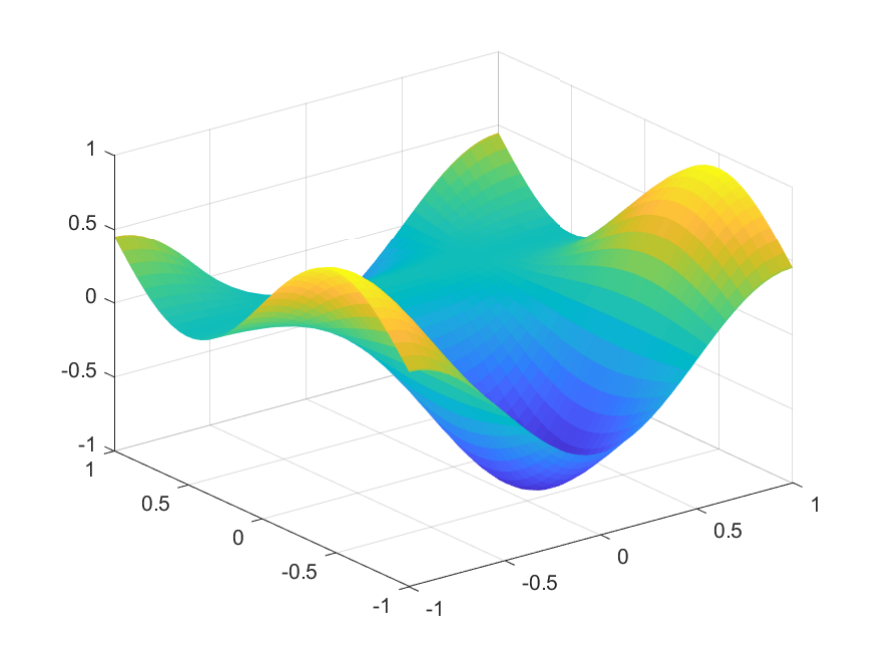}
    \includegraphics[width=0.25\textwidth]{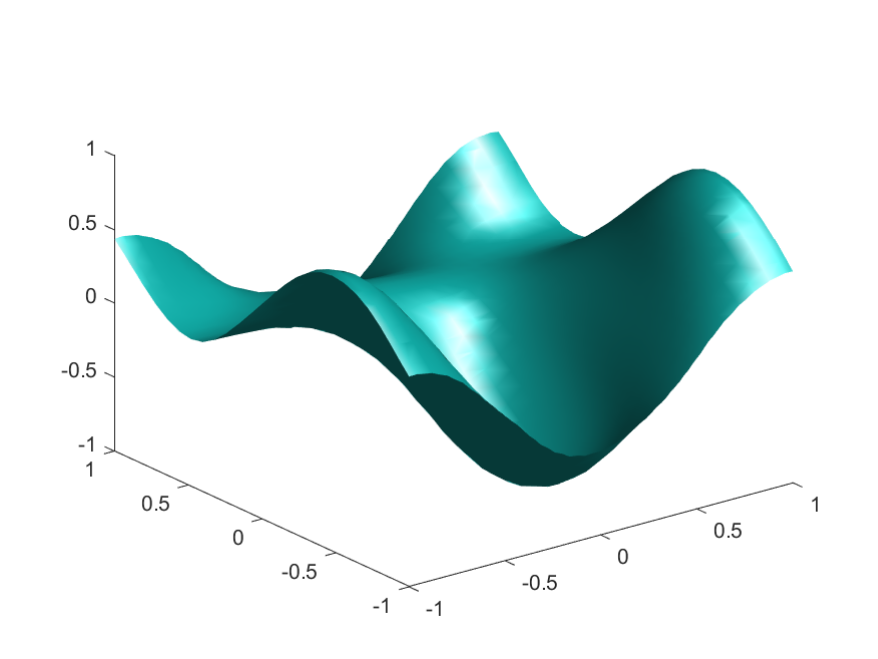}
    \includegraphics[width=0.25\textwidth]{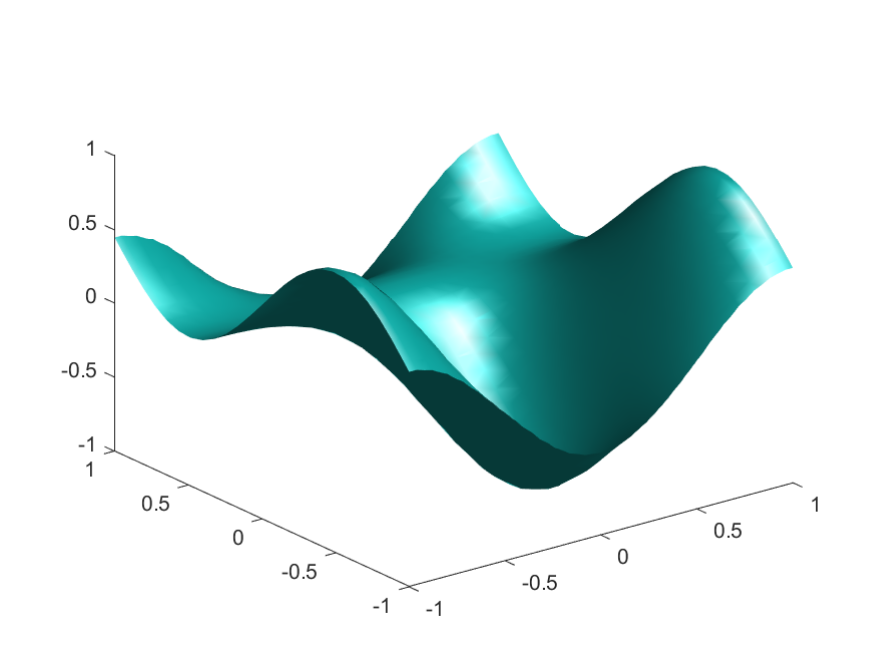}\\
    \includegraphics[width=0.25\textwidth]{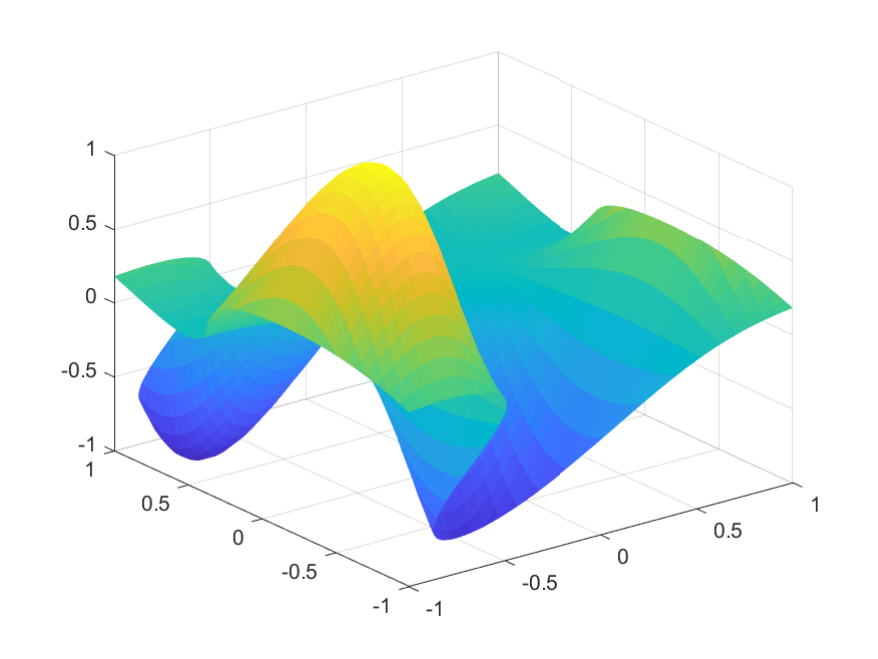}
    \includegraphics[width=0.25\textwidth]{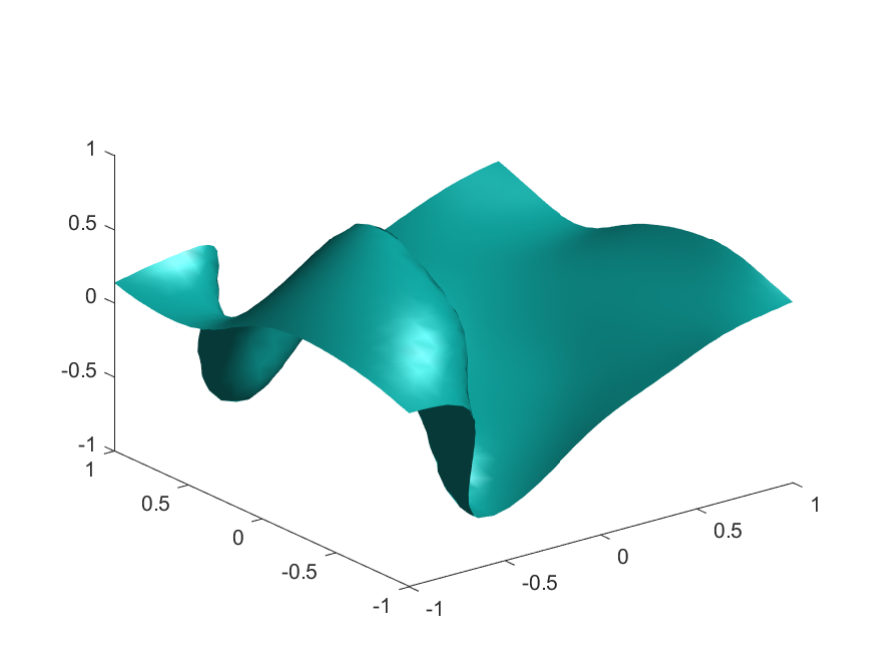}
    \includegraphics[width=0.25\textwidth]{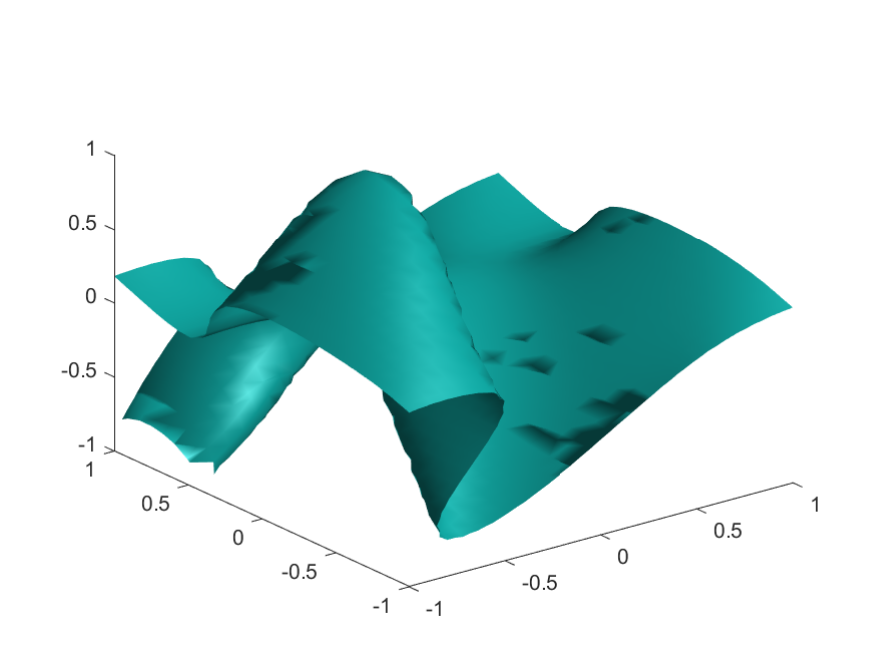}
    \caption{AG-RaNN method results for Case 1 of Example \ref{ex2}. The top and bottom rows correspond to $t=0$ and $t=1$, respectively. Columns (left to right) show the exact solutions, results by normal collocation points, and results by collocation set $\Lambda_A$ ($N_A^I=362245$, $N_A^B=11902$). Running times: $(t_1,t_2)=(12.85,45.42)$.}
    \label{fig:ex2_case1}
\end{figure}

\noindent\textbf{Case 2: A Riemann problem.}

In this case, the initial condition is
\begin{align}
    u_0(x,y)=\begin{cases}
        -1,\quad x>0,y>0,\\
        -0.2,\quad x<0,y>0,\\
        0.5,\quad x<0,y<0,\\
        0,\quad x>0,y<0.
    \end{cases}
\end{align}
The numerical results are shown in Figure \ref{fig:ex2_case2}.

\begin{figure}[!htbp]
    \centering
    \includegraphics[width=0.19\textwidth]{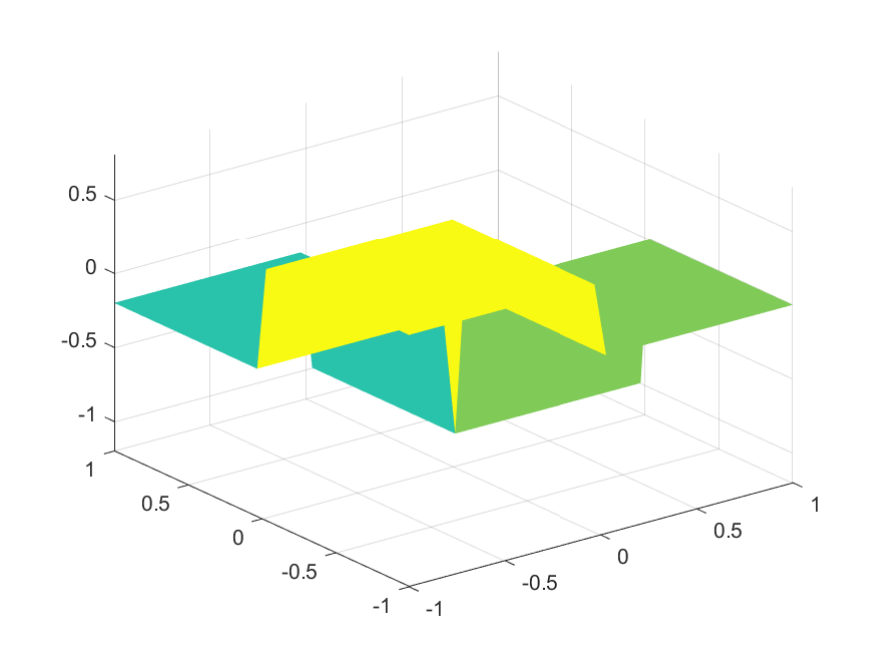}
    \includegraphics[width=0.19\textwidth]{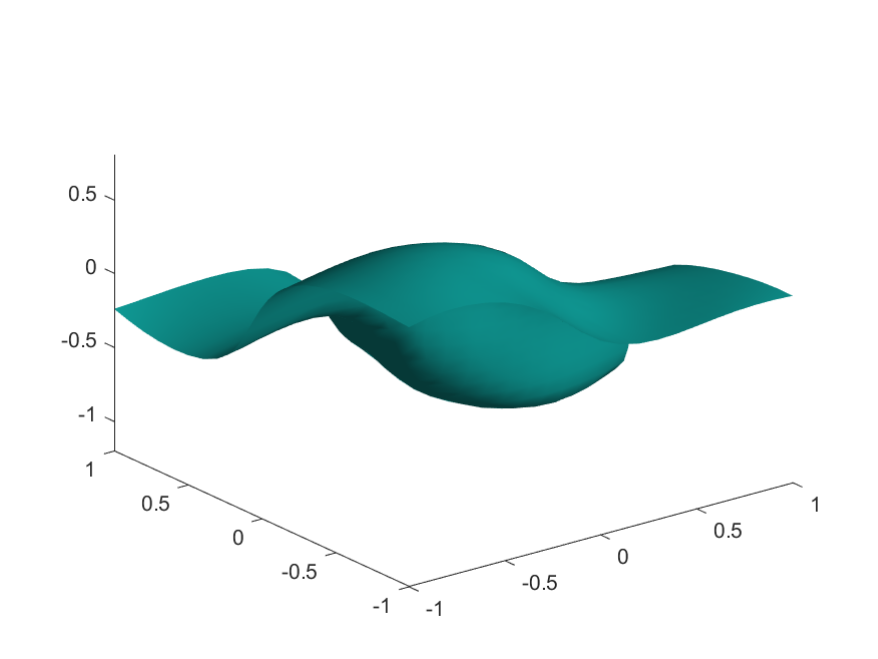}
    \includegraphics[width=0.19\textwidth]{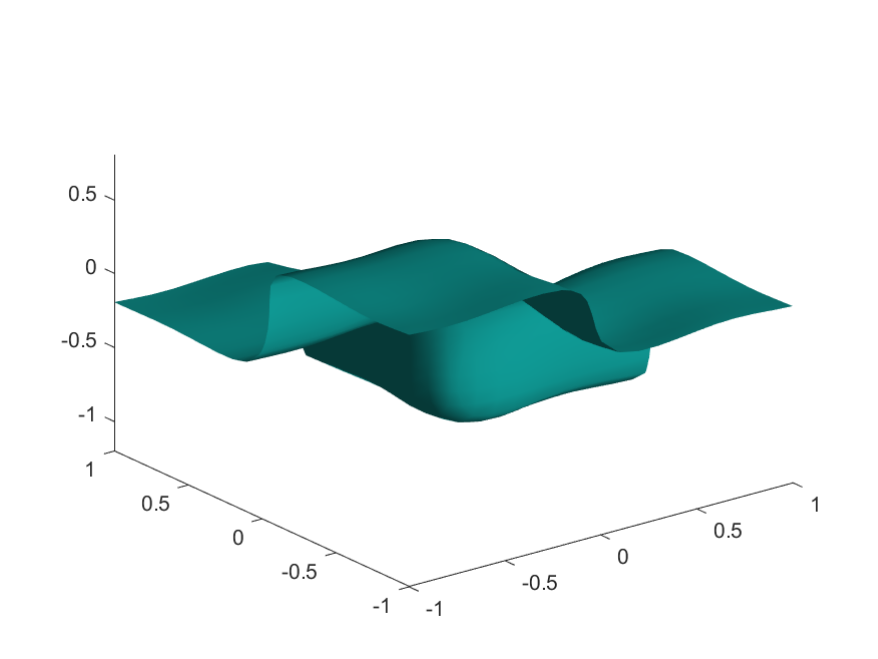}
    \includegraphics[width=0.19\textwidth]{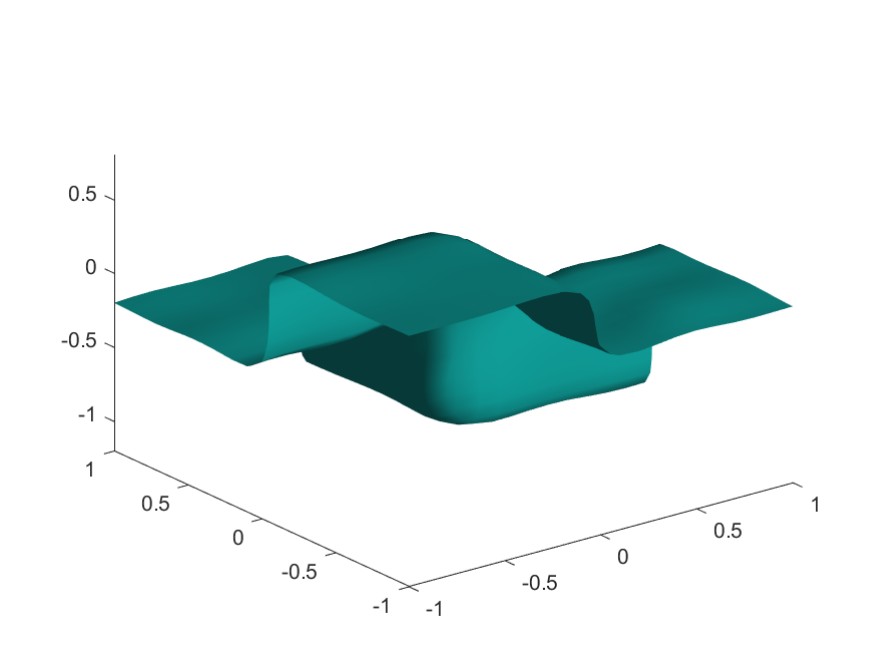}
    \includegraphics[width=0.19\textwidth]{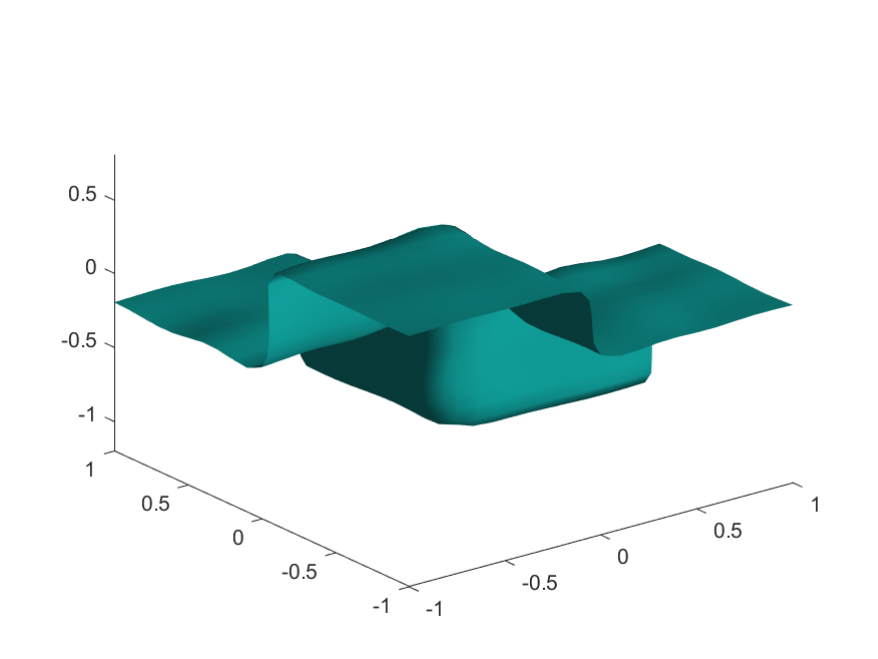}
    \includegraphics[width=0.19\textwidth]{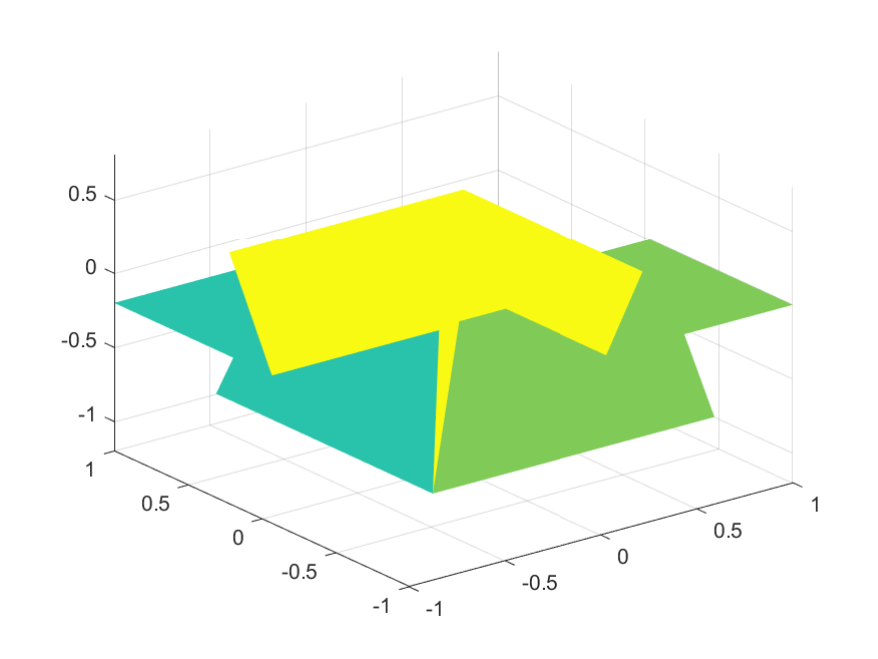}
    \includegraphics[width=0.19\textwidth]{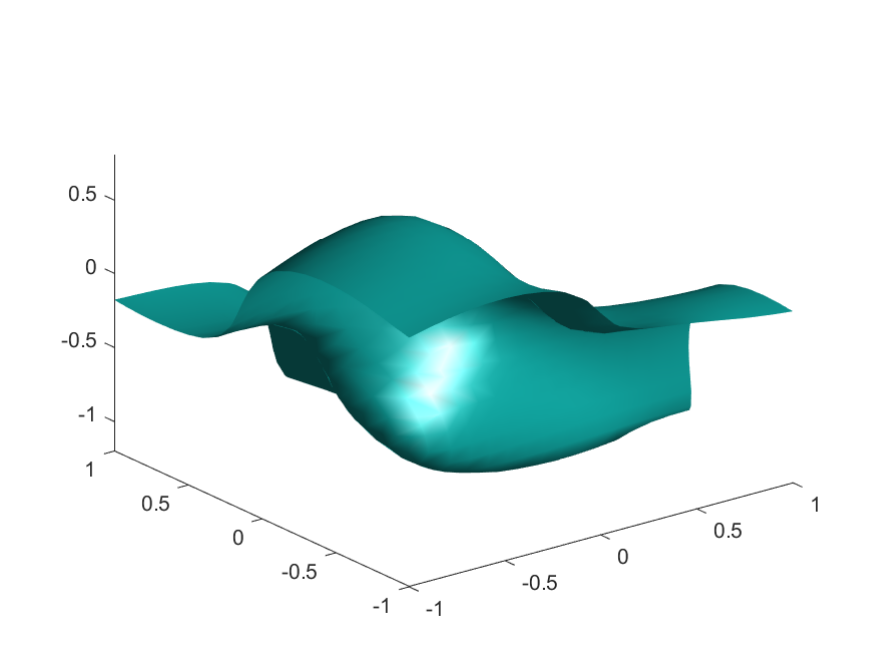}
    \includegraphics[width=0.19\textwidth]{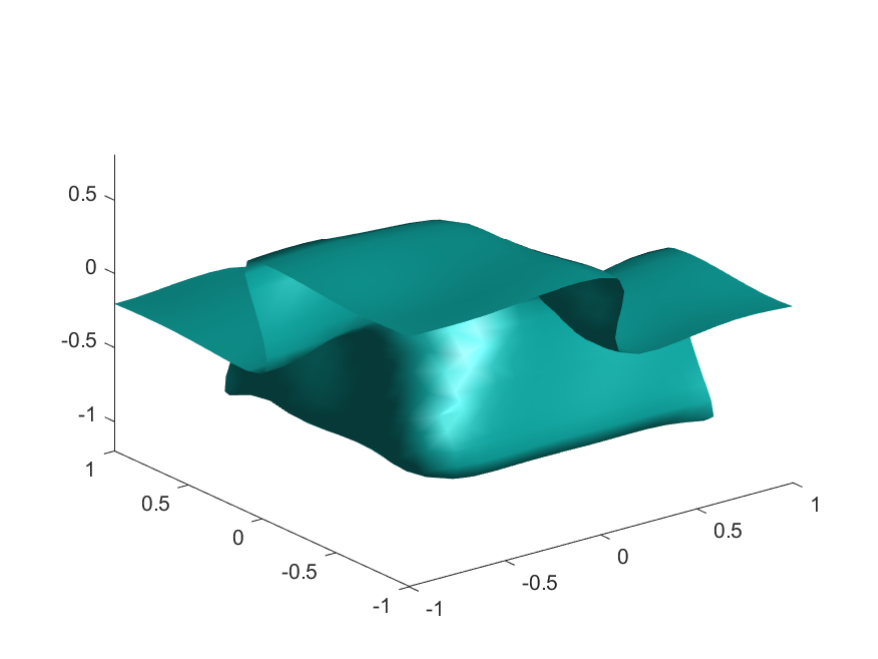}
    \includegraphics[width=0.19\textwidth]{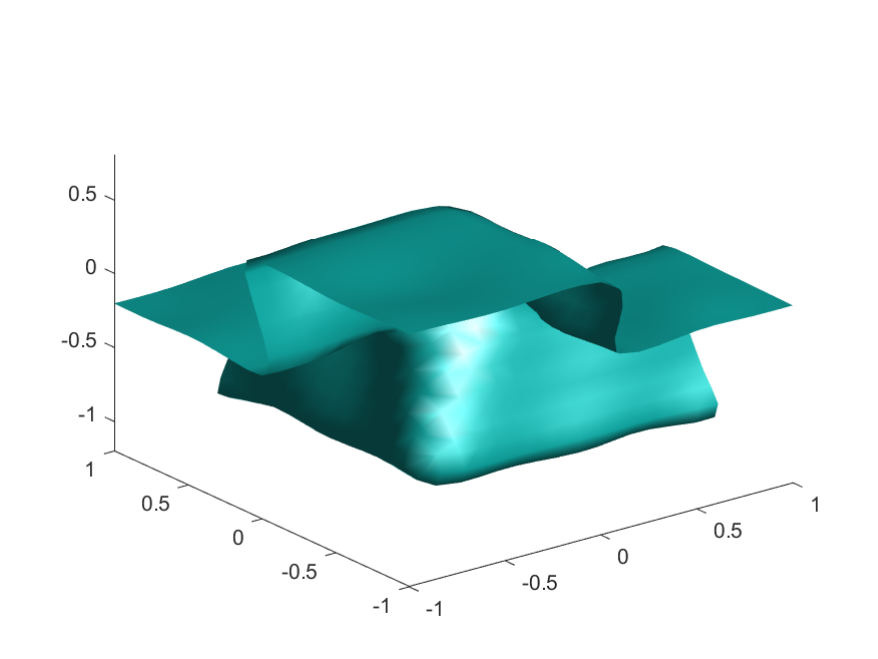}
    \includegraphics[width=0.19\textwidth]{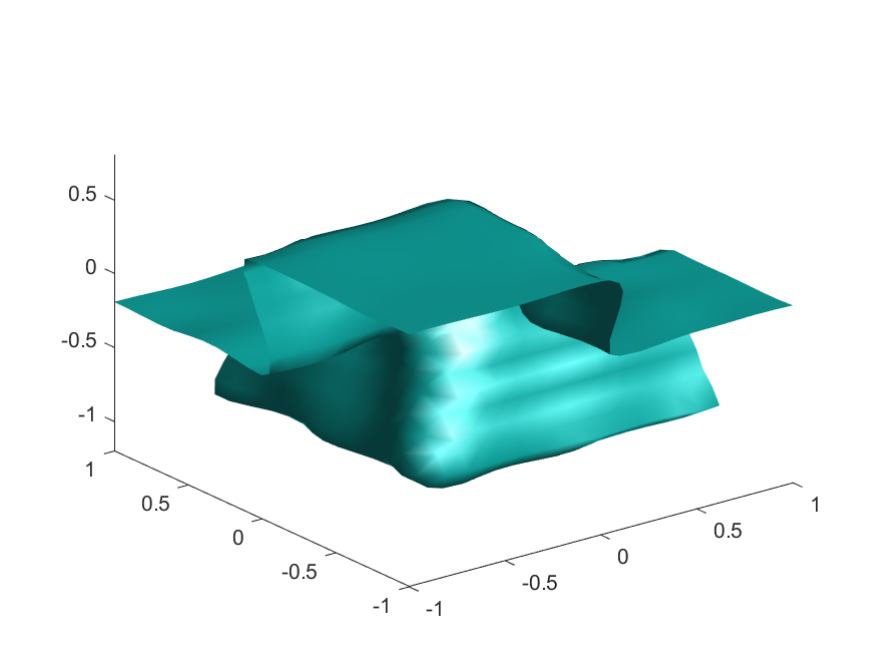}
    \caption{AG-RaNN method results for Case 2 of Example \ref{ex2}. The top and bottom rows correspond to $t=0$ and $t=0.5$, respectively. Columns 1-3 show the exact solutions, results by normal collocation points, results by collocation set $\Lambda_A$ ($N_A^I=469858$, $N_A^B=13757$), and the right two columns are obtained by the layer growth strategy.}
    \label{fig:ex2_case2}
\end{figure}

In this example, it is straightforward to calculate the numerical solution with adaptive collocation points when the initial condition is continuous as shown in Case 1. However, for the Riemann problem in Case 2, adaptive collocation based on a coarse solution may fail to localize the discontinuity accurately. We therefore employ the layer-growth strategy to enhance the expressive power of the approximation. As shown in Figure \ref{fig:ex2_case2}, increasing the number of layers improves resolution near the discontinuity, although the shocks remain mildly smeared.

\begin{example}[1D Hamilton--Jacobi Equation] \label{ex3}
We consider the following 1D Hamilton--Jacobi equation:
\begin{align}
    \partial_t S + \frac{1}{2}|\partial_x S|^2 + V &= 0, \quad(t,x)\in(0,\infty)\times\Real,\label{ex3:eq1}\\
    S(0,x) &= S_0(x),\quad x\in\Real,
\end{align}
where $V(x)$ is the potential. We explore four cases, varying $V(x)$ and $S_0(x)$. 
\end{example}

In this example, we compute only the gradient field $\partial_x S$ via the gradient-based level-set formulation \eqref{eq:HJ5}. The parameters for each case are listed in Table \ref{tab:ex3_parameter}.
\begin{table}[!htbp]
    \centering
    \begin{tabular}{|c|c|c|c|c|c|c|c|c|c|c|c|c|c|c|c|}
    \hline
     & $N^I$ & $N^B$ & E & Var & $N$ & $\varepsilon_A$ & T & $\Omega$ & $\br_1$ \\ \hline
    \textbf{Case 1} & \multirow{4}{*}{20000} & \multirow{4}{*}{5000} & \multirow{4}{*}{(0,0)} & \multirow{4}{*}{(2,2)} & \multirow{4}{*}{$51^3$} & 1 & 10 & $[-2,2]\times[-2,2]$ & \multirow{4}{*}{(2,2,2)} \\ \cline{1-1} \cline{7-9}
    \textbf{Case 2} &  &  &  &  &  & 1 & 10 & $[-2,2]\times[-2,2]$ &  \\ \cline{1-1} \cline{7-9}
    \textbf{Case 3} &  &  &  &  &  & 0.6 & 2 & $[-2,2]\times[-2,2]$ &  \\ \cline{1-1} \cline{7-9}
    \textbf{Case 4} &  &  &  &  &  & 0.6 & $2\pi$ & $[-3,3]\times[-3,3]$ &  \\ \hline    
    \end{tabular}
    \caption{Parameters used in Example \ref{ex3}. }
    \label{tab:ex3_parameter}
\end{table}

\noindent\textbf{Case 1: No caustic.}

In this case, the exact solution is
\begin{align}
    S(t,x)=\frac{x^2}{2(t+1)},\quad \partial_x S(t,x)=\frac{x}{t+1}
\end{align}
and the potential is $V(x)=0$. The numerical results are shown in Figure \ref{fig:ex3_case1}.

\begin{figure}[!htbp]
    \centering
    \includegraphics[width=0.20\textwidth]{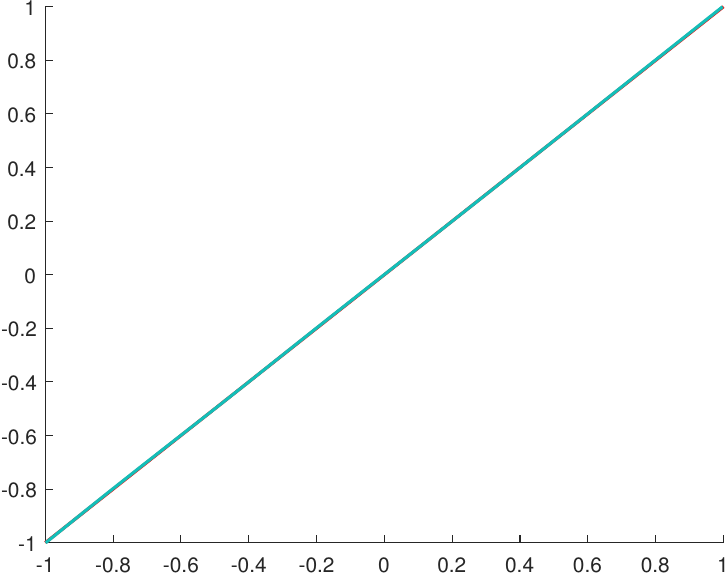}
    \includegraphics[width=0.20\textwidth]{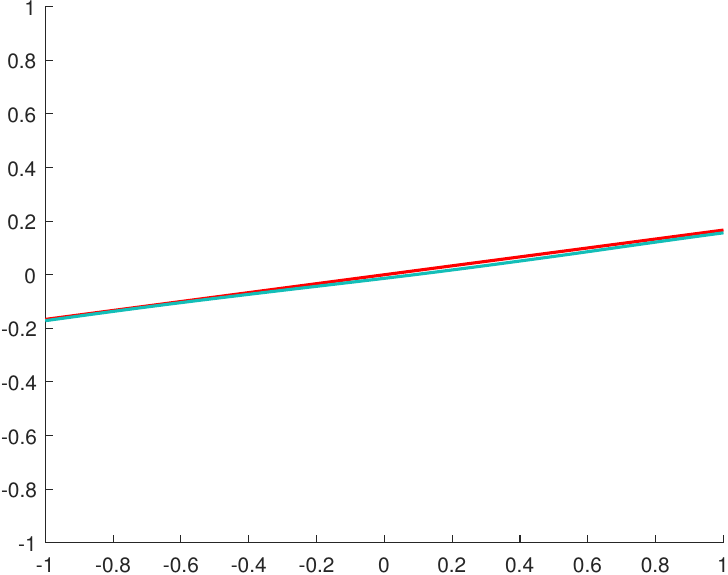}
    \includegraphics[width=0.20\textwidth]{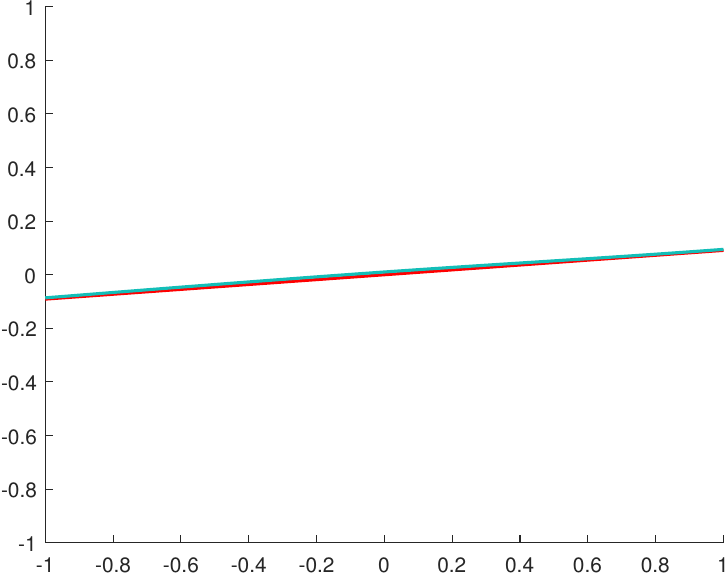}
    \includegraphics[width=0.20\textwidth]{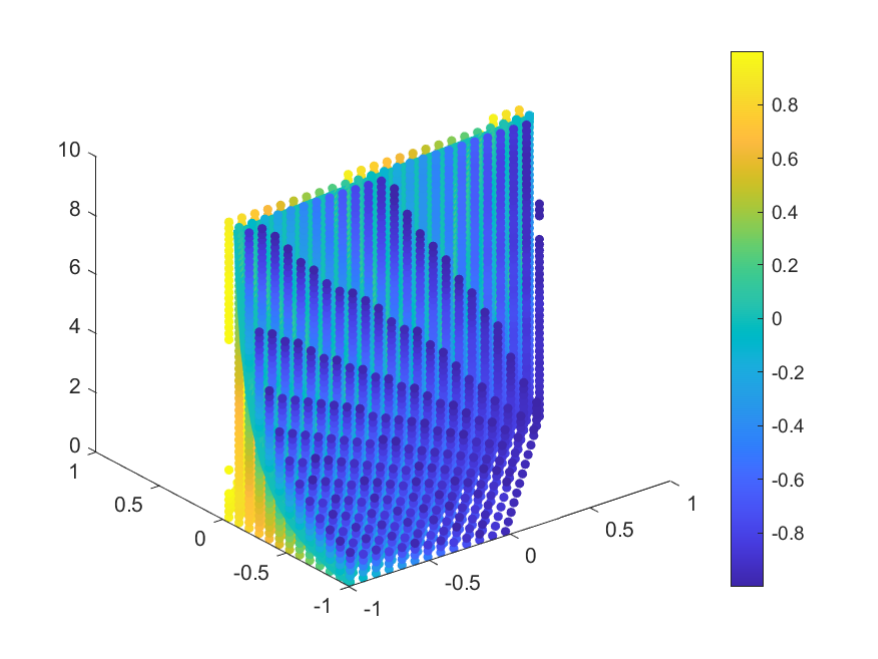}
    \includegraphics[width=0.20\textwidth]{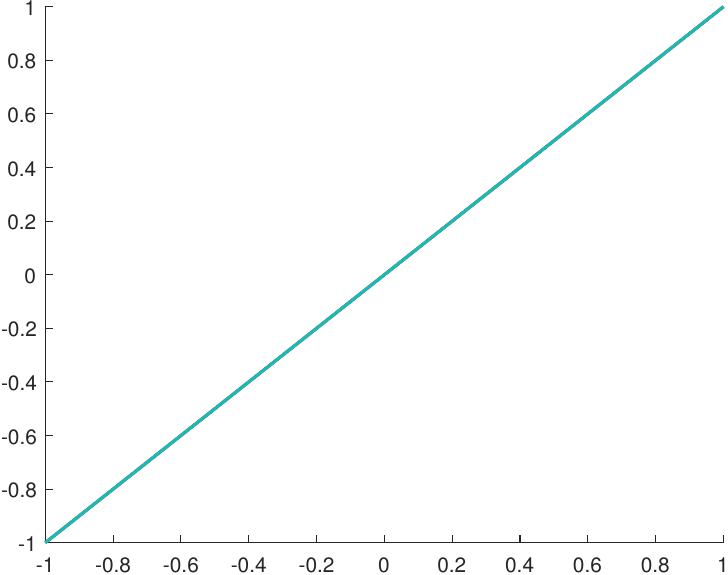}
    \includegraphics[width=0.20\textwidth]{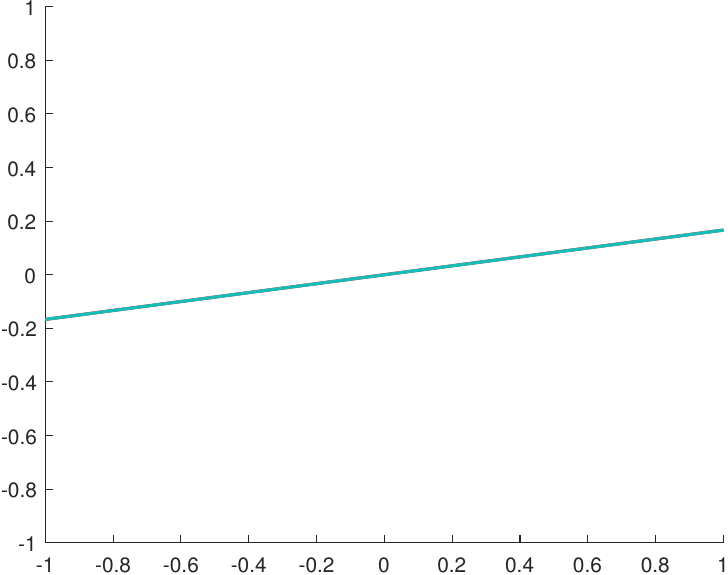}
    \includegraphics[width=0.20\textwidth]{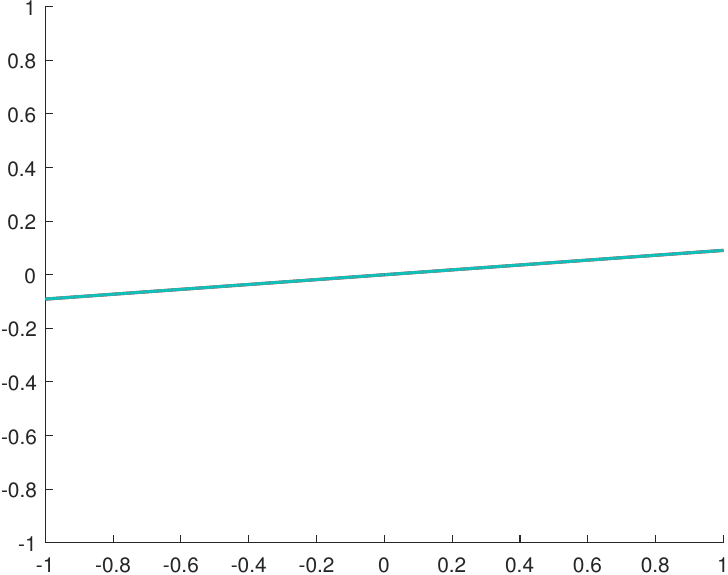}
    \includegraphics[width=0.20\textwidth]{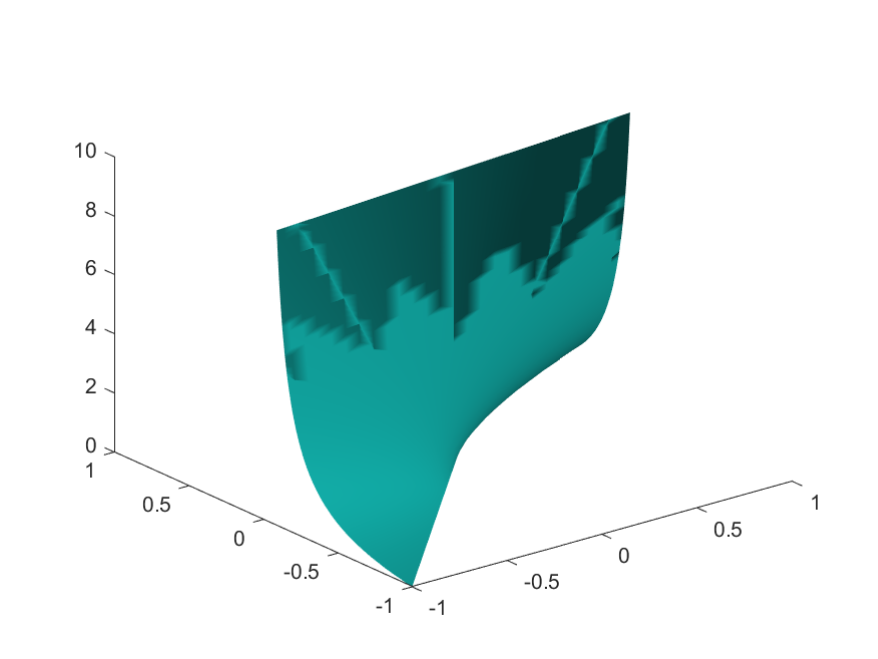}
    \caption{AG-RaNN method results for Case 1 of Example \ref{ex3}. Columns 1-3 correspond to $t=0,5,10$, and column 4 shows the zero level set. The top and bottom rows are obtained by normal collocation points and collocation set $\Lambda_A$ ($N_A^I=15933$, $N_A^B=1119$), respectively. Running times: $(t_1,t_2)=(0.95,0.76)$.}

    \label{fig:ex3_case1}
\end{figure}

\noindent\textbf{Case 2: Focusing at a point.}

In this case, the exact solution is
\begin{align}
    S(t,x)=\frac{x^2}{2(t-1)},\quad \partial_x S(t,x)=\frac{x}{t-1}
\end{align}
and the potential is $V(x)=0$. The numerical results are shown in Figure \ref{fig:ex3_case2}.

\begin{figure}[!htbp]
    \centering
    \includegraphics[width=0.20\textwidth]{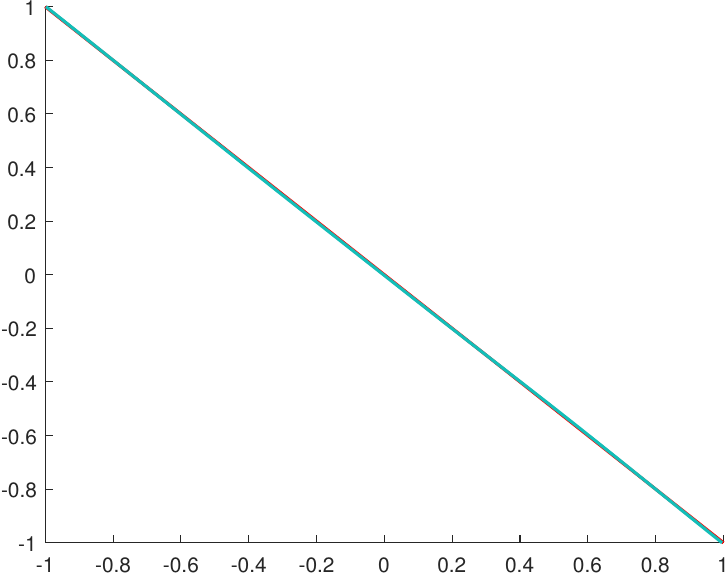}
    \includegraphics[width=0.20\textwidth]{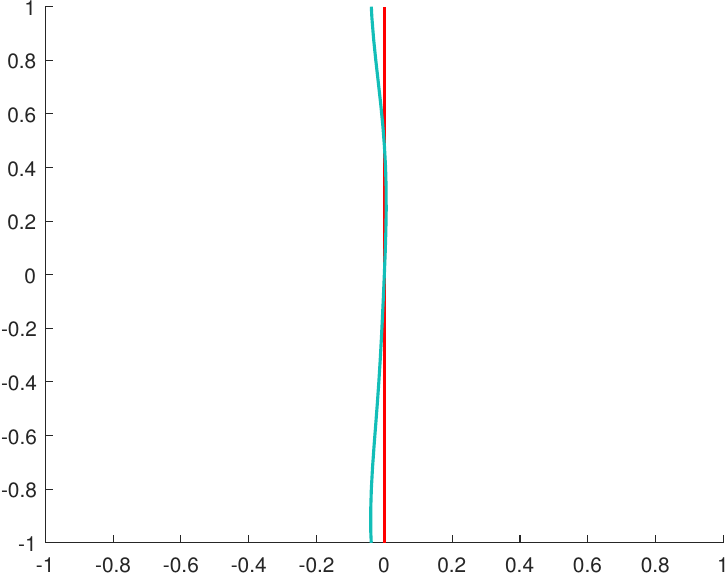}
    \includegraphics[width=0.20\textwidth]{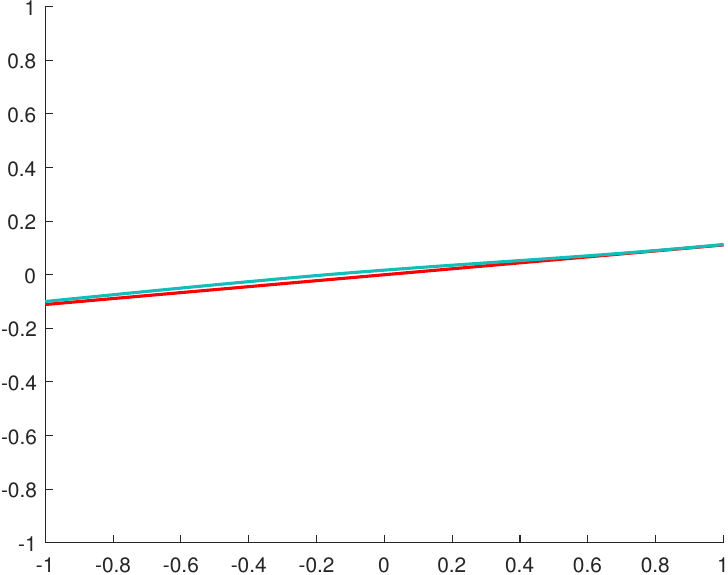}
    \includegraphics[width=0.20\textwidth]{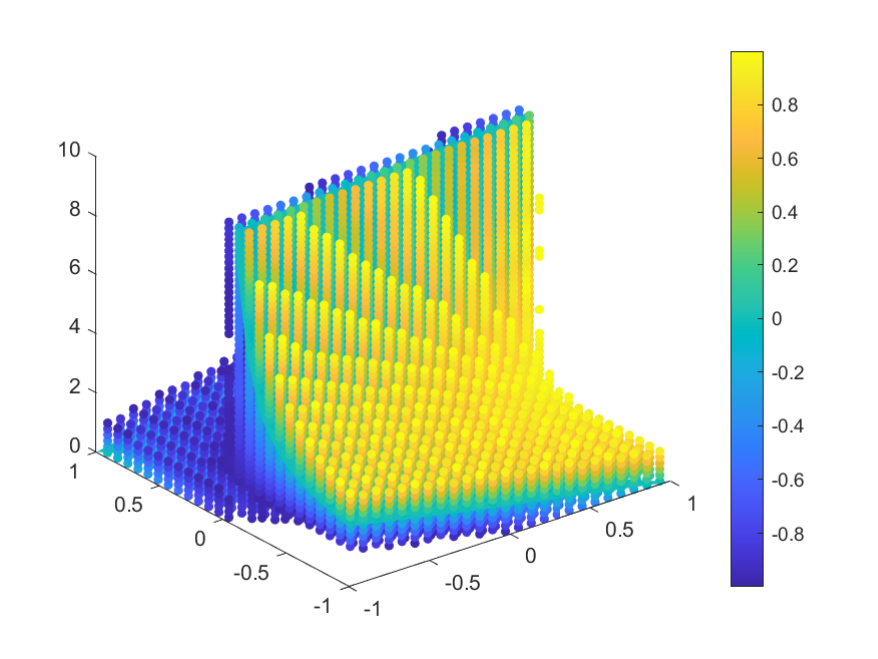}
    \includegraphics[width=0.20\textwidth]{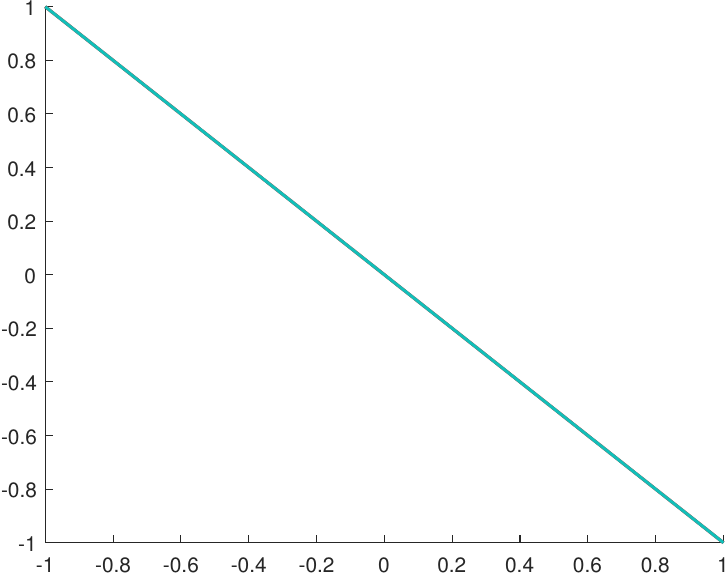}
    \includegraphics[width=0.20\textwidth]{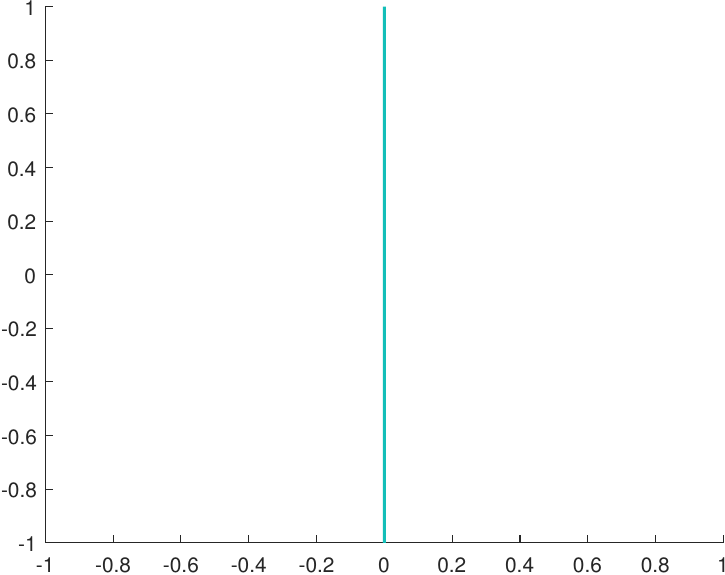}
    \includegraphics[width=0.20\textwidth]{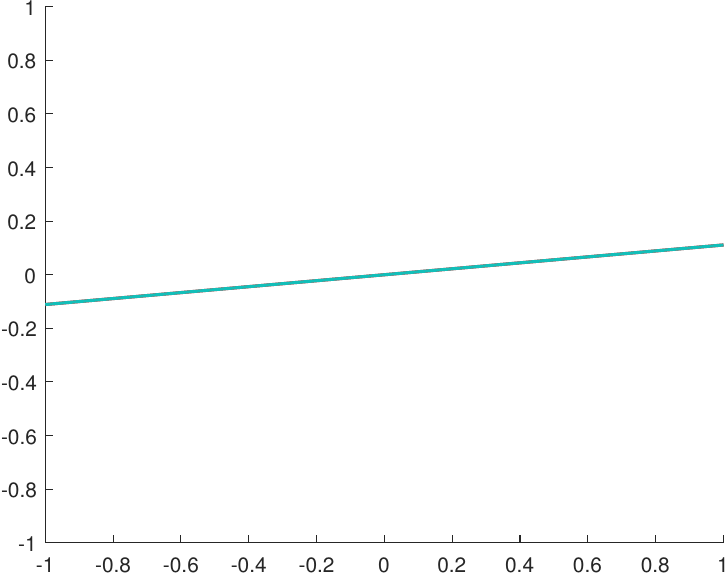}
    \includegraphics[width=0.20\textwidth]{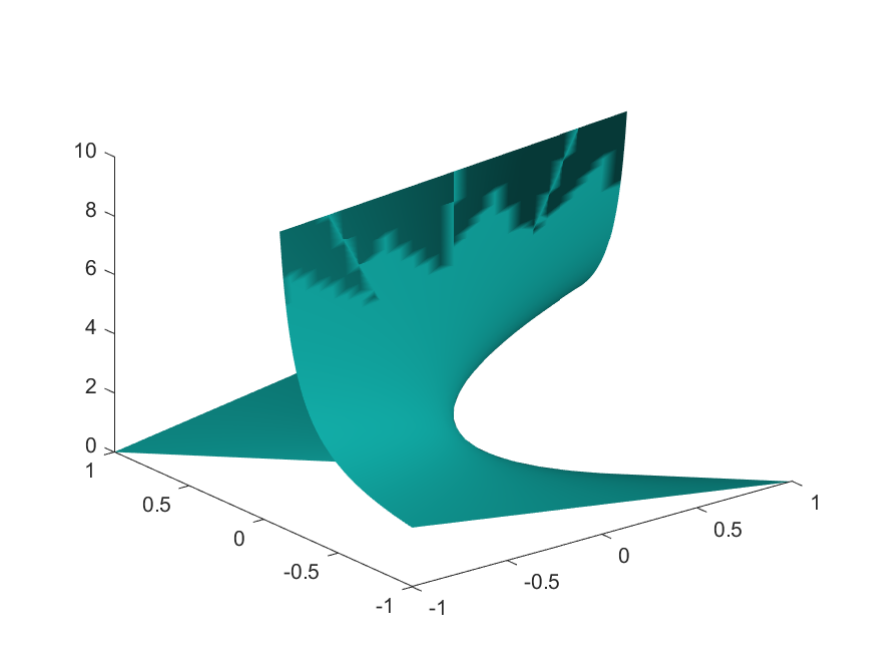}
    \caption{AG-RaNN method results for Case 2 of Example \ref{ex3}. Columns 1-3 correspond to $t=0,2,10$, and column 4 shows the zero level set. The top and bottom rows are obtained by normal collocation points and collocation set $\Lambda_A$ ($N_A^I=26942$, $N_A^B=1119$), respectively. Running times: $(t_1,t_2)=(0.92,1.16)$.}
    \label{fig:ex3_case2}
\end{figure}

\noindent\textbf{Case 3: Caustic.}

In this case, the initial condition is
\begin{align}
    S_0(x)=-\ln(\cosh(x)),\quad \partial_x S_0(x)=-\tanh(x)
\end{align}
and the potential is $V(x)=0$. The numerical results are shown in Figure \ref{fig:ex3_case3}.

\begin{figure}[!htbp]
    \centering
    \includegraphics[width=0.20\textwidth]{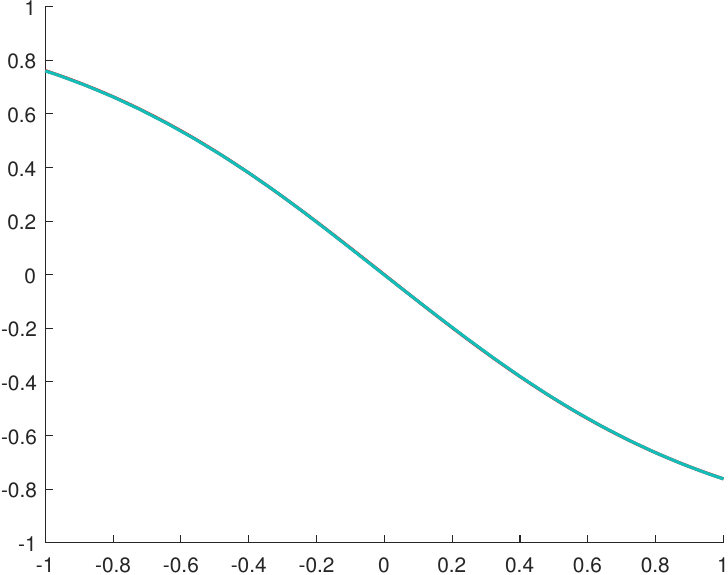}
    \includegraphics[width=0.20\textwidth]{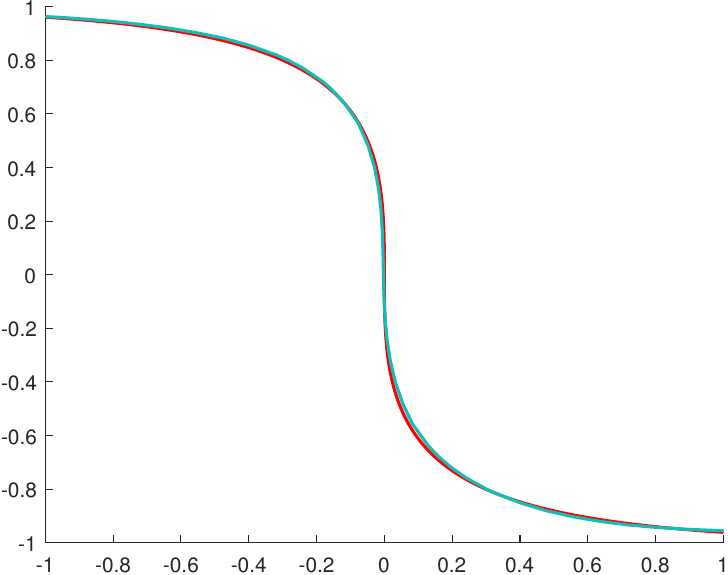}
    \includegraphics[width=0.20\textwidth]{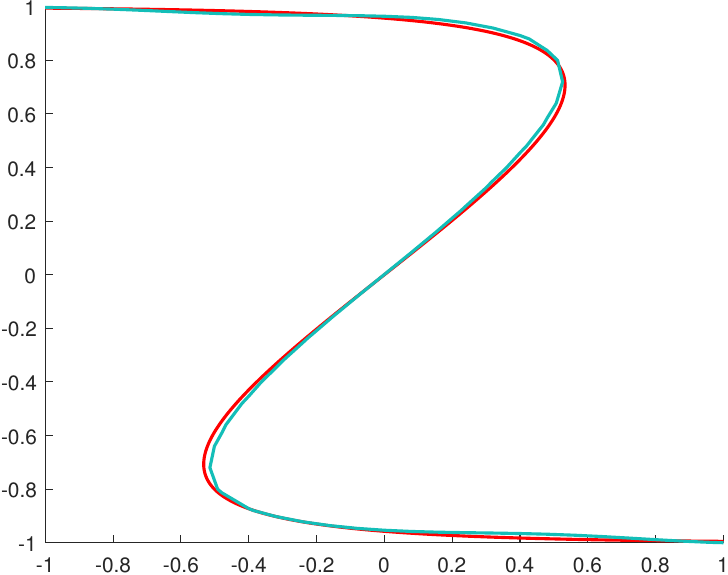}
    \includegraphics[width=0.20\textwidth]{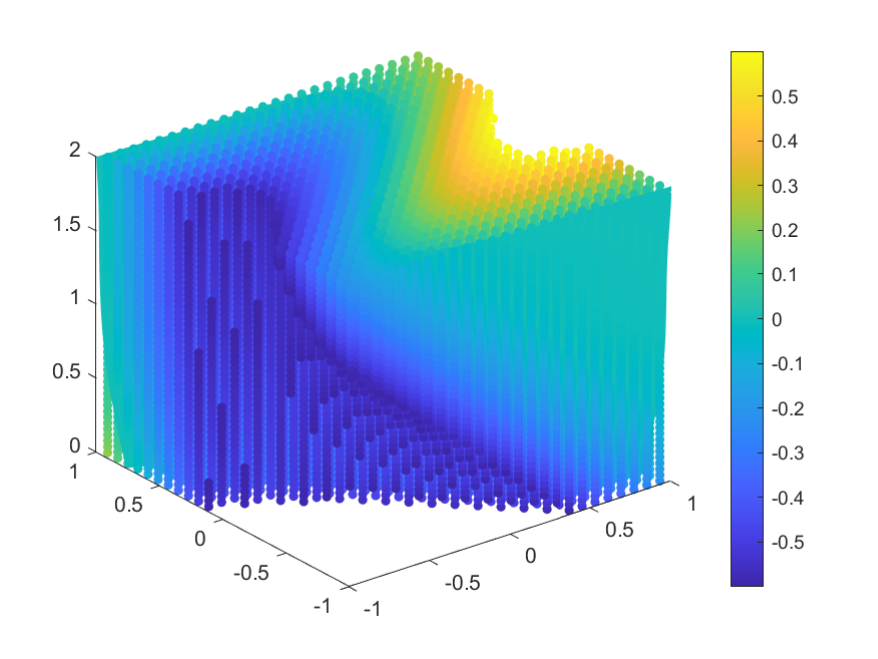}
    \includegraphics[width=0.20\textwidth]{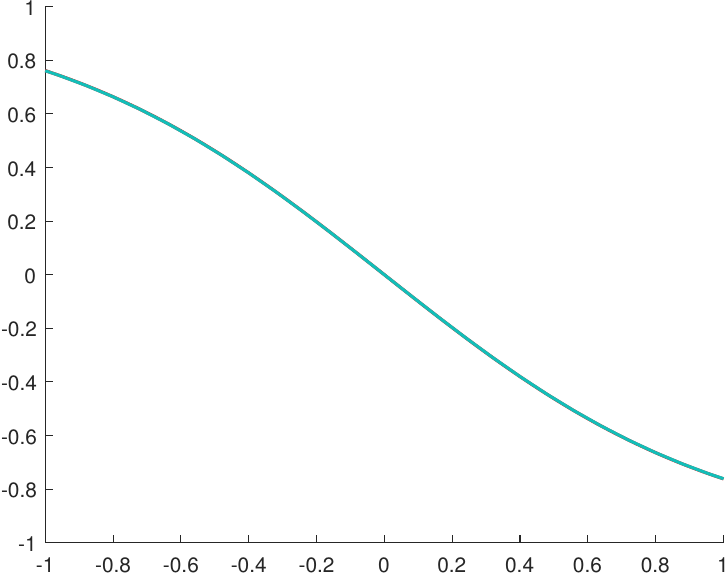}
    \includegraphics[width=0.20\textwidth]{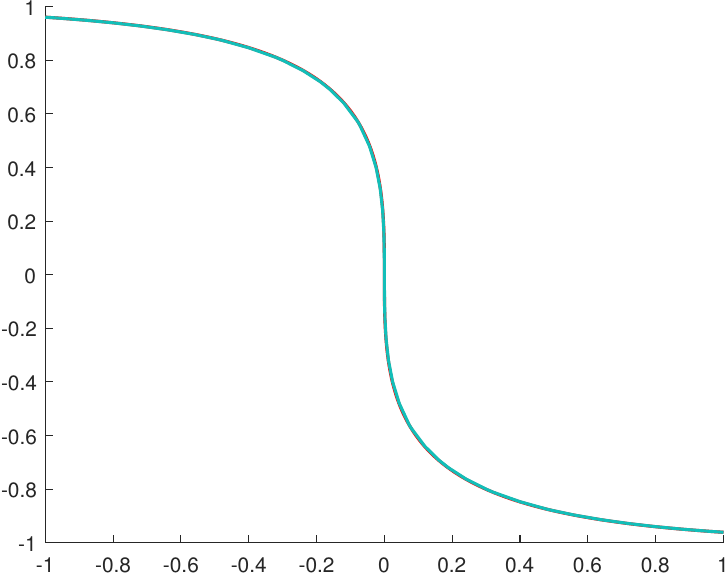}
    \includegraphics[width=0.20\textwidth]{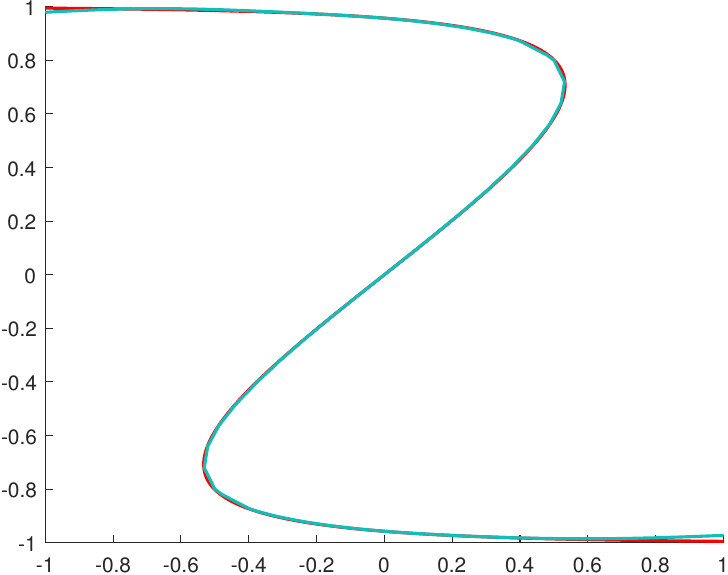}
    \includegraphics[width=0.20\textwidth]{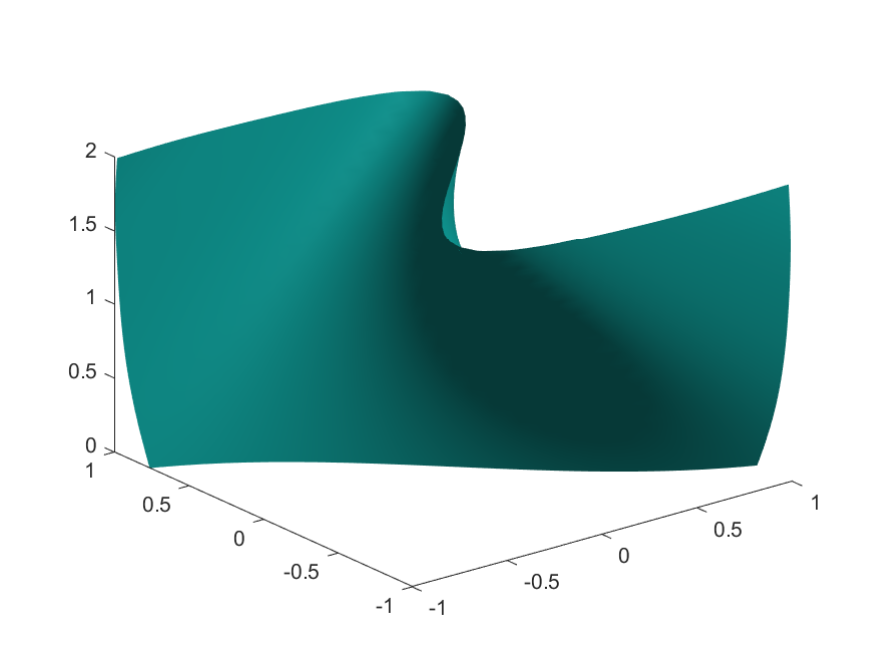}
    \caption{AG-RaNN method results for Case 3 of Example \ref{ex3}. Columns 1-3 correspond to $t=0,1,2$, and column 4 shows the zero level set. The top and bottom rows are obtained by normal collocation points and collocation set $\Lambda_A$ ($N_A^I=57777$, $N_A^B=765$), respectively. Running times: $(t_1,t_2)=(0.99,2.10)$.}
    \label{fig:ex3_case3}
\end{figure}

\noindent\textbf{Case 4: The harmonic oscillator.}

In this case, the exact solution is
\begin{align}
    S(t,x)=-\frac{1}{2}(x^2+1)\tan(t)+\frac{x}{\cos(t)},\quad \partial_x S(t,x)=-x\tan(t)+\frac{1}{\cos(t)}
\end{align}
and the potential is $V(x)=x^2/2$. The numerical results are shown in Figure \ref{fig:ex3_case4}.

\begin{figure}[!htbp]
    \centering
    \includegraphics[width=0.20\textwidth]{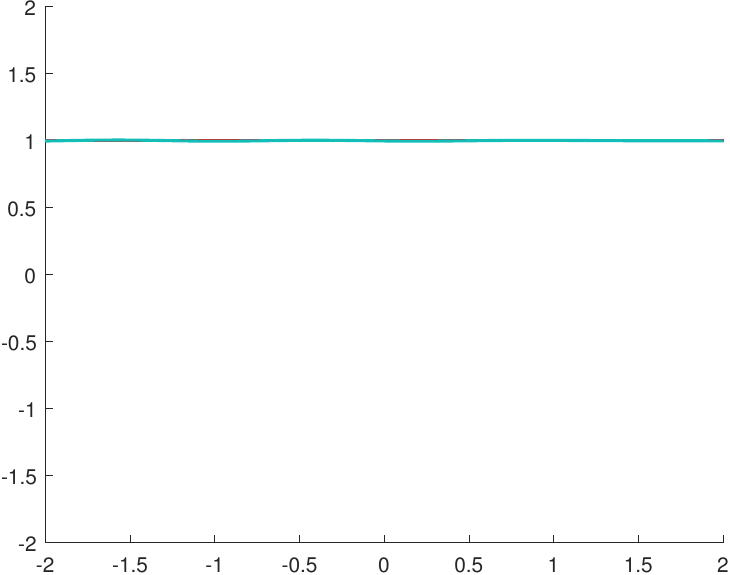}
    \includegraphics[width=0.20\textwidth]{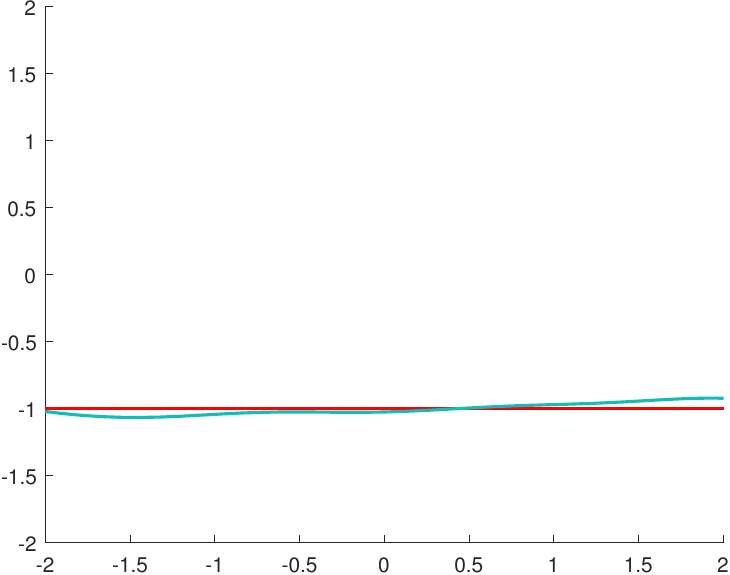}
    \includegraphics[width=0.20\textwidth]{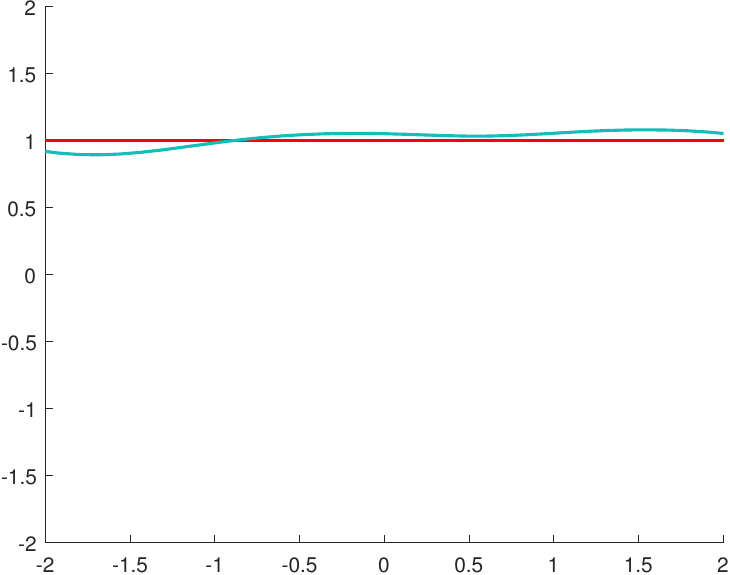}
    \includegraphics[width=0.20\textwidth]{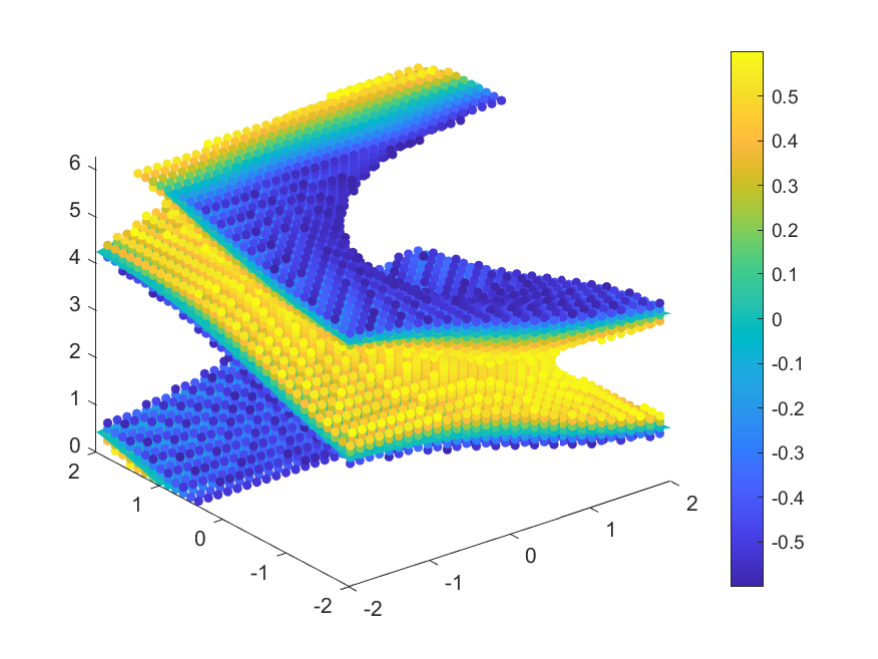}
    \includegraphics[width=0.20\textwidth]{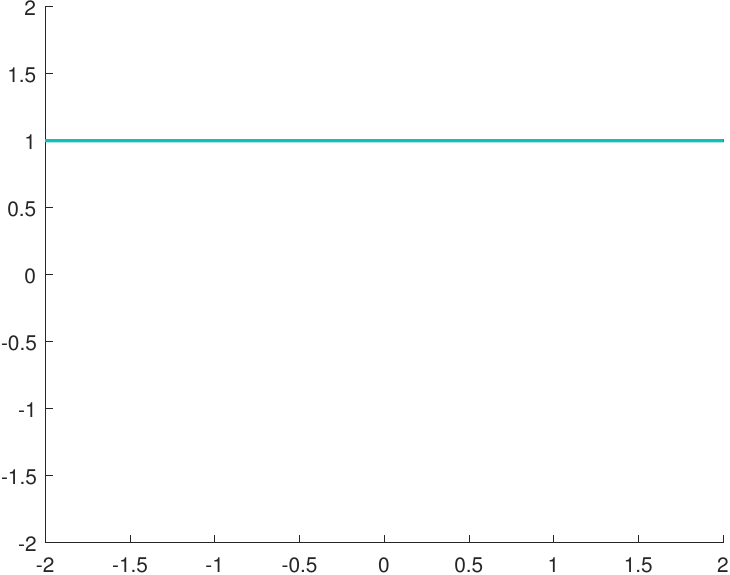}
    \includegraphics[width=0.20\textwidth]{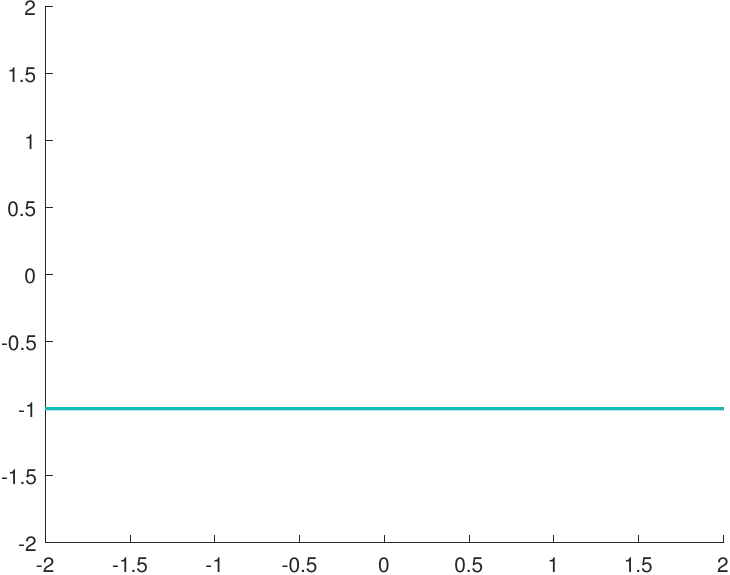}
    \includegraphics[width=0.20\textwidth]{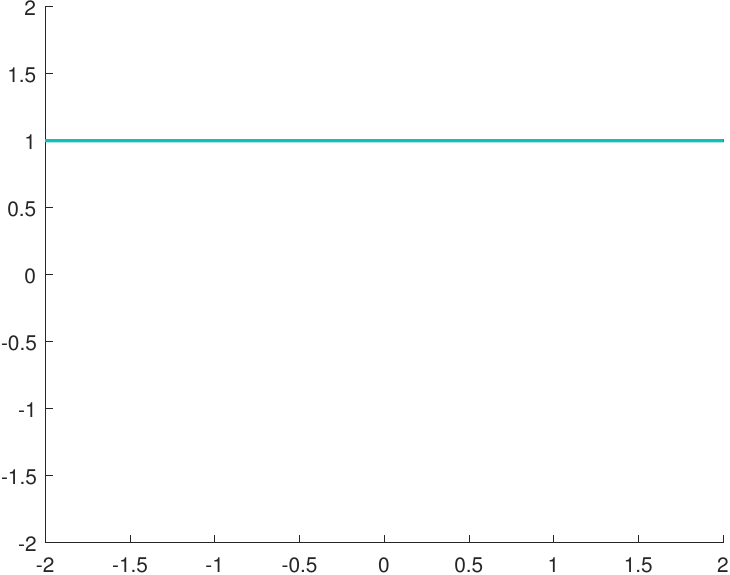}
    \includegraphics[width=0.20\textwidth]{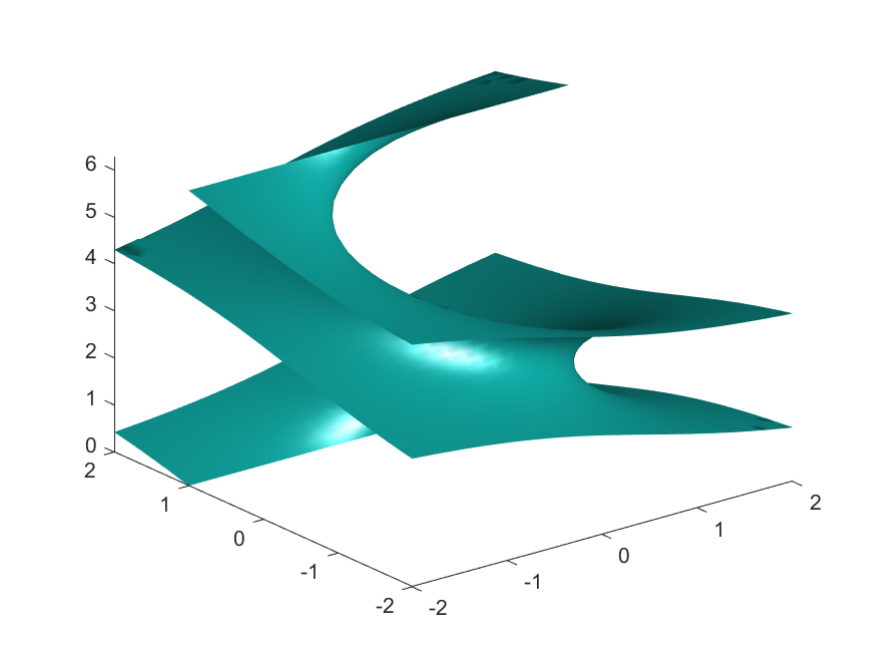}
    \caption{AG-RaNN method results for Case 4 of Example \ref{ex3}. Columns 1-3 correspond to $t=0,\pi,2\pi$, and column 4 shows the zero level set. The top and bottom rows are obtained by normal collocation points and collocation set $\Lambda_A$ ($N_A^I=28170$, $N_A^B=510$), respectively. Running times: $(t_1,t_2)=(0.96,1.18)$.}
    \label{fig:ex3_case4}
\end{figure}

The original equation is a two-dimensional problem of time and space $S(t,x)$, which is transformed into a three-dimensional problem $\phi(t,x,p)$ by introducing the level-set function. But we can only find the gradient of $S$. In all cases, the initial values have good smoothness. It can be seen that RaNN method also has good performance for this problem.

\begin{example}[1D Hamilton--Jacobi Equation: recovering both $S$ and $\partial_x S$] \label{ex3-1}
In this example, we consider the same setting as Example \ref{ex3}, and solve $S$ as well as the gradient of $S$ based on level-set equation \eqref{eq:HJ2}, all parameters for each case are presented in Table \ref{tab:ex3-1_parameter}.
\end{example}
\begin{table}[!htbp]
    \centering
    \begin{tabular}{|c|c|c|c|c|c|c|c|c|c|c|c|c|c|c|c|}
    \hline
     & $N^I$ & $N^B$ & E & Var & $N$ & $\varepsilon_A$ & T & $\Omega$ & $\br_1$ \\ \hline
    \textbf{Case 1} & \multirow{3}{*}{100000} & \multirow{3}{*}{5000} & \multirow{4}{*}{(0,0,0)} & \multirow{3}{*}{(1.5,1.5,1.5)} & \multirow{4}{*}{$21^4$} & \multirow{3}{*}{0.5} & \multirow{3}{*}{2} & \multirow{3}{*}{$[-1.5,1.5]^3$} & \multirow{3}{*}{(3,3,3,3)} \\ \cline{1-1}
    \textbf{Case 2} &  &  &  &  &  &  &  &  &  \\ \cline{1-1}
    \textbf{Case 3} &  &  &  &  &  &  &  &  &  \\ \cline{1-3} \cline{5-5} \cline{7-10}
    \textbf{Case 4} & 200000 & 20000 &  & (3,3,3) &  & 1 & $2\pi$ & $[-2.5,2.5]^3$ & (0.3,0.3,0.3,0.3) \\ \hline    
    \end{tabular}
    \caption{Parameters used in Example \ref{ex3-1}. }
    \label{tab:ex3-1_parameter}
\end{table}

\noindent\textbf{Case 1: No caustic.}

In this case, the exact solution is
\begin{align}
    S(t,x)=\frac{x^2}{2(t+1)}
\end{align}
and the potential is $V(x)=0$. The numerical results are shown in Figure \ref{fig:ex3-1_case1_HJ}.

\begin{figure}[!htbp]
    \centering
    \includegraphics[width=0.16\textwidth]{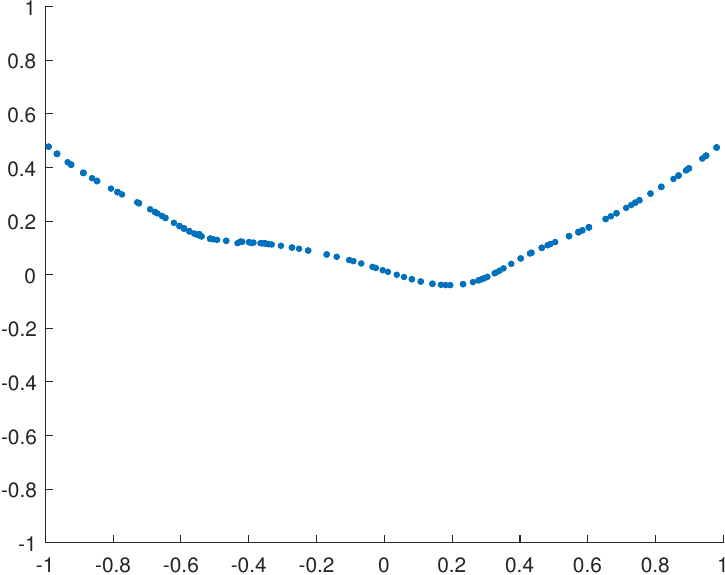}
    \includegraphics[width=0.16\textwidth]{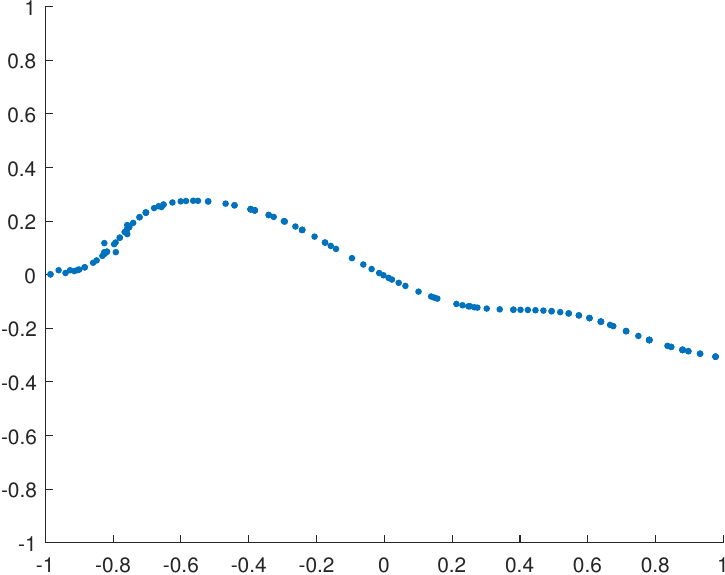}
    \includegraphics[width=0.16\textwidth]{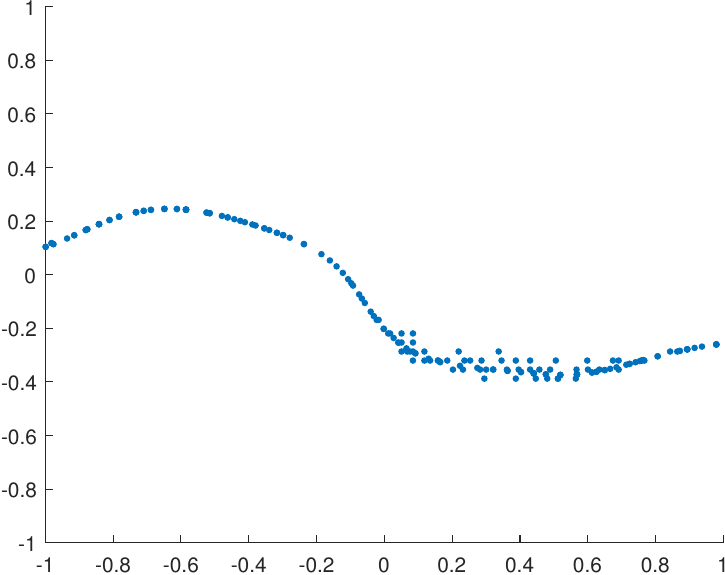}
    \includegraphics[width=0.16\textwidth]{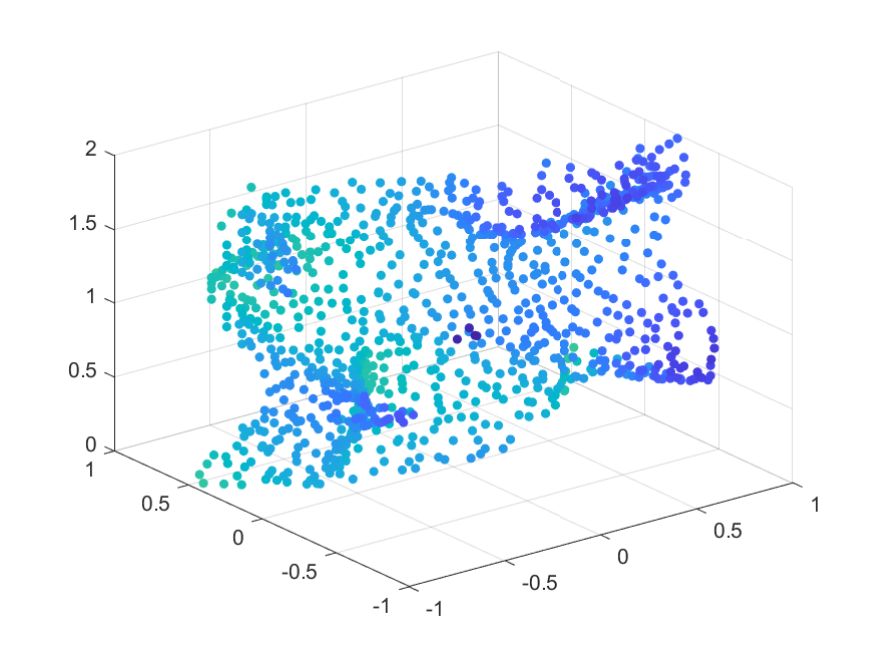}\\
    \includegraphics[width=0.16\textwidth]{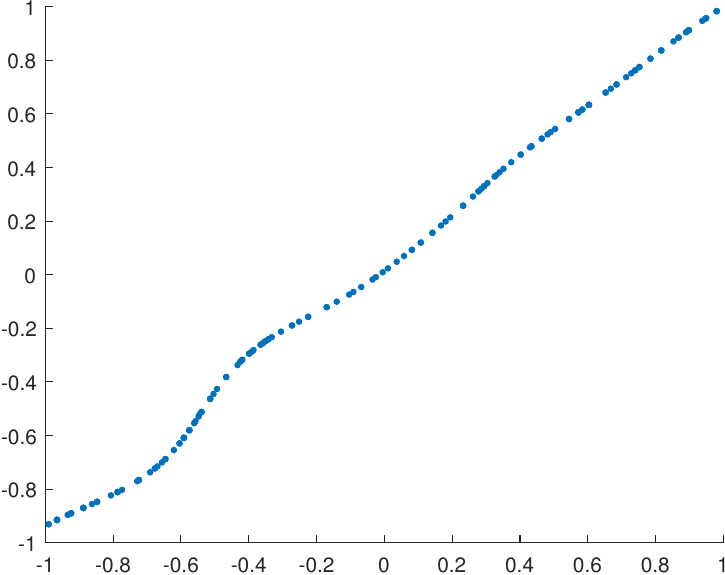}
    \includegraphics[width=0.16\textwidth]{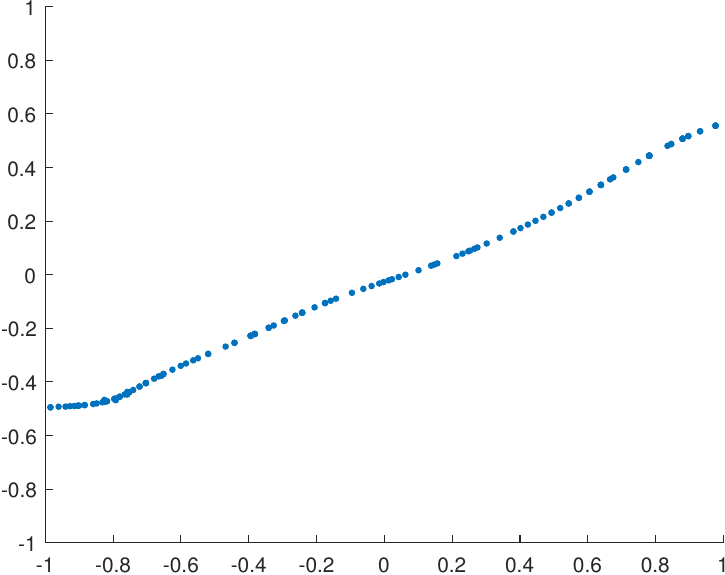}
    \includegraphics[width=0.16\textwidth]{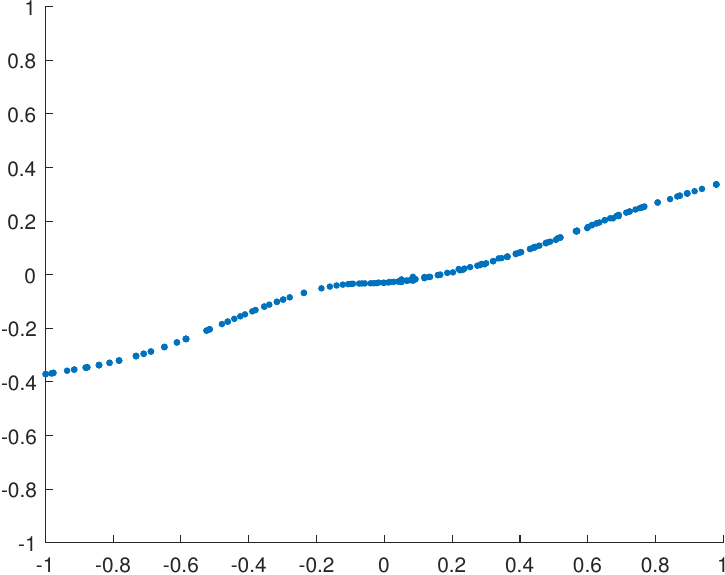}
    \includegraphics[width=0.16\textwidth]{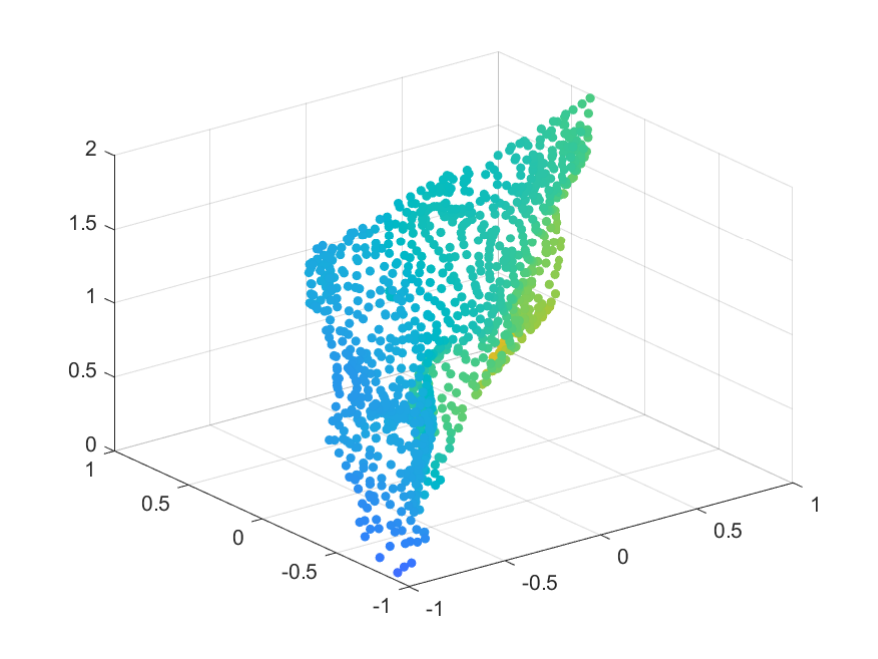}\\
    \includegraphics[width=0.16\textwidth]{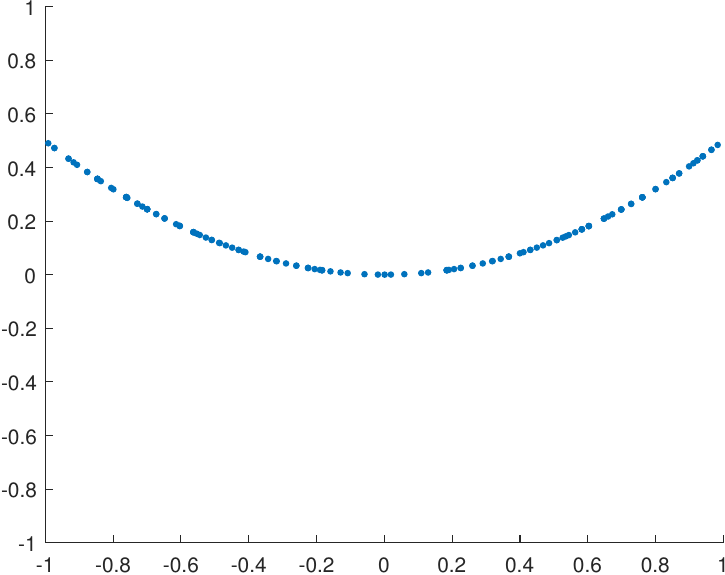}
    \includegraphics[width=0.16\textwidth]{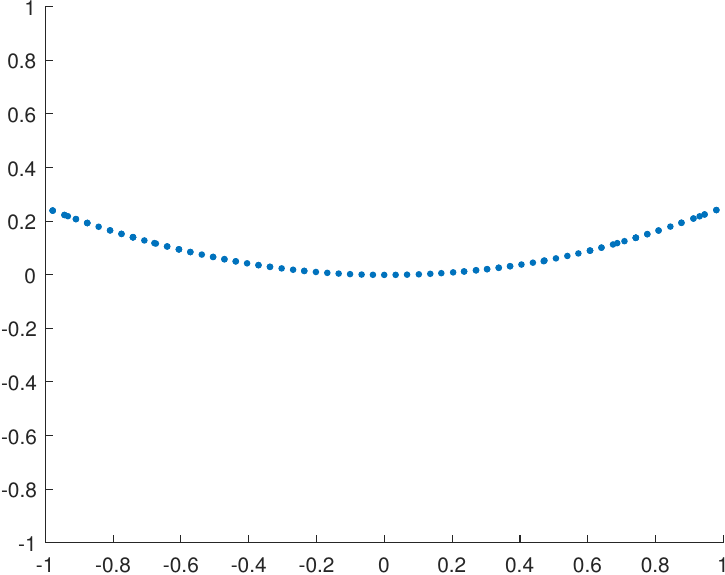}
    \includegraphics[width=0.16\textwidth]{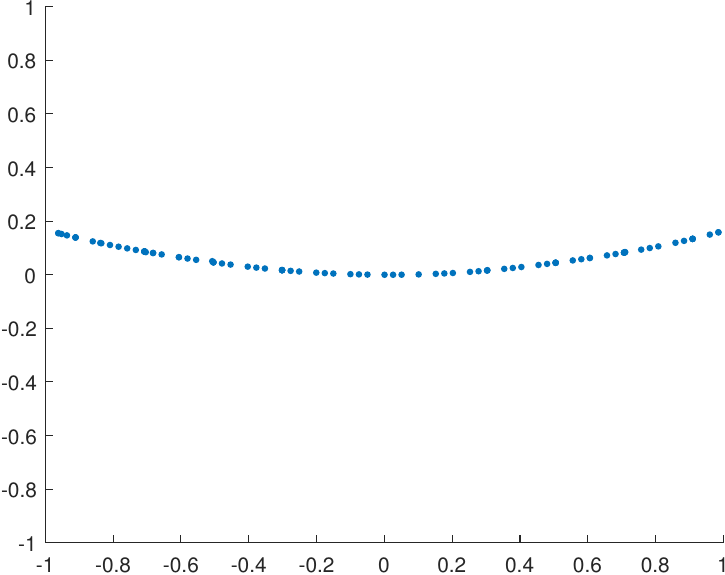}
    \includegraphics[width=0.16\textwidth]{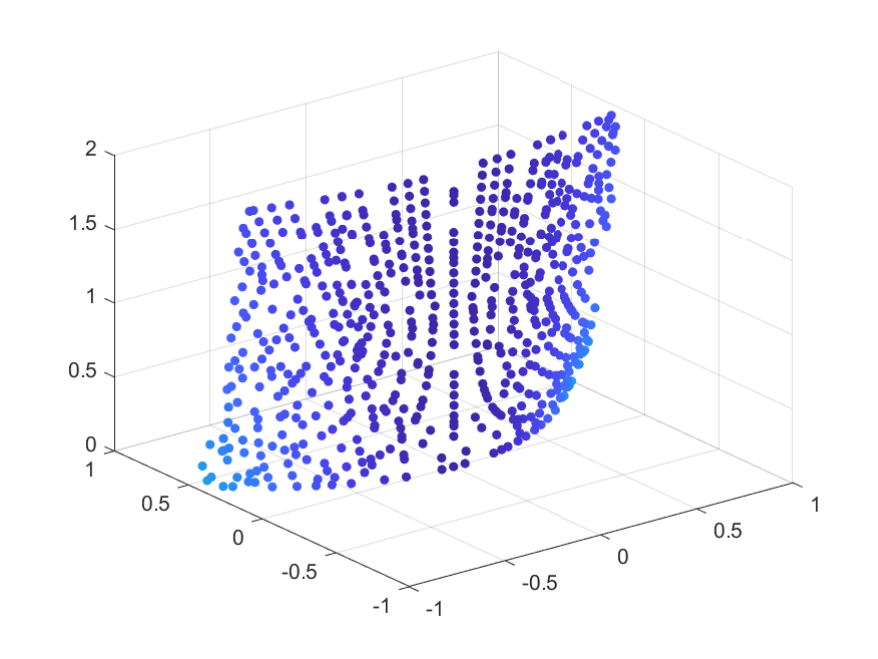}\\
    \includegraphics[width=0.16\textwidth]{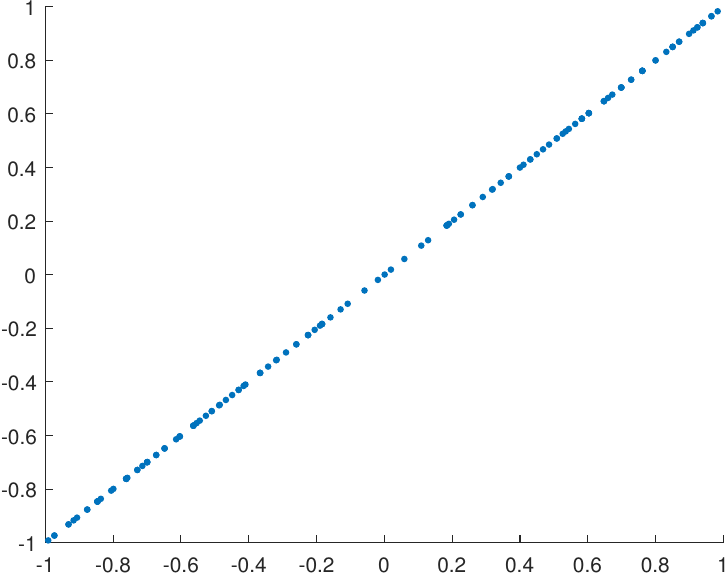}
    \includegraphics[width=0.16\textwidth]{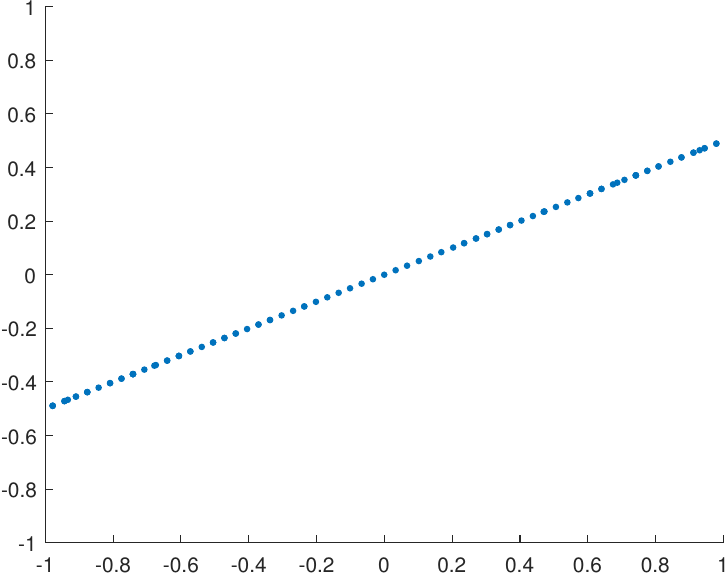}
    \includegraphics[width=0.16\textwidth]{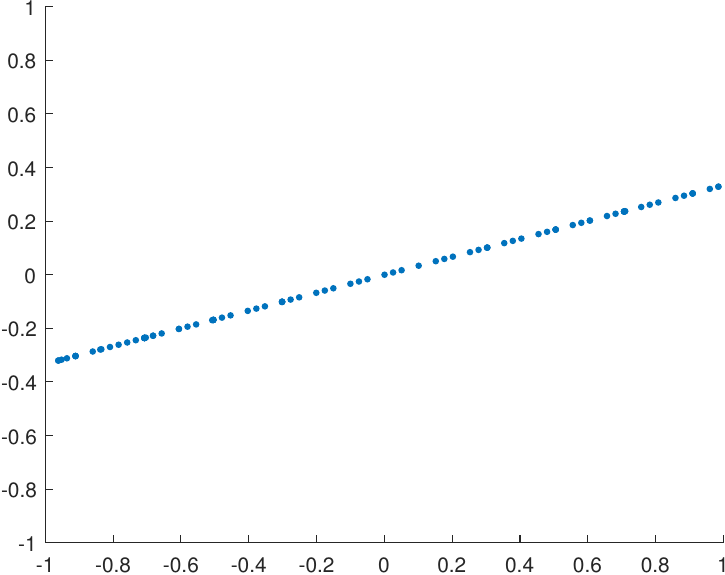}
    \includegraphics[width=0.16\textwidth]{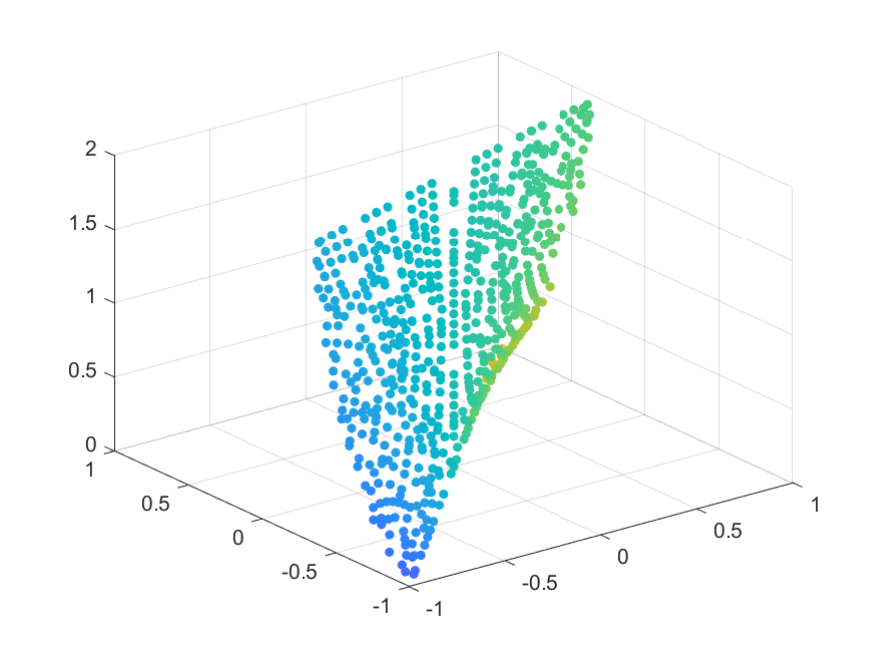}
    \caption{AG-RaNN method results for Case 1 of Example \ref{ex3-1}. Columns 1-3 display $S$ and $\partial_x S$ at $t=0,1,2$, respectively, while column 4 shows the zero level set. The top two rows use normal collocation points, and the bottom two rows use collocation set $\Lambda_A$ ($N_A^I=(42014,30178)$, $N_A^B=(2955,2719)$). Running times: $(t_1,t_2)=(3.74,4.82)$.}
    \label{fig:ex3-1_case1_HJ}
\end{figure}

\noindent\textbf{Case 2: Focusing at a point.}

In this case, the exact solution is
\begin{align}
    S(t,x)=\frac{x^2}{2(t-1)}
\end{align}
and the potential is $V(x)=0$. The numerical results are shown in Figure \ref{fig:ex3-1_case2_HJ}.

\begin{figure}[!htbp]
    \centering
    \includegraphics[width=0.16\textwidth]{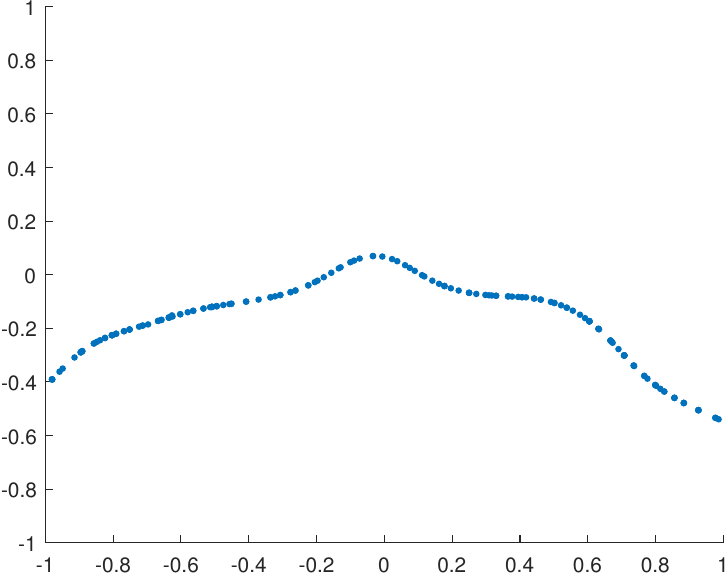}
    \includegraphics[width=0.16\textwidth]{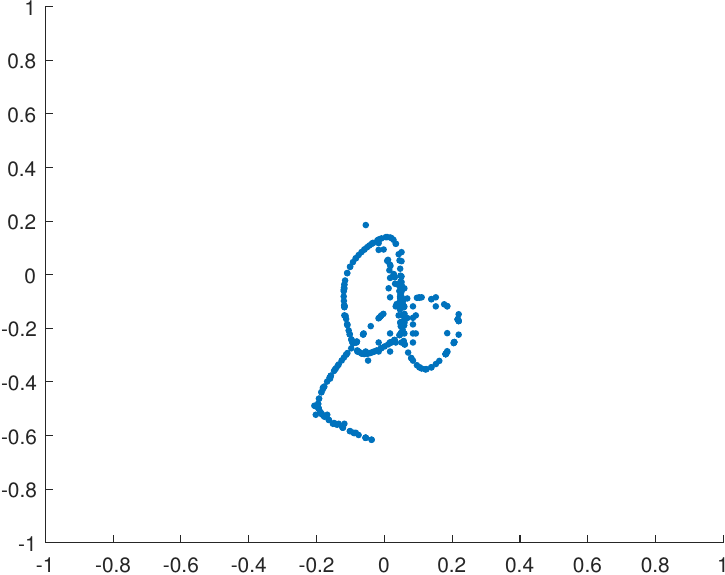}
    \includegraphics[width=0.16\textwidth]{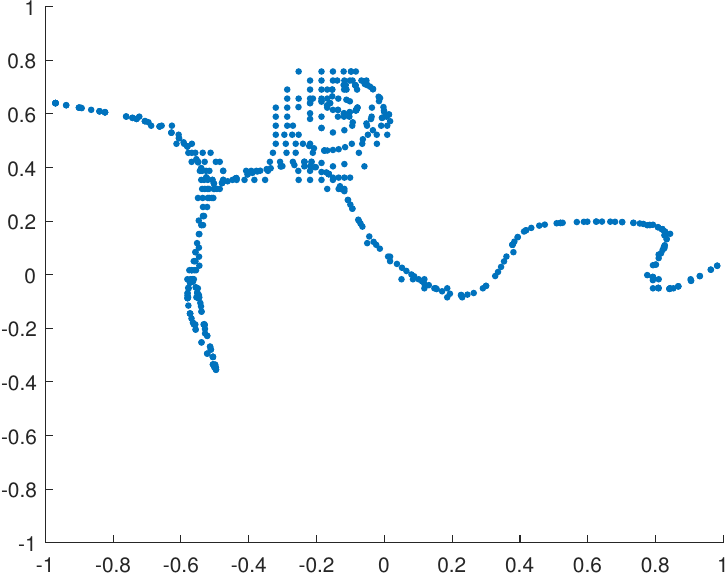}
    \includegraphics[width=0.16\textwidth]{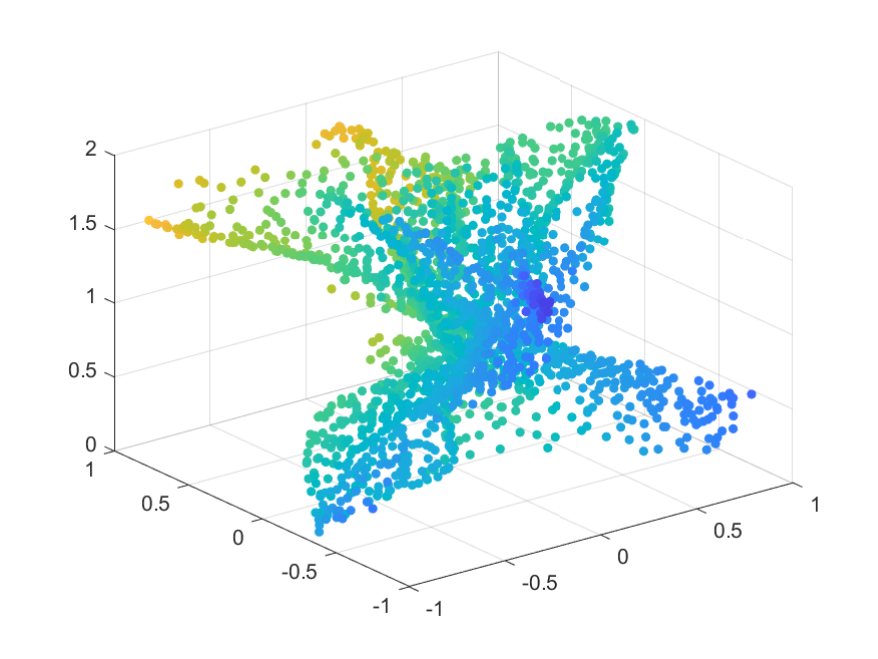}\\
    \includegraphics[width=0.16\textwidth]{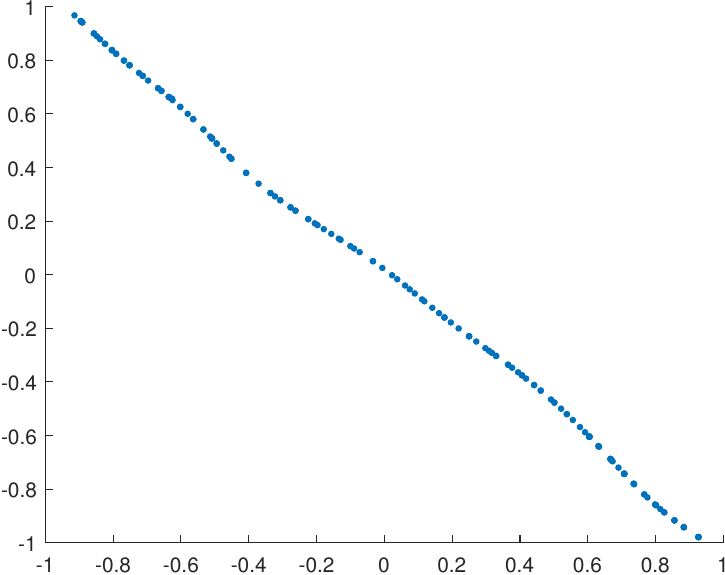}
    \includegraphics[width=0.16\textwidth]{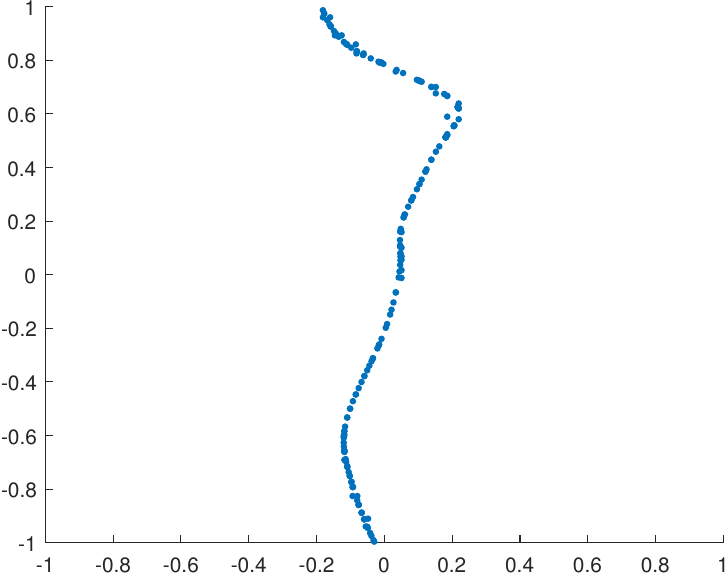}
    \includegraphics[width=0.16\textwidth]{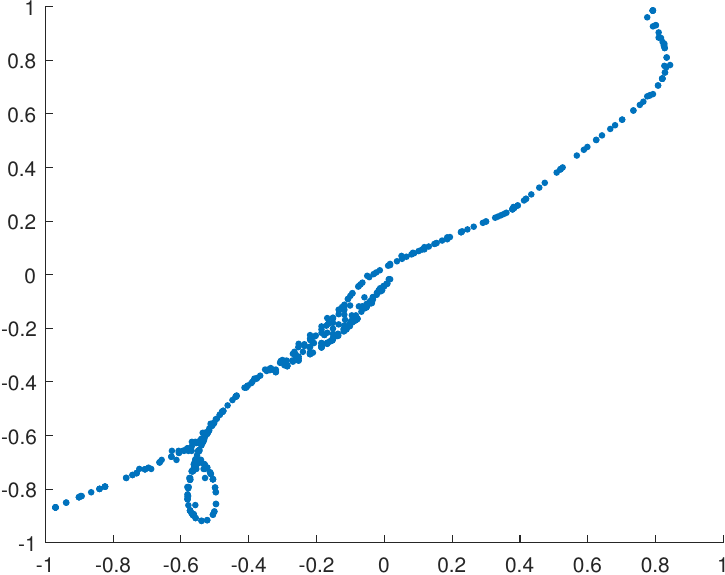}
    \includegraphics[width=0.16\textwidth]{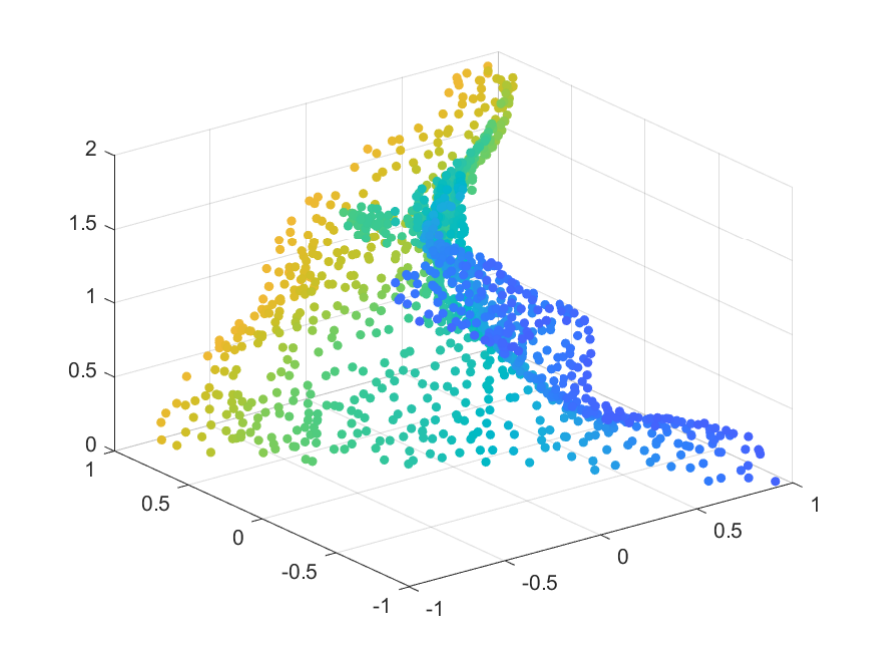}\\
    \includegraphics[width=0.16\textwidth]{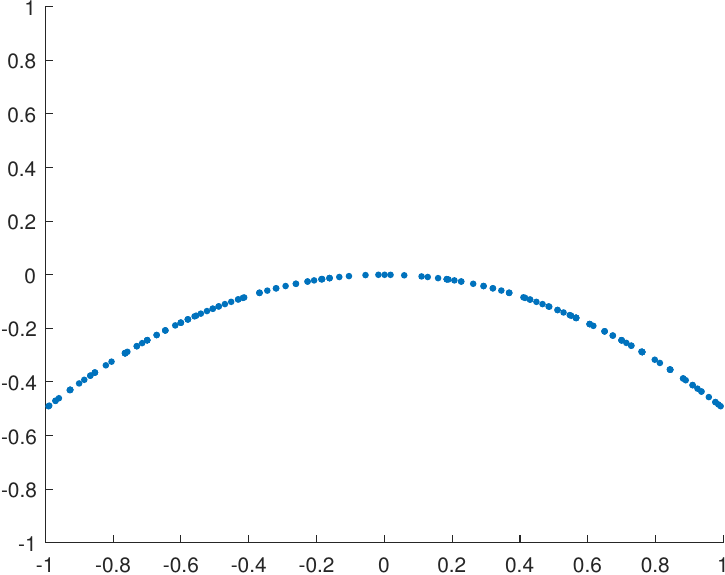}
    \includegraphics[width=0.16\textwidth]{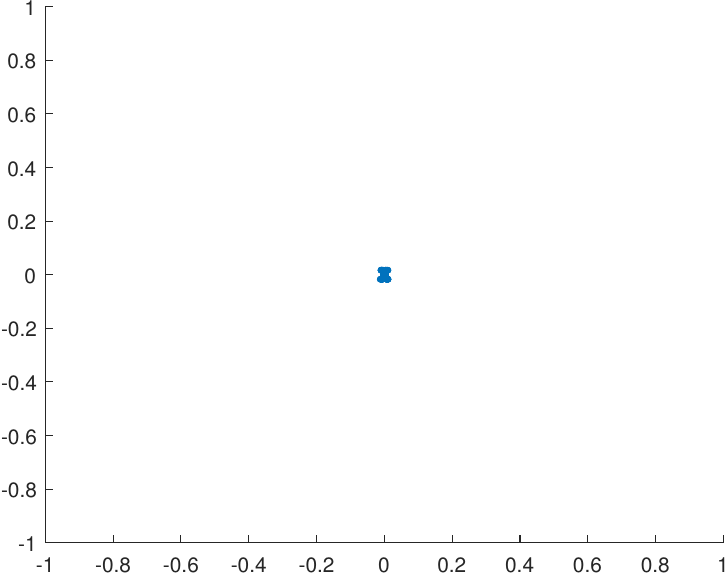}
    \includegraphics[width=0.16\textwidth]{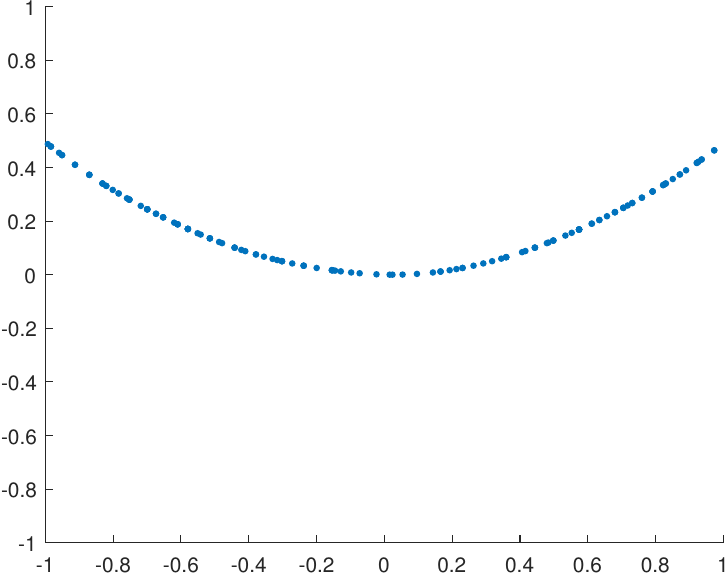}
    \includegraphics[width=0.16\textwidth]{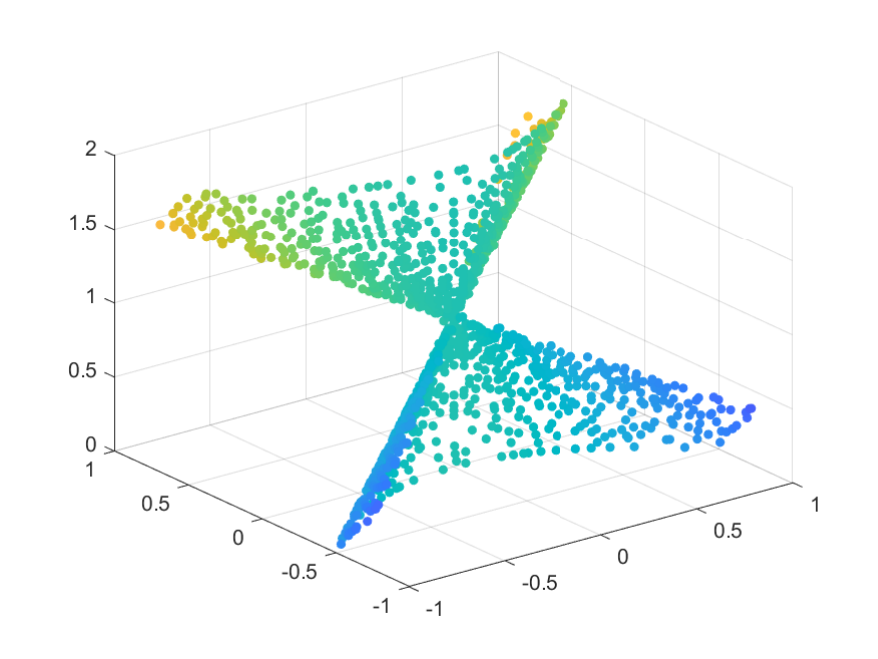}\\
    \includegraphics[width=0.16\textwidth]{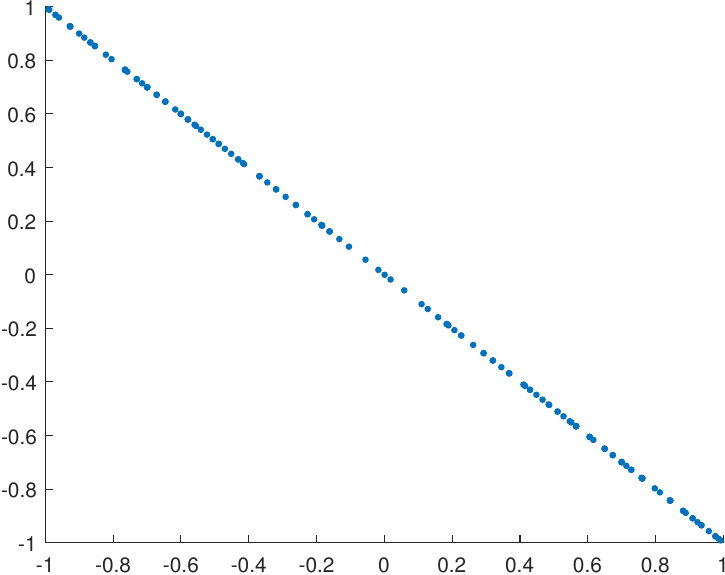}
    \includegraphics[width=0.16\textwidth]{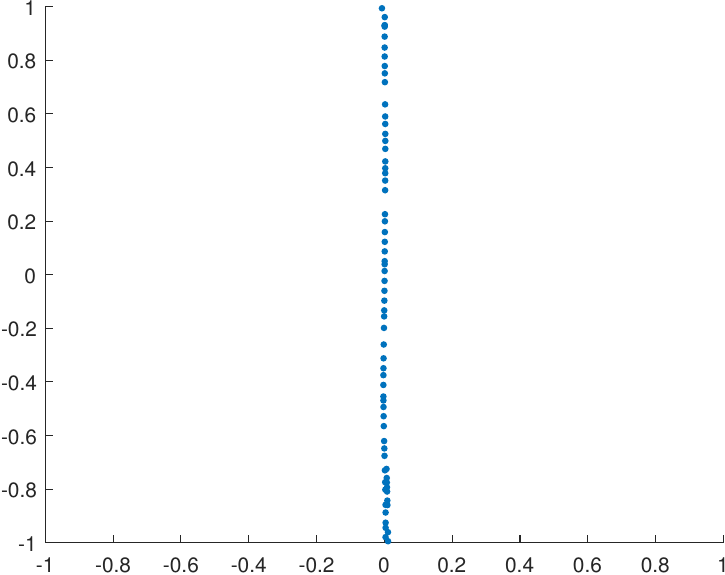}
    \includegraphics[width=0.16\textwidth]{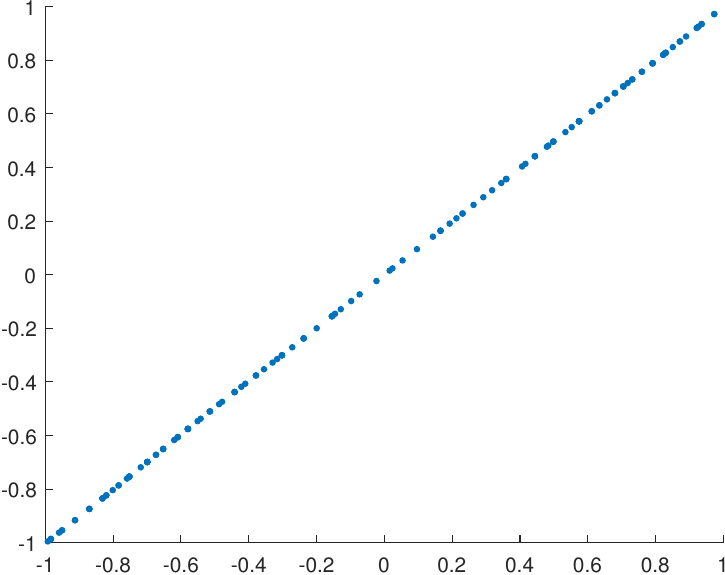}
    \includegraphics[width=0.16\textwidth]{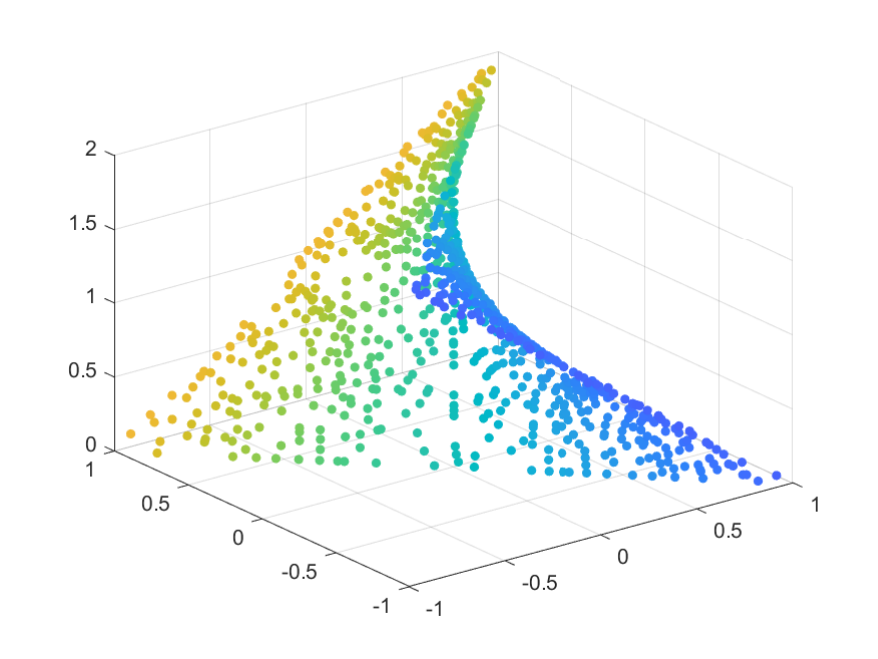}
    \caption{AG-RaNN method results for Case 2 of Example \ref{ex3-1}. Columns 1-3 display $S$ and $\partial_x S$ at $t=0,1,2$, respectively, while column 4 shows the zero level set. The top two rows use normal collocation points, and the bottom two rows use collocation set $\Lambda_A$ ($N_A^I=(53139,62047)$, $N_A^B=(2833,2766)$). Running times: $(t_1,t_2)=(3.75,6.18)$.}
    \label{fig:ex3-1_case2_HJ}
\end{figure}

\noindent\textbf{Case 3: Caustic.}

In this case, the initial condition is
\begin{align}
    S_0(x)=-\ln(\cosh(x))
\end{align}
and the potential is $V(x)=0$. The numerical results are shown in Figure \ref{fig:ex3-1_case3_HJ}.

\begin{figure}[!htbp]
    \centering
    \includegraphics[width=0.16\textwidth]{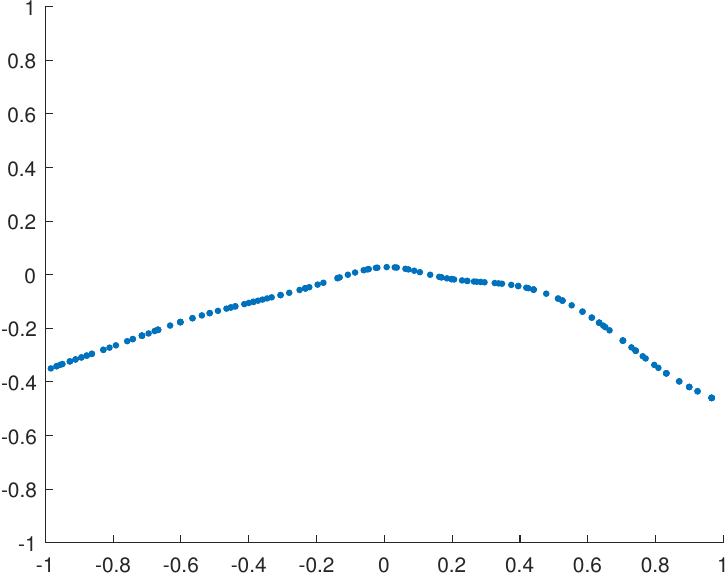}
    \includegraphics[width=0.16\textwidth]{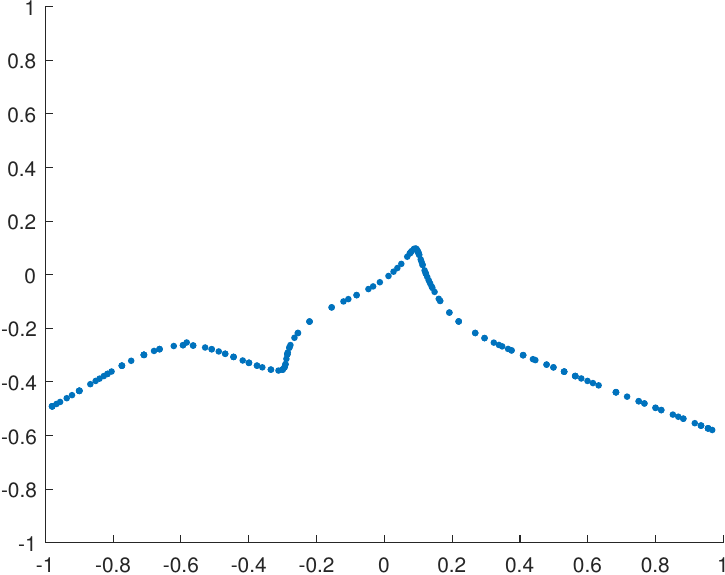}
    \includegraphics[width=0.16\textwidth]{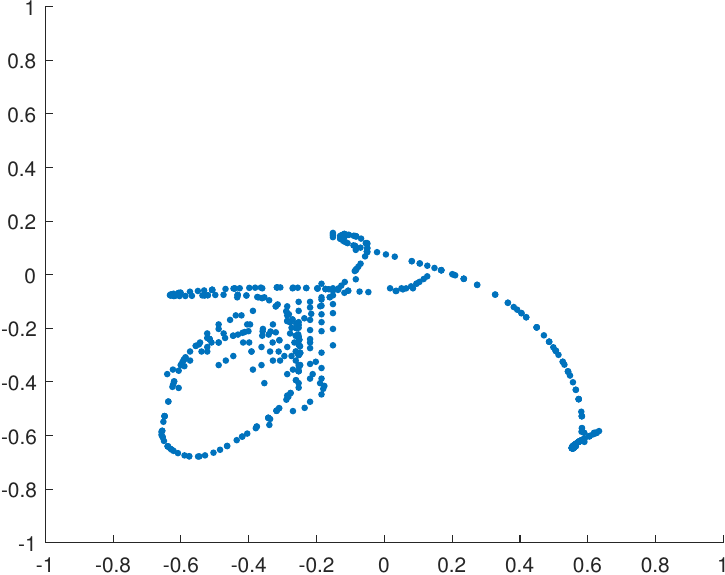}
    \includegraphics[width=0.16\textwidth]{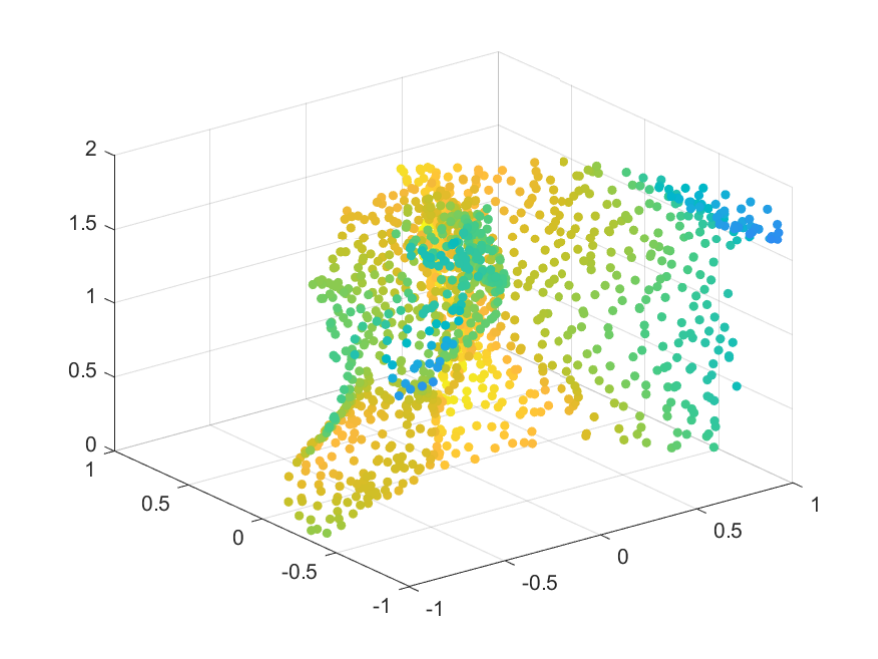}\\
    \includegraphics[width=0.16\textwidth]{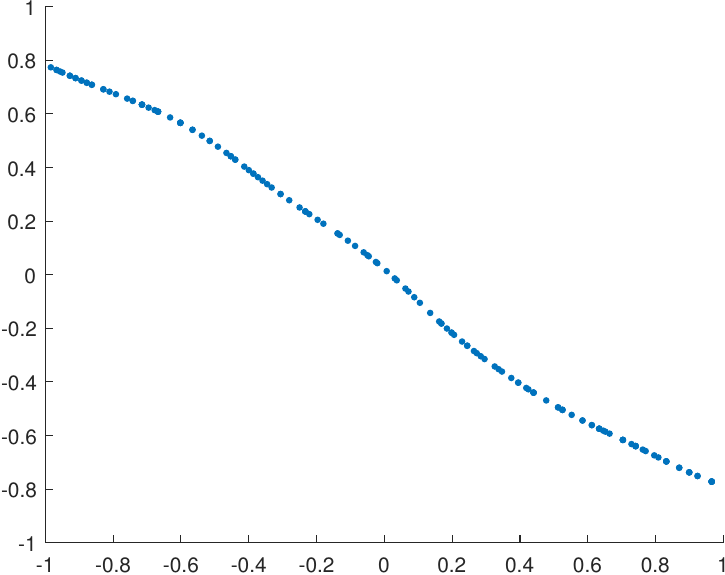}
    \includegraphics[width=0.16\textwidth]{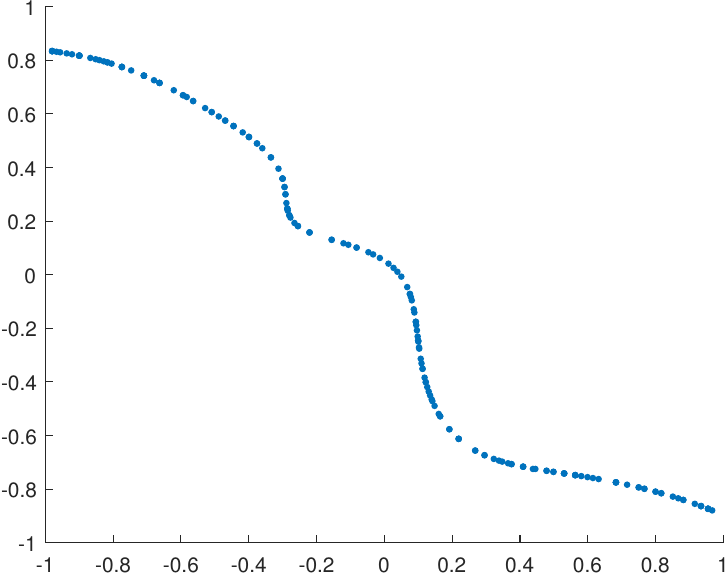}
    \includegraphics[width=0.16\textwidth]{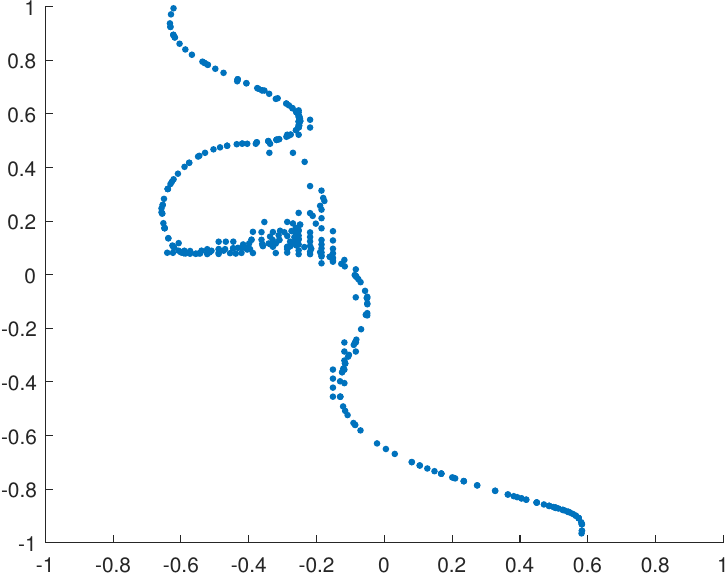}
    \includegraphics[width=0.16\textwidth]{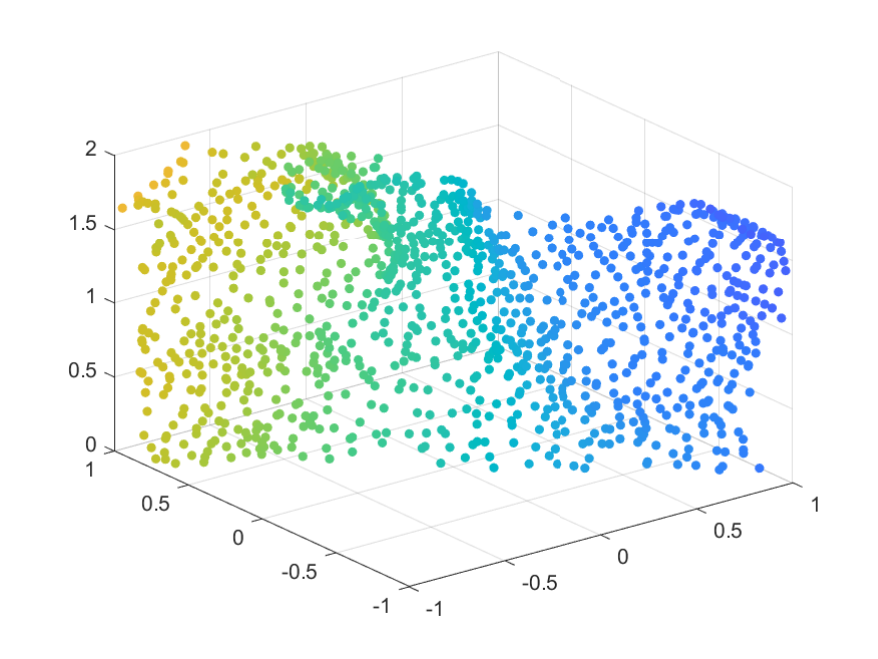}\\
    \includegraphics[width=0.16\textwidth]{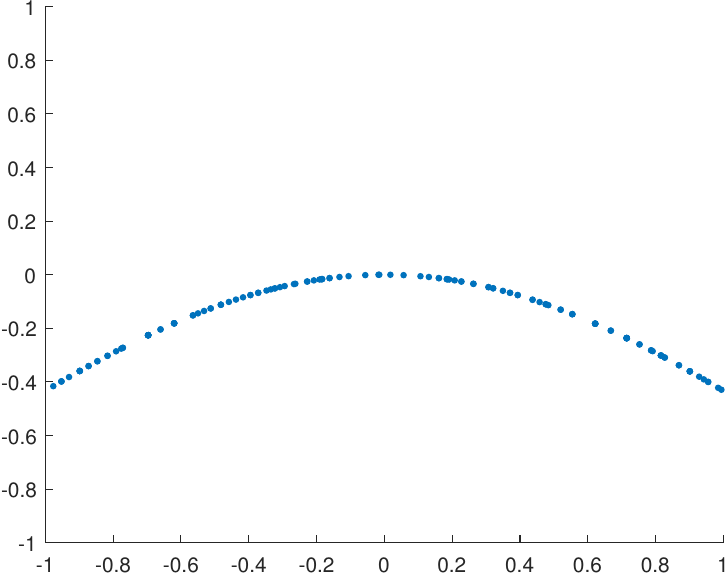}
    \includegraphics[width=0.16\textwidth]{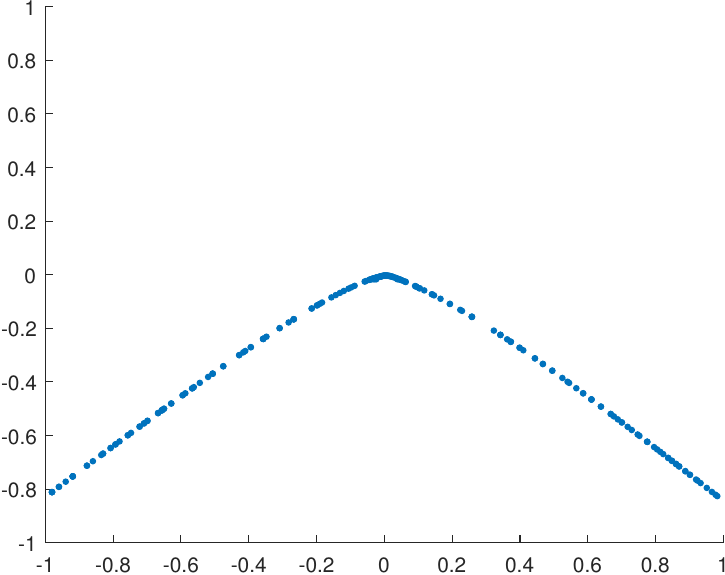}
    \includegraphics[width=0.16\textwidth]{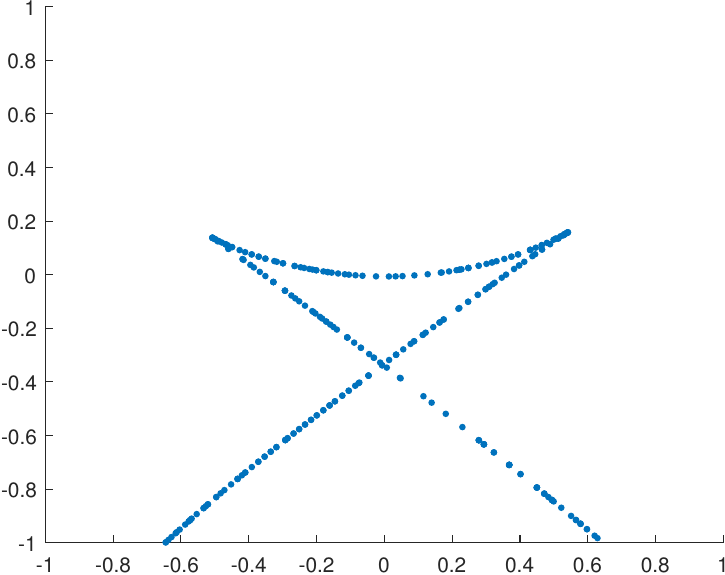}
    \includegraphics[width=0.16\textwidth]{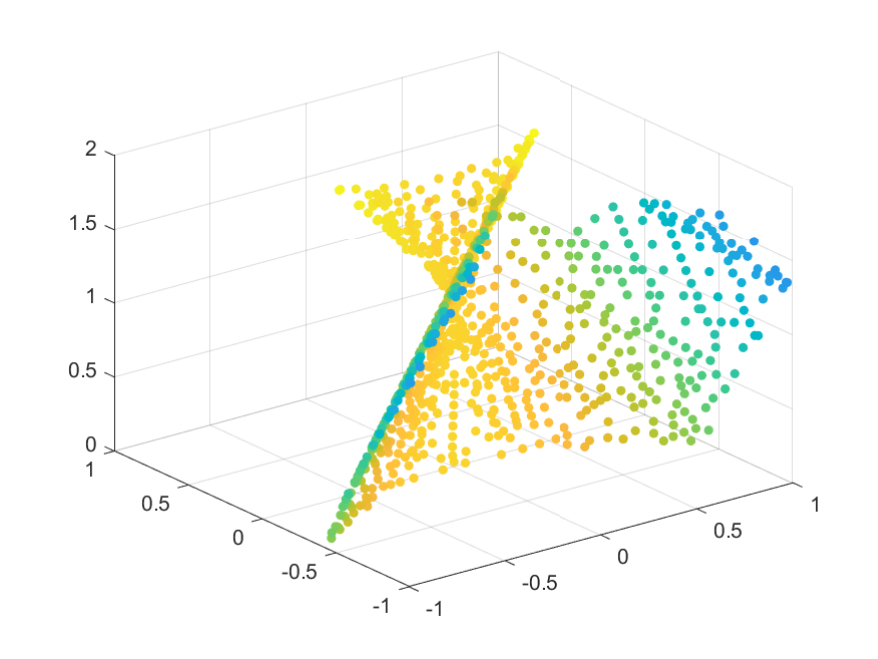}\\
    \includegraphics[width=0.16\textwidth]{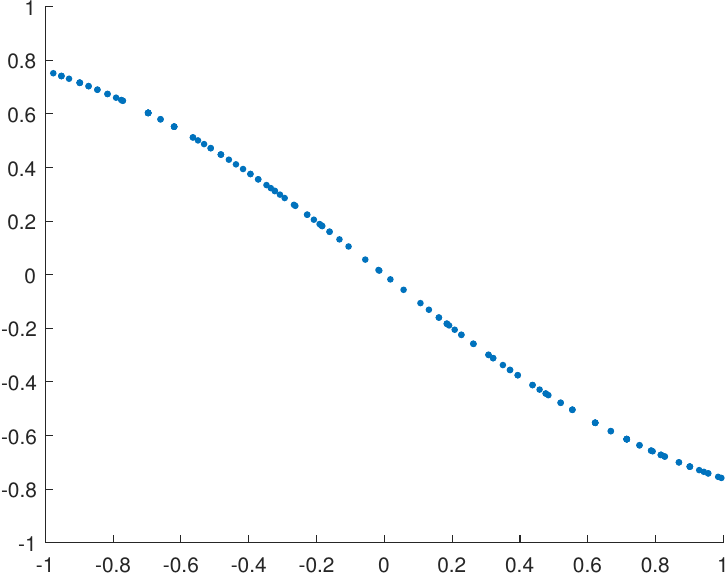}
    \includegraphics[width=0.16\textwidth]{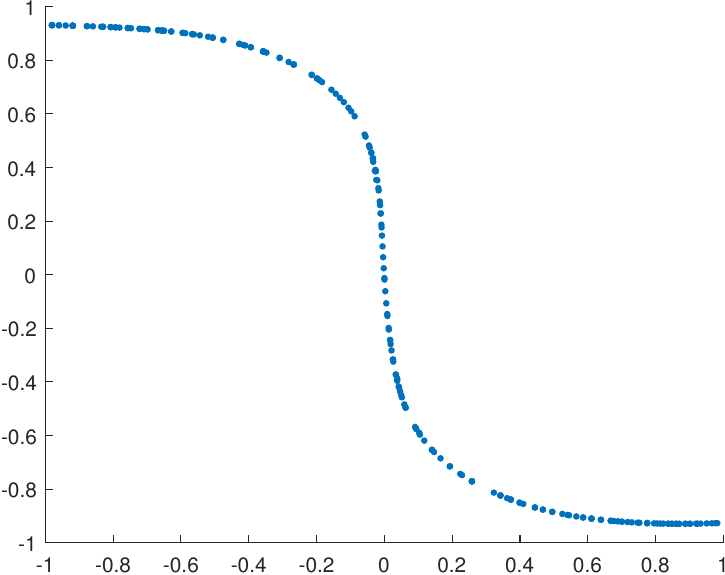}
    \includegraphics[width=0.16\textwidth]{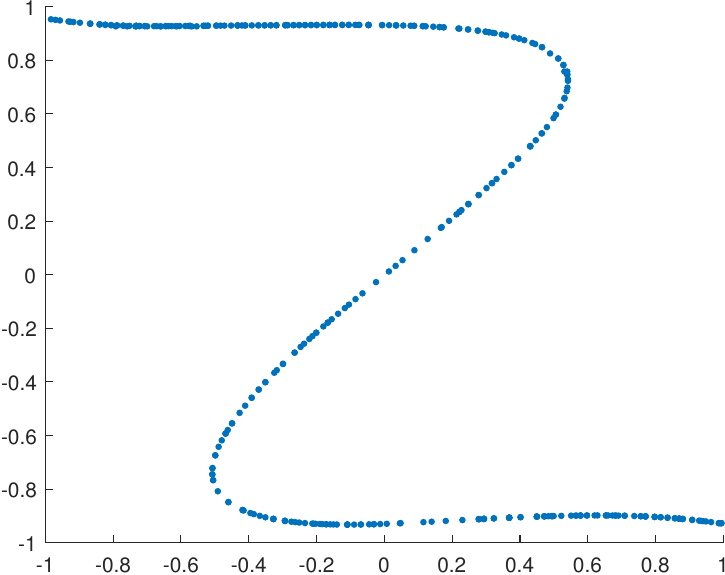}
    \includegraphics[width=0.16\textwidth]{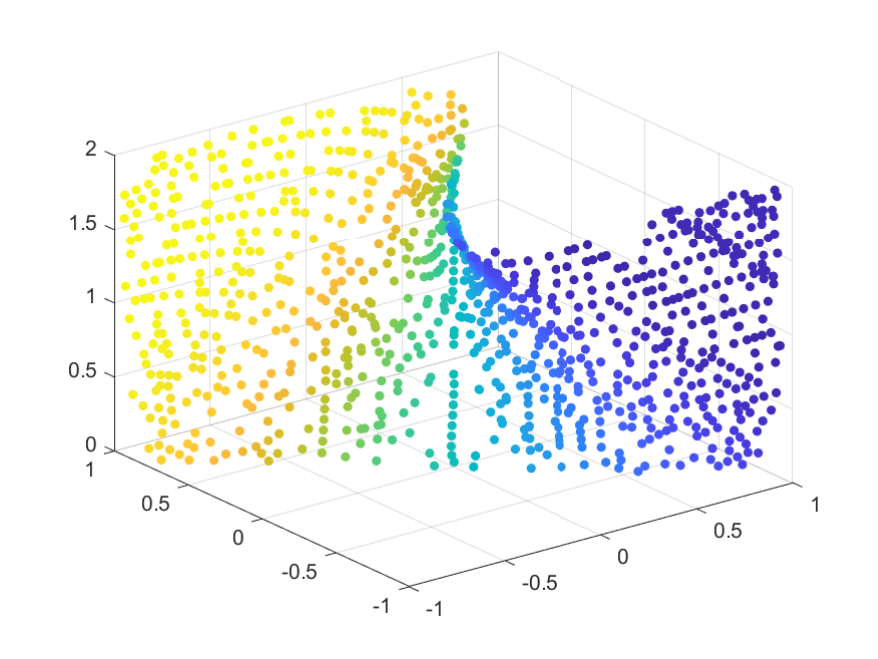}
    \caption{AG-RaNN method results for Case 3 of Example \ref{ex3-1}. Columns 1-3 display $S$ and $\partial_x S$ at $t=0,1,2$, respectively, while column 4 shows the zero level set. The top two rows use normal collocation points, and the bottom two rows use collocation set $\Lambda_A$ ($N_A^I=(58433,95927)$, $N_A^B=(2892,2942)$). Running times: $(t_1,t_2)=(3.77,7.49)$.}
    \label{fig:ex3-1_case3_HJ}
\end{figure}

\noindent\textbf{Case 4: The harmonic oscillator.}

In this case, the exact solution is
\begin{align}
    S(t,x)=-\frac{1}{2}(x^2+1)\tan(t)+\frac{x}{\cos(t)}
\end{align}
and the potential is $V(x)=x^2/2$. The numerical results are shown in Figure \ref{fig:ex3-1_case4_HJ}.

\begin{figure}[!htbp]
    \centering
    \includegraphics[width=0.16\textwidth]{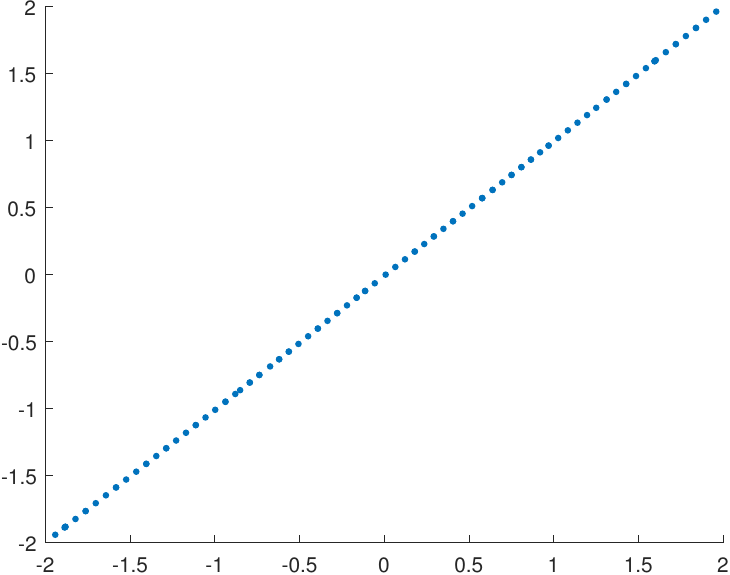}
    \includegraphics[width=0.16\textwidth]{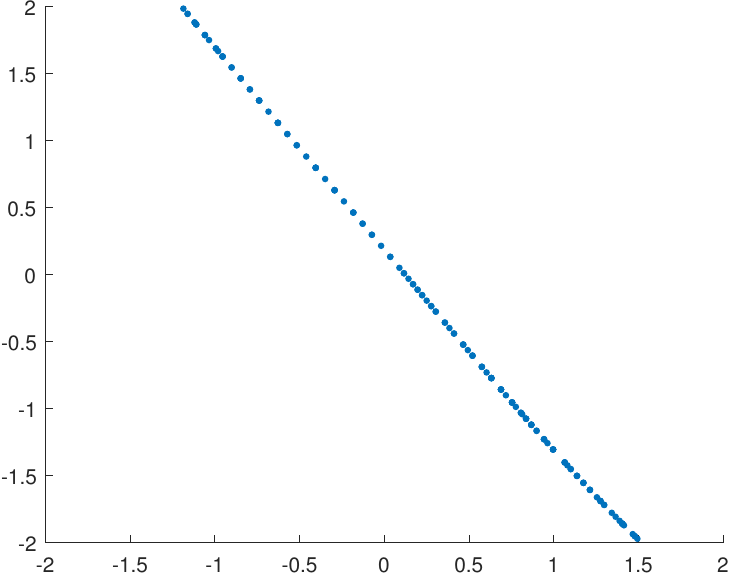}
    \includegraphics[width=0.16\textwidth]{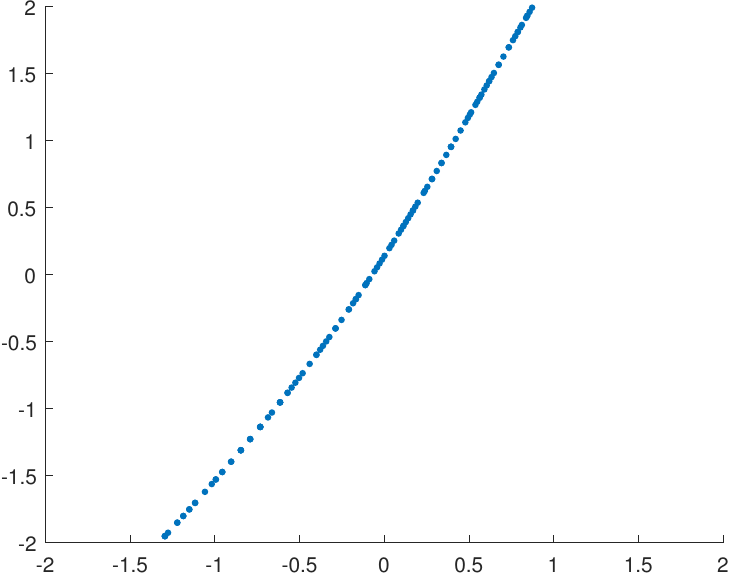}
    \includegraphics[width=0.16\textwidth]{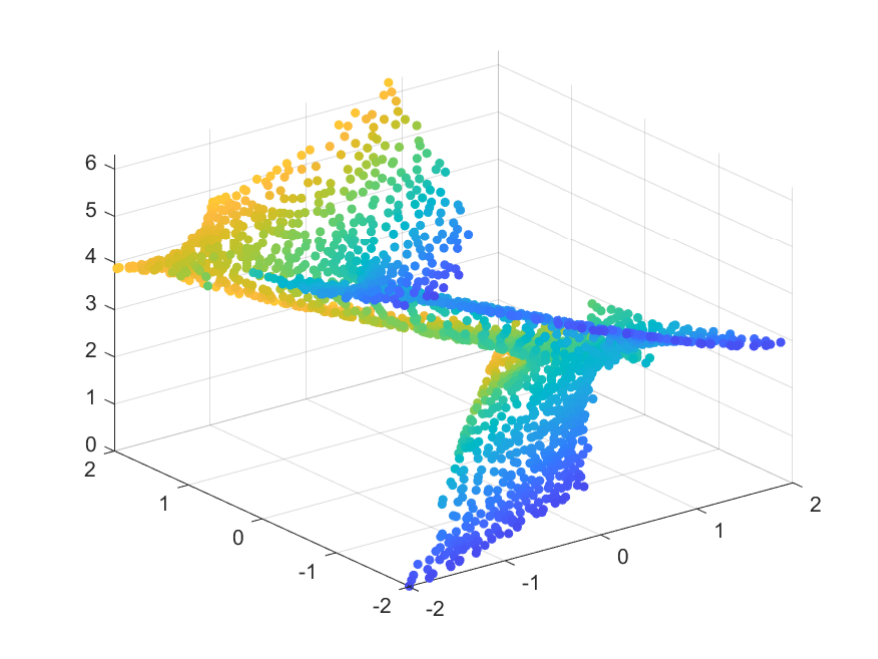}\\
    \includegraphics[width=0.16\textwidth]{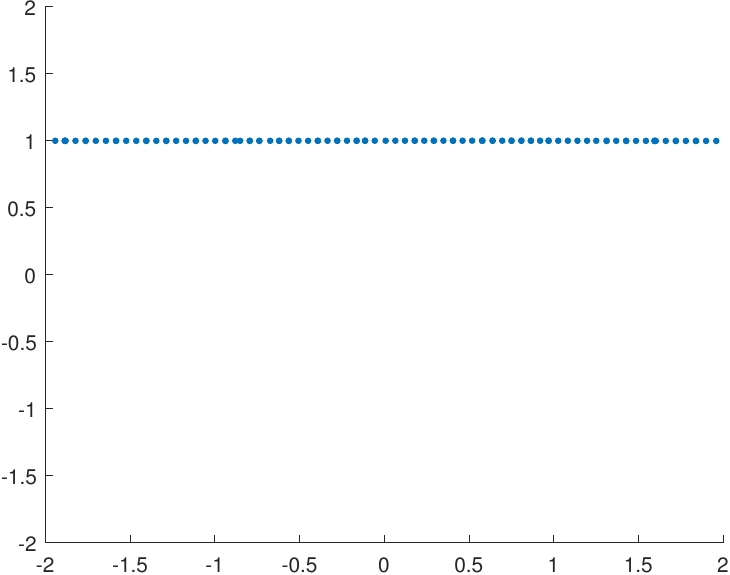}
    \includegraphics[width=0.16\textwidth]{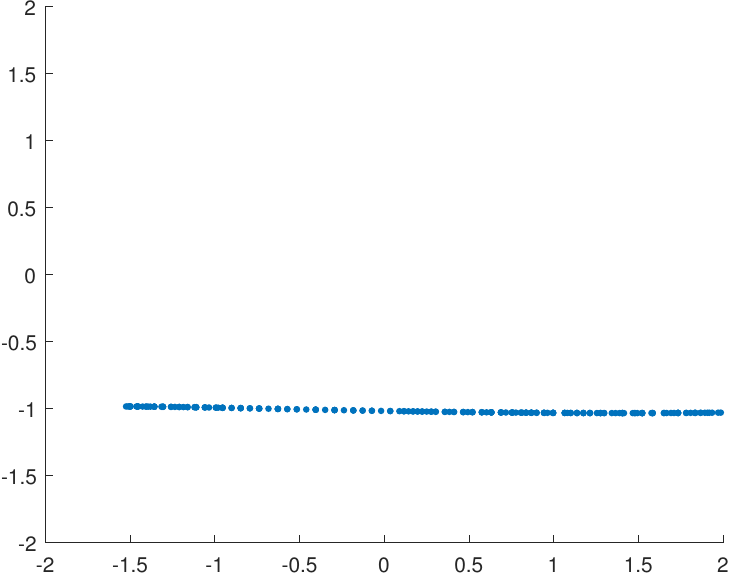}
    \includegraphics[width=0.16\textwidth]{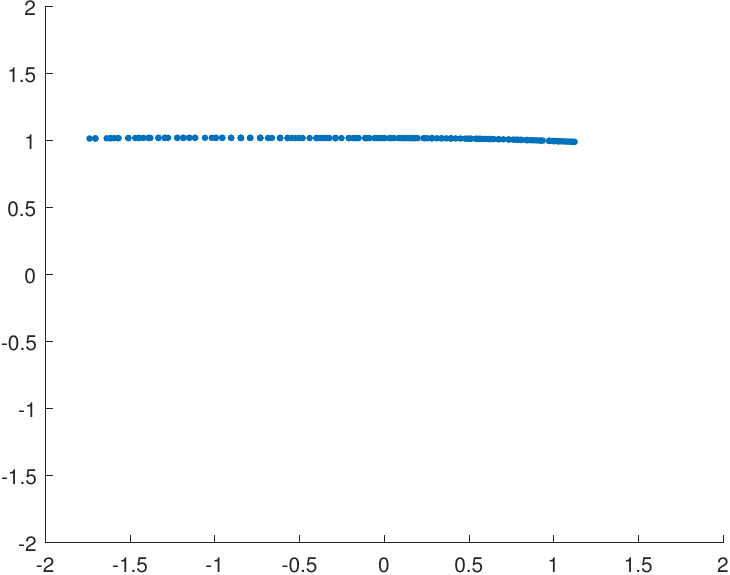}
    \includegraphics[width=0.16\textwidth]{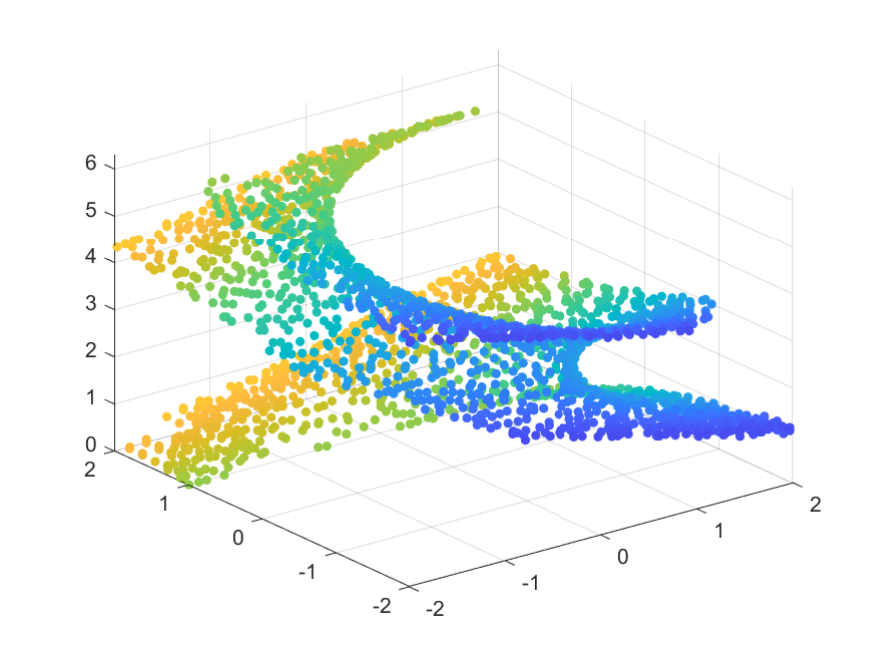}\\
    \includegraphics[width=0.16\textwidth]{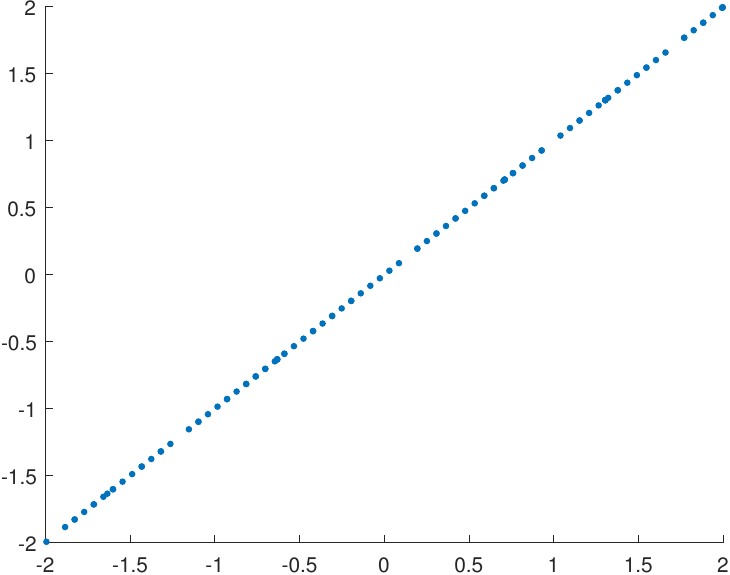}
    \includegraphics[width=0.16\textwidth]{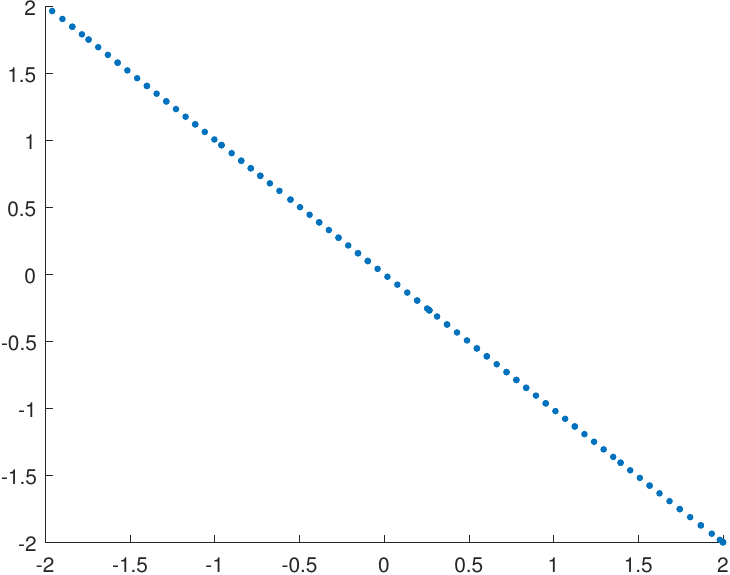}
    \includegraphics[width=0.16\textwidth]{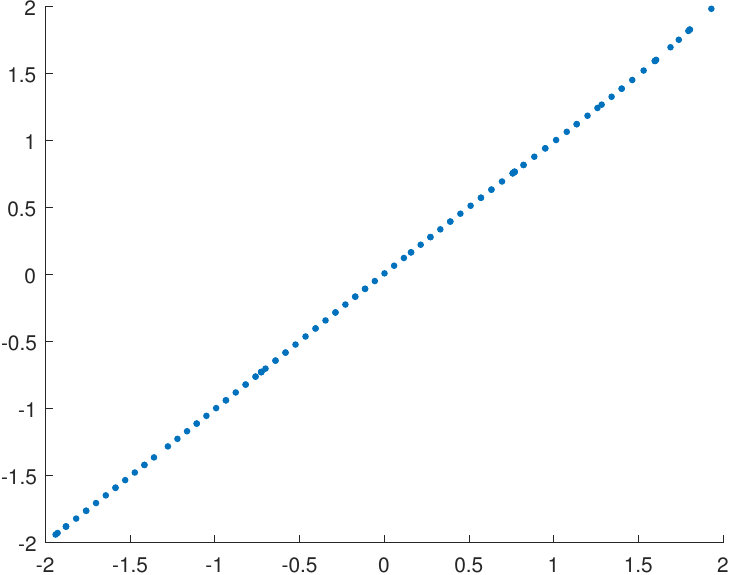}
    \includegraphics[width=0.16\textwidth]{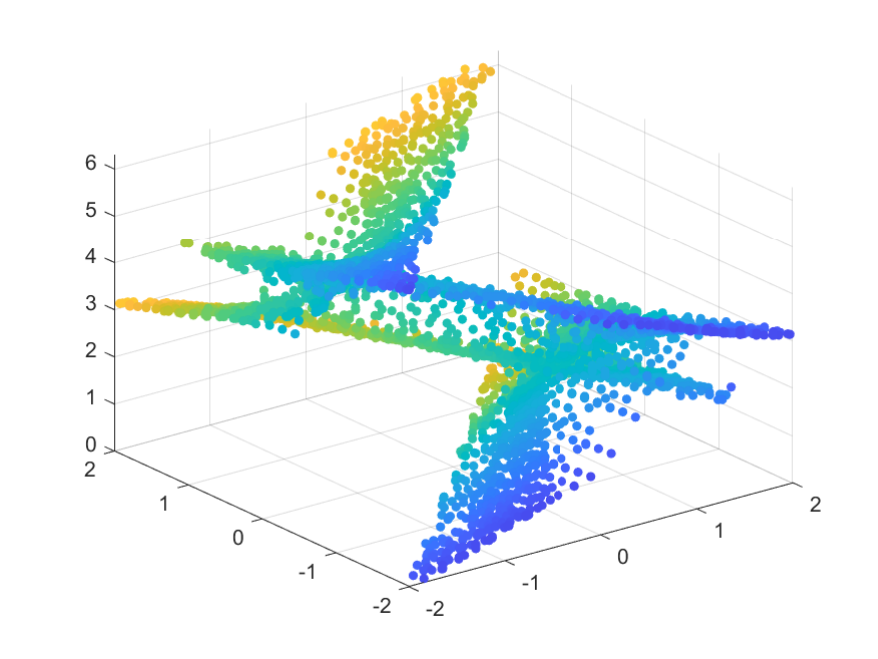}\\
    \includegraphics[width=0.16\textwidth]{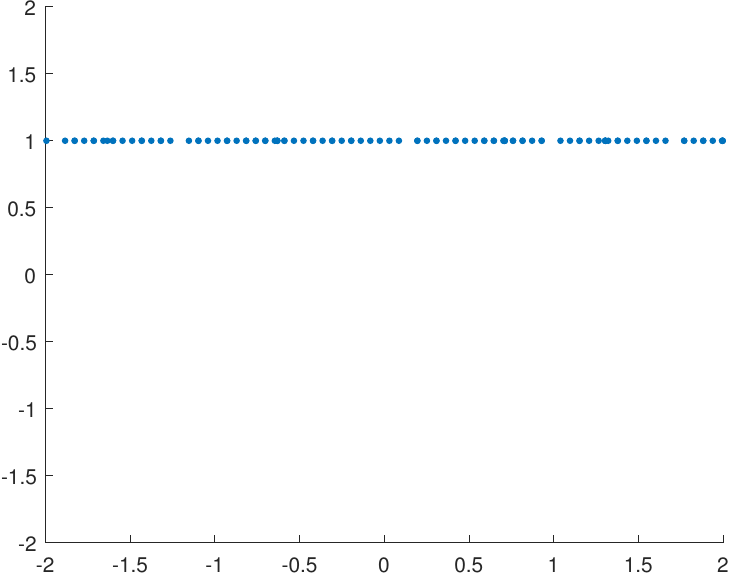}
    \includegraphics[width=0.16\textwidth]{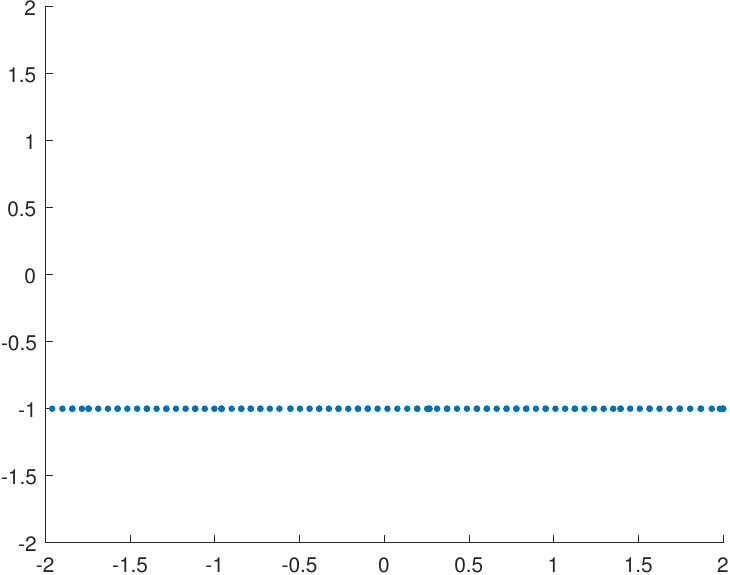}
    \includegraphics[width=0.16\textwidth]{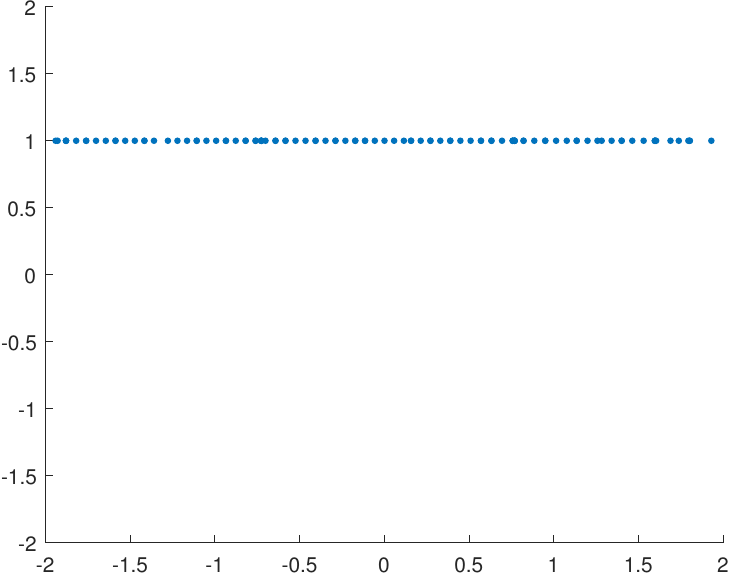}
    \includegraphics[width=0.16\textwidth]{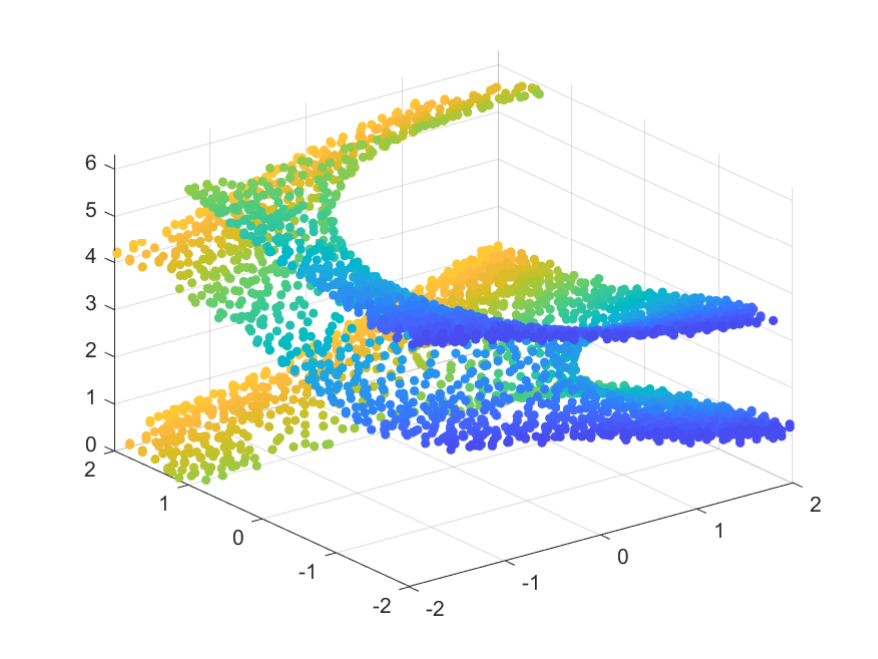}
    \caption{AG-RaNN method results for Case 4 of Example \ref{ex3-1}. Columns 1-3 display $S$ and $\partial_x S$ at $t=0,\pi,2\pi$, respectively, while column 4 shows the zero level set. The top two rows use normal collocation points, and the bottom two rows use collocation set $\Lambda_A$ ($N_A^I=(64218,76954)$, $N_A^B=(3264,3485)$). Running times: $(t_1,t_2)=(11.70,6.99)$.}
    \label{fig:ex3-1_case4_HJ}
\end{figure}

This example is the same as in \cite{Liu2006Levelset}, in order to compute the gradient of $S$ and $S$, we transform the original two-dimensional problem $S(t,x)$ into a four-dimensional problem $\phi_1(t,x,z,p)$ and $\phi_2(t,x,z,p)$. The intersection of the zero level sets
\begin{align*}
    \begin{cases}
        \phi_1(t,x,z,p)=0\\\phi_2(t,x,z,p)=0
    \end{cases}
\end{align*}
is a two-dimensional surface in the four-dimensional space. Then we project it to $p$ and $z$, and get two surfaces $z=S(t,x)$ (the right subfigures in the first and third rows of Figure \ref{fig:ex3-1_case1_HJ}, \ref{fig:ex3-1_case2_HJ}, \ref{fig:ex3-1_case3_HJ} and \ref{fig:ex3-1_case4_HJ}) and $p=\partial_xS(t,x)$ (the right subfigures in the second and fourth rows of Figure \ref{fig:ex3-1_case1_HJ}, \ref{fig:ex3-1_case2_HJ}, \ref{fig:ex3-1_case3_HJ} and \ref{fig:ex3-1_case4_HJ}). From the numerical results, these results are similar to the results of Example \ref{ex3}, and all have good numerical accuracy.

\begin{example}[2D Hamilton--Jacobi Equation] \label{ex4}
Finally, we consider the following 2D Hamilton--Jacobi equation:
\begin{align}
    \partial_t S+\frac{1}{2}|\nabla_{\bx}S|^2&= 0,\quad(t,\bx)\in(0,\infty)\times\Real^2,\\
    S(0,\bx) &= S_0(\bx),\quad \bx\in\Real^2.
\end{align}
Taking the gradient $\bv=\nabla_{\bx}S$, one has 
\begin{align}
    \partial_t\bv+\bv\cdot\nabla_{\bx}\bv&= 0, \quad(t,\bx)\in(0,\infty)\times\Real^2,\label{eq:ex4_1}\\
    \bv(0,\bx) = \bv_0(\bx) : &= \nabla_{\bx}S_0(\bx),\quad \bx\in\Real^2.\label{eq:ex4_2}
\end{align}

\end{example}
According to the characteristics method, the solution of \eqref{eq:ex4_1}-\eqref{eq:ex4_2} has the following form
\begin{align}
   \bv(t,\bx)=\bv_0\big(\bx-t\,\bv(t,\bx)\big), \label{eq:ex4_3}
\end{align}
which is an implicit equation. We solve the implicit equation \eqref{eq:ex4_3} for $\bv(t,\bx)$. We take
\begin{align}
    S_0(\bx)=\frac{0.45}{\pi}(\sin(\pi x_1)-1)(\sin(\pi x_2)-1)\quad\text{and}\quad\bv_0(\bx)=\nabla_{\bx}S_0(\bx).
\end{align}
We present all parameters for each case in Table \ref{tab:ex4_parameter}. The exact solutions and numerical results are shown in Figure \ref{fig:ex4}.
\begin{table}[!htbp]
    \centering
    \begin{tabular}{|c|c|c|c|c|c|c|c|c|c|c|c|c|c|c|c|}
    \hline
    $N^I$ & $N^B$ & E & Var & $N$ & $\varepsilon_A$ & T & $\Omega$ & $\br_1$ \\ \hline
    500000 & 50000 & (0,0,0,0) & (1,1,1,1) & $20^5$ & 0.5 & 1 & $[-1.2,1.2]^4$ & (2,2,2,2,2) \\ \hline
    \end{tabular}
    \caption{Parameters used in Example \ref{ex4}. }
    \label{tab:ex4_parameter}
\end{table}

\begin{figure}[!htbp]
    \centering
    \includegraphics[width=0.20\textwidth]{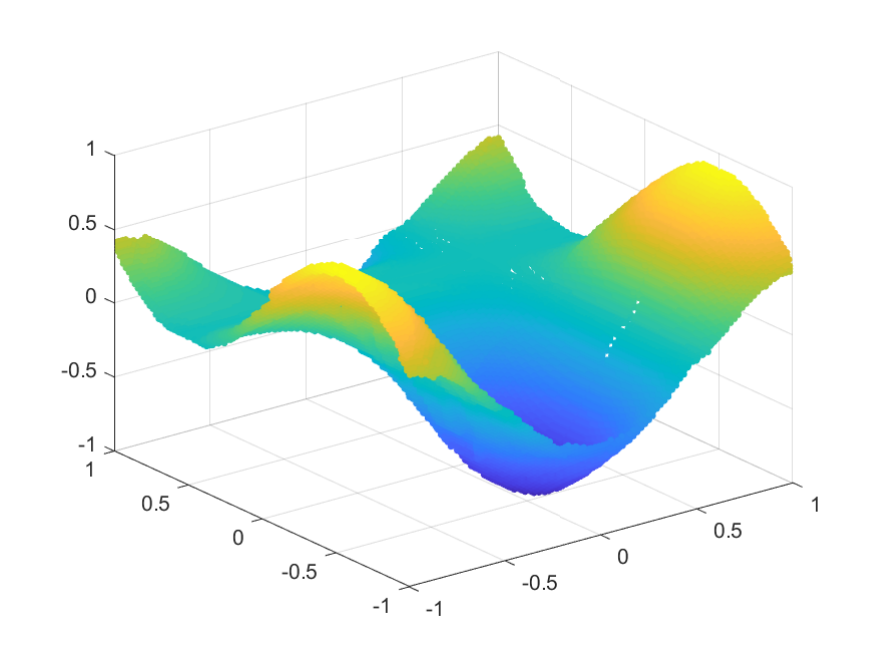}
    \includegraphics[width=0.20\textwidth]{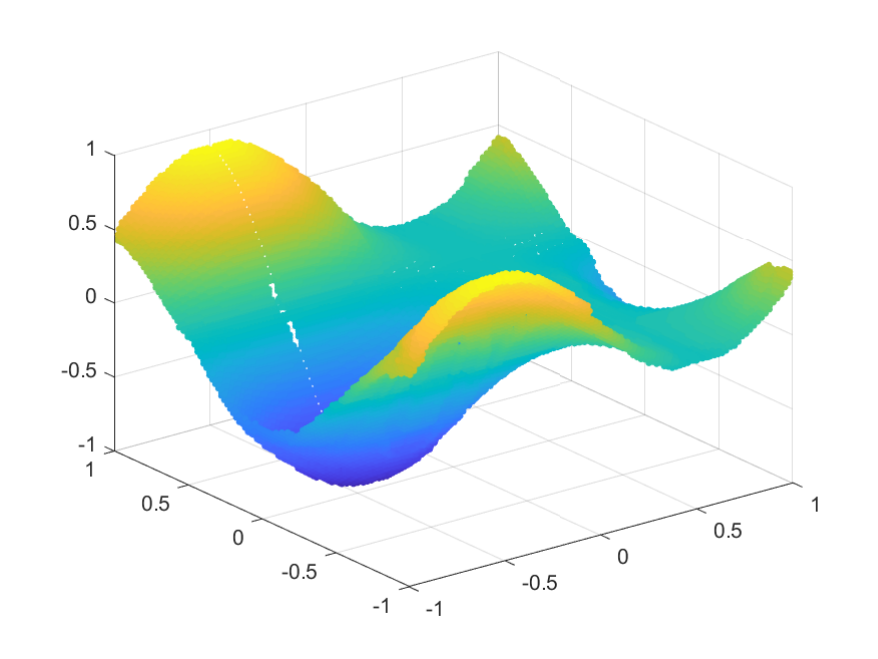}
    \includegraphics[width=0.20\textwidth]{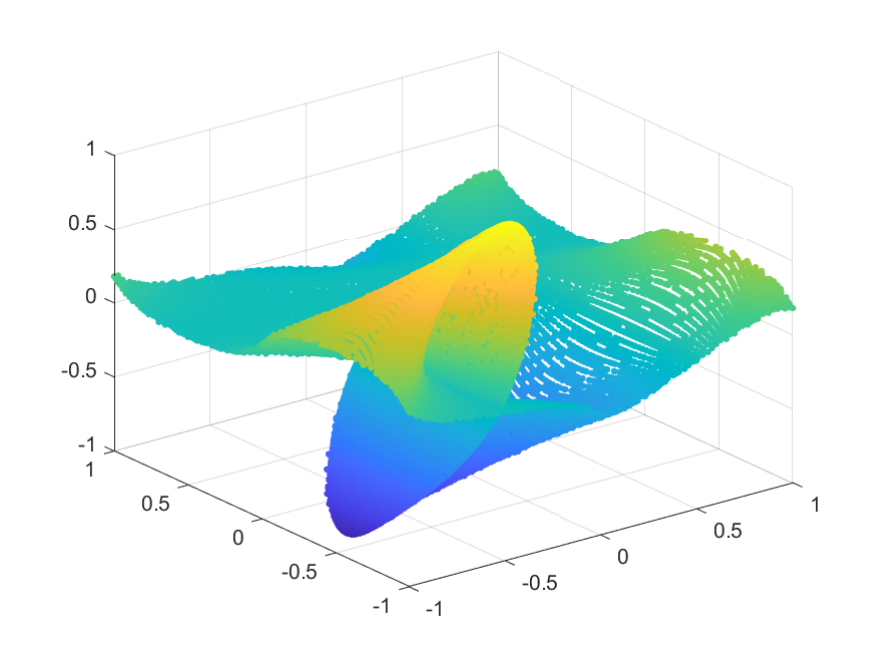}
    \includegraphics[width=0.20\textwidth]{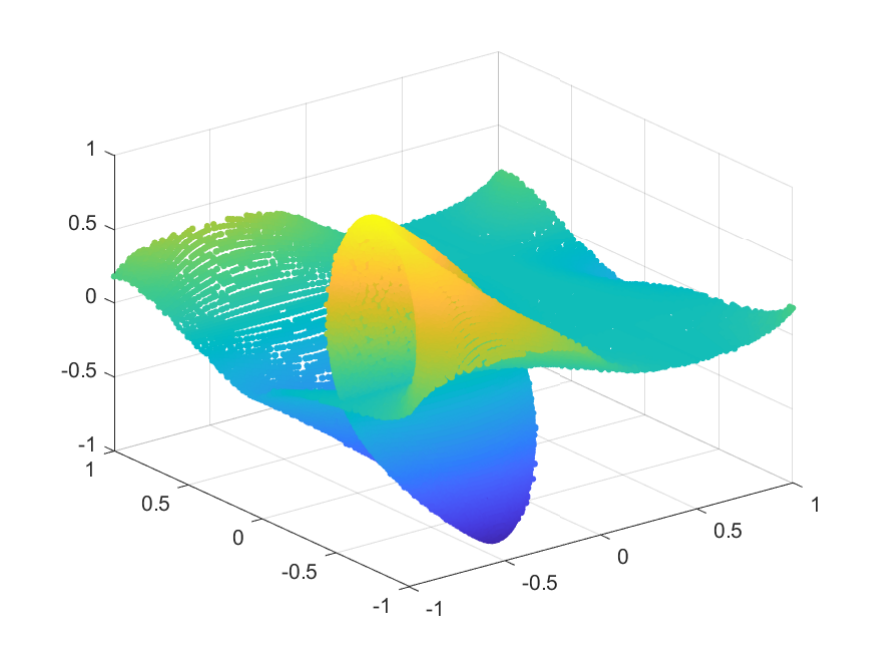}
    \includegraphics[width=0.20\textwidth]{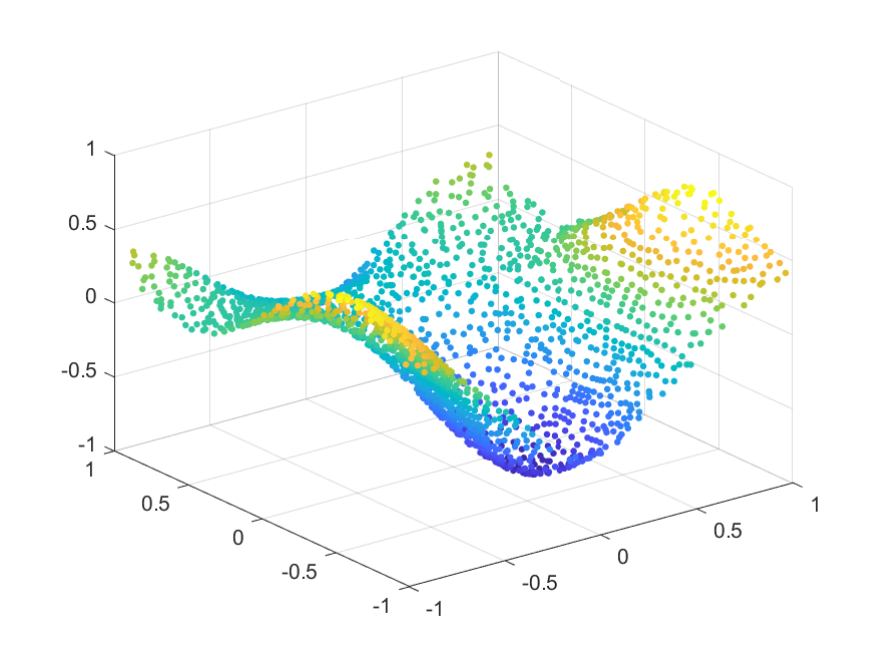}
    \includegraphics[width=0.20\textwidth]{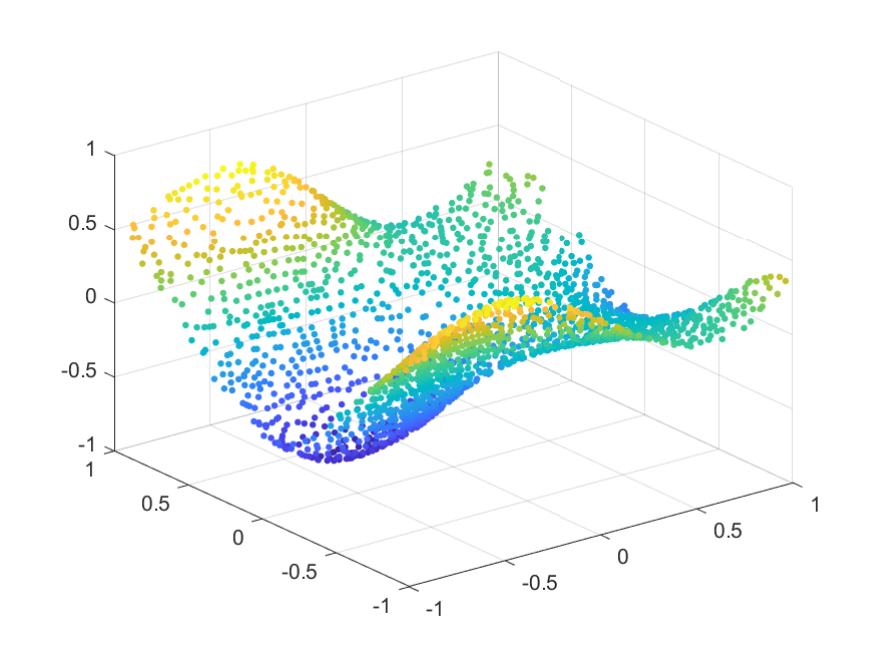}
    \includegraphics[width=0.20\textwidth]{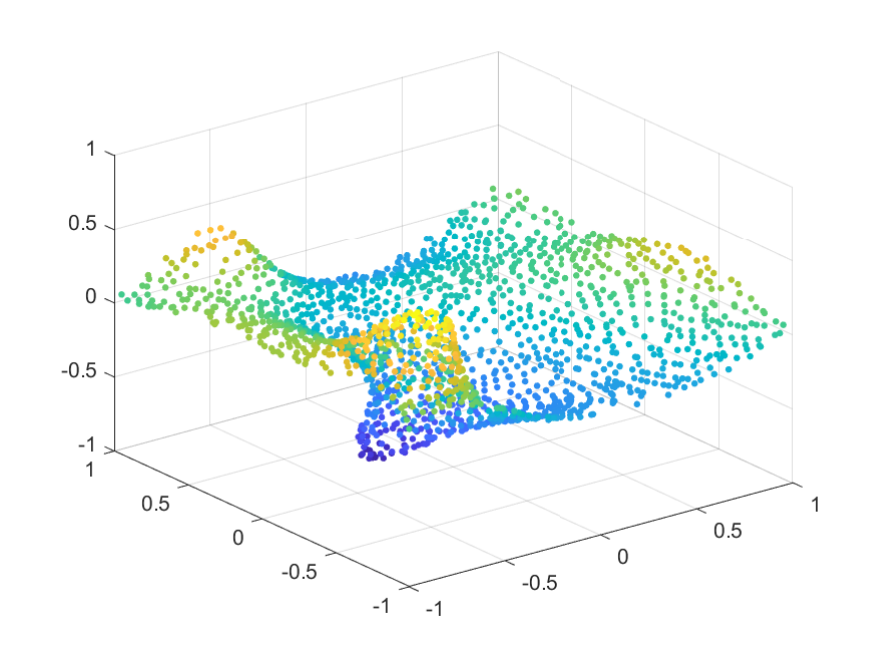}
    \includegraphics[width=0.20\textwidth]{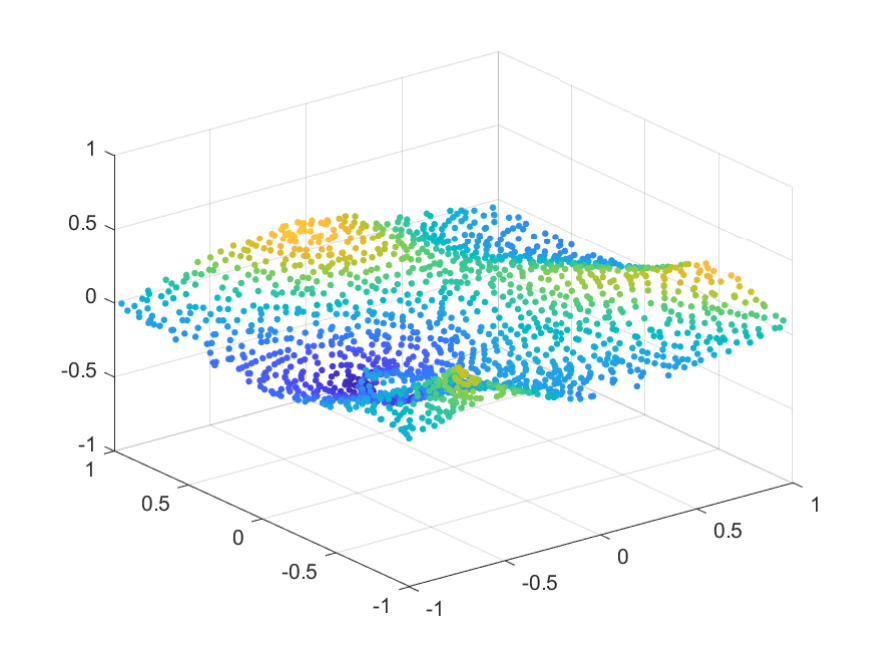}
    \includegraphics[width=0.20\textwidth]{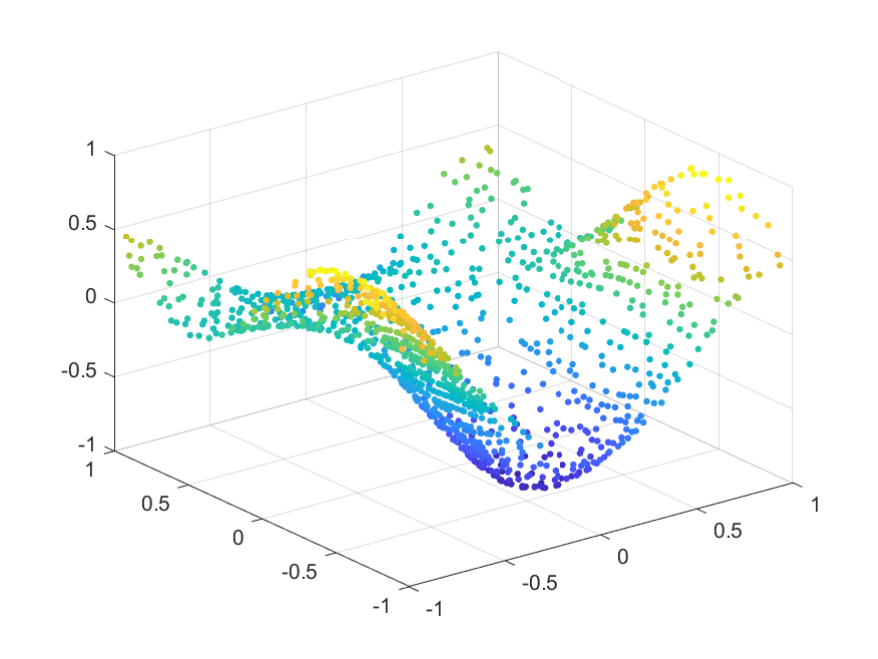}
    \includegraphics[width=0.20\textwidth]{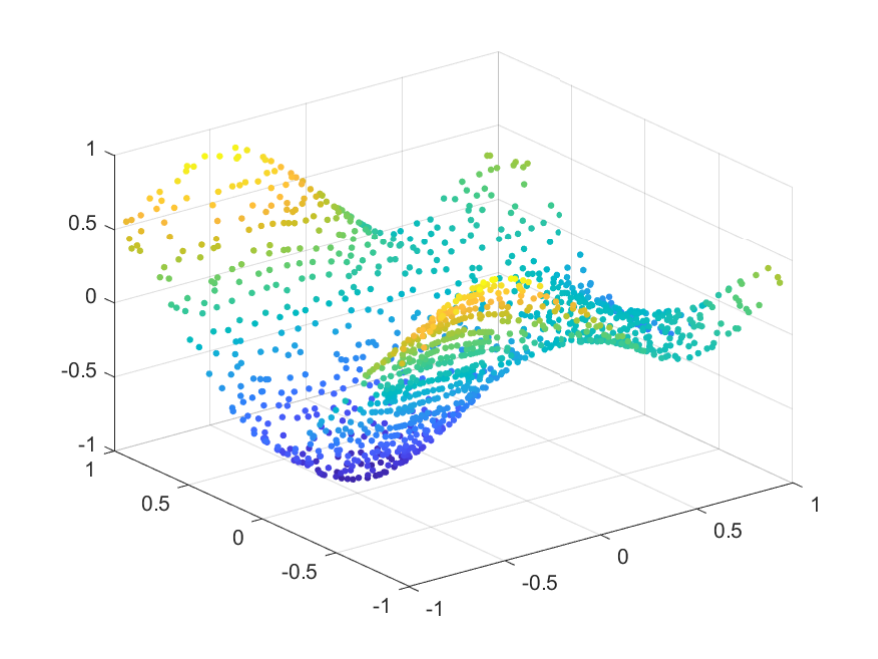}
    \includegraphics[width=0.20\textwidth]{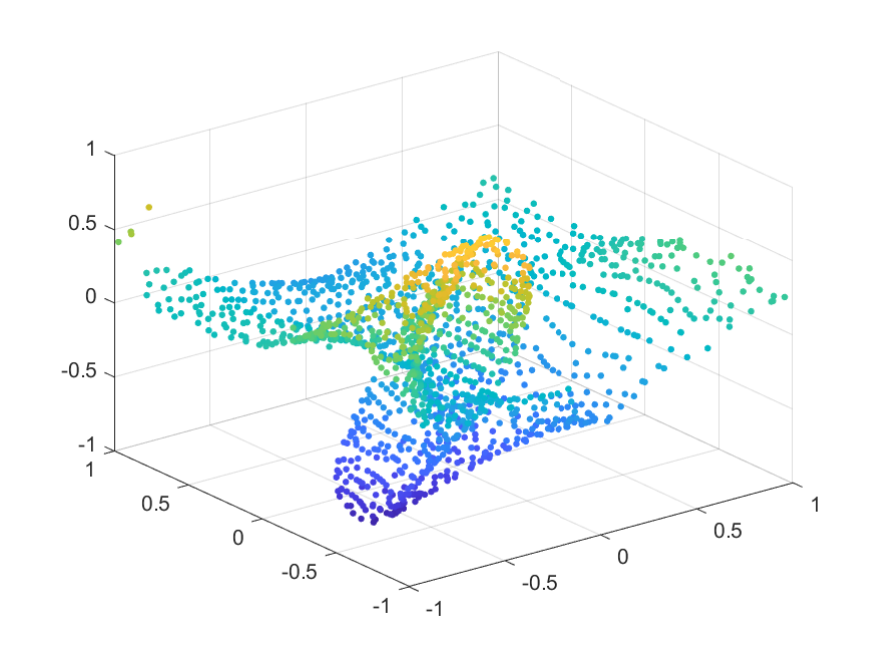}
    \includegraphics[width=0.20\textwidth]{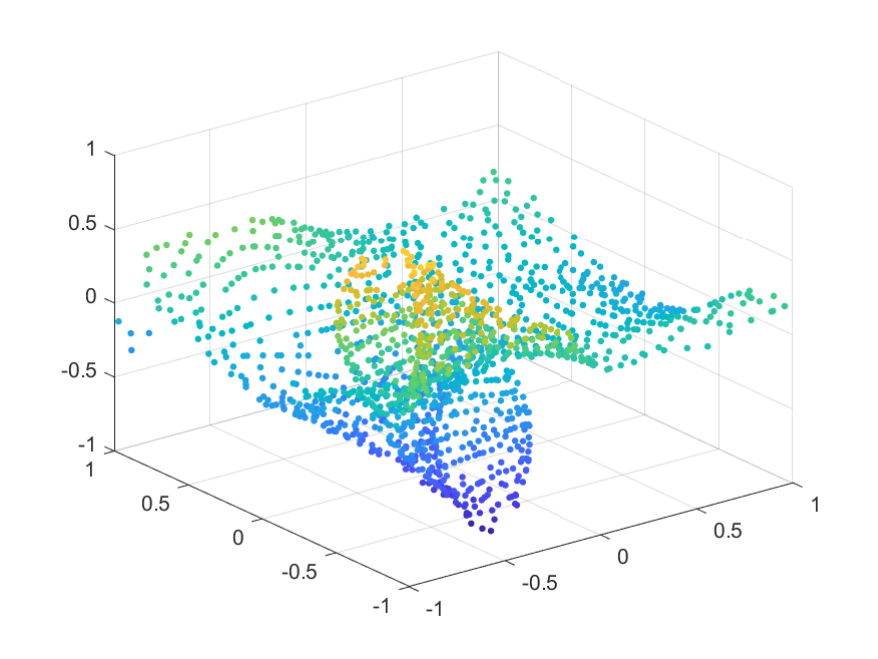}
    \caption{AG-RaNN method results for Example \ref{ex4}. The first two columns show the two components of $\bv(t,\bx)$ at $t=0$, and the last two columns show those at $t=1$. Rows (top to bottom) present the exact solution, results by normal collocation points, and results by collocation set $\Lambda_A$ ($N_A^I=(1286734,1277793)$, $N_A^B=(62291,62637)$).}
    \label{fig:ex4}
\end{figure}

In this example, the level-set functions are defined on a five-dimensional space:
$\phi_1(t,x_1,x_2,p_1,p_2)$ and $\phi_2(t,x_1,$ $x_2,p_1,p_2)$.
The intersection of the zero level sets
\begin{align*}
    \begin{cases}
        \phi_1(t,x_1,x_2,p_1,p_2)=0\\\phi_2(t,x_1,x_2,p_1,p_2)=0
    \end{cases}
\end{align*}
is a three-dimensional hypersurface in the five-dimensional space. Then we project it to $p_2$ and $p_1$, and get two hypersurfaces $p_1=v_1(t,x_1,x_2)$ and $p_2=v_2(t,x_1,x_2)$. Figure \ref{fig:ex4} only shows the surfaces when $t=0$ (the first two columns) and $t=1$ (the last two columns). This example involves a genuinely high-dimensional augmented phase space and does not rely on low-dimensional parameterization of the solution manifold.

\section{Summary}
\label{sec:summary}

This paper addresses computational challenges in recovering multivalued solution manifolds of nonlinear first-order PDEs, including Hamilton--Jacobi equations and scalar hyperbolic balance laws. Such solutions arise in geometric optics, seismic waves, semiclassical limit of quantum dynamics and high frequency limit of linear waves, and differ markedly from the viscosity solutions. The main computational challenges lie in that the solutions are no longer functions, and become union of multiple branches, after the formation of singularities. Our approach is based on a level-set formulation, which transforms the original nonlinear dynamics into a linear transport equation posed in an augmented phase space which unfolds the singularity and the evolution equations remain valid globally in time.  While this reformulation simplifies the governing equation, the increased dimensionality can become a computational bottleneck. To alleviate this difficulty, we employ randomized neural networks, whose convex least-squares training leads to robust optimization and enables efficient approximation in high dimensions.

Two ingredients are critical to the efficiency and accuracy of the proposed method. First, adaptive collocation concentrates sampling in a narrow tube around the zero level set, focusing computational effort on the region of interest. Second, a layer-wise growth strategy progressively enriches the randomized feature space, improving resolution and accuracy. For the resulting level-set equations, we establish a corresponding error analysis that supports the convergence of the RaNN approximation. Extensive numerical experiments confirm the effectiveness of the proposed method in capturing multivalued structures and discontinuities.

\bigskip
\noindent{\bf Acknowledgement}

Haoning Dang and Fei Wang acknowledge the support by the National Natural Science Foundation of China (Grant No. 92470115). 
Shi Jin acknowledges the support of the NSFC grant No. 12531016,  the Science and Technology Commission of Shanghai Municipality (STCSM) grant no. 23JC1402300, and the Fundamental Research Funds for the
Central Universities.

\end{document}